\documentclass[11pt]{article}

\usepackage{multirow}
\usepackage{comment}
\usepackage{dsfont}
\usepackage{multicol}
\usepackage{bbm}
\usepackage{enumitem}
\usepackage{array}
\usepackage{optidef}
\usepackage{tikz-network}
\usepackage{xcolor}
\usepackage{algorithm}
\usepackage{algpseudocode}
\usepackage{amssymb}
\usepackage{amsmath, amsthm}
\usepackage{graphicx}
\usepackage{url}
\usepackage{subfigure}
\usepackage[a4paper, total={6.2in, 9in}]{geometry}
\usetikzlibrary{decorations.pathreplacing}
\usetikzlibrary{arrows.meta, bending}
\usepackage[backend=biber,style=authoryear,natbib=true, maxbibnames=999, dashed = false]{biblatex}
\definecolor{awesome}{rgb}{1.0, 0.13, 0.32}
\definecolor{electricviolet}{rgb}{0.56, 0.0, 1.0}
\definecolor{electricyellow}{rgb}{1.0, 1.0, 0.0}
\definecolor{electricgreen}{rgb}{0.0, 1.0, 0.0}
\definecolor{coral}{rgb}{1.0, 0.5, 0.31}
\newcolumntype{P}[1]{>{\centering\arraybackslash}p{#1}}
\newcolumntype{R}[1]{>{\raggedleft\arraybackslash}p{#1}}
\usepackage[colorlinks=true,linkcolor=awesome,citecolor=blue,urlcolor=electricgreen]{hyperref}

\DeclareMathOperator*{\argmin}{arg\,min}
\DeclareMathOperator*{\Id}{Id}
\DeclareMathOperator{\E}{\mathbb{E}}
\DeclareMathOperator{\Tr}{\mathrm{Tr}}
\DeclareMathOperator{\Prob}{\mathbb{P}}
\DeclareMathOperator{\one}{\mathds{1}}
\DeclarePairedDelimiter{\ceil}{\lceil}{\rceil}
\DeclarePairedDelimiter{\floor}{\lfloor}{\rfloor}

\newcommand{\F}{\mathrm{F}}

\newcommand{\ope}{\mathrm{op}}
\newcommand{\llangle}{\langle \!\langle}
\newcommand{\rrangle}{\rangle \!\rangle}

\newcommand{\wip}[2]{\llangle #1, #2 \rrangle_{\mathcal{W}_2}}
\newcommand{\ipm}[3]{\mathrm{IPM}(#1, #2; #3)}
\newcommand{\lip}[2]{\| #1 \|_{\mathrm{Lip}(#2)}}
\newcommand{\sw}[2]{\ensuremath{(#1, #2)}\textup{-sub-Weibull}}

\newtheorem{theorem}{Theorem}%[section]
\newtheorem{proposition}[theorem]{Proposition}
\newtheorem{lemma}[theorem]{Lemma}%[section]

\newtheorem{definition}[theorem]{Definition}

\renewbibmacro{in:}{%
}

\newcounter{assumptioncounter}
\newenvironment{assumption}[1][]{%
  \refstepcounter{assumptioncounter}%
  \par\medskip\noindent\textbf{Assumption~\theassumptioncounter #1 \hspace{3pt}} \itshape
}{\par\medskip}

\newenvironment{proofsketch}[1][]{
    \par\medskip\noindent\textbf{Proof Sketch of Theorem #1\hspace{3pt}}\rmfamily
}
{\par\medskip}

\begin{document}
\title{Multiple-output composite quantile regression through an optimal transport lens} 

\author{Xuzhi Yang and Tengyao Wang\\
Department of Statistics, London School of Economics}
\date{February 14, 2024}

\maketitle

\begin{abstract}%
Composite quantile regression has been used to obtain robust estimators of regression coefficients in linear models with good statistical efficiency. By revealing an intrinsic link between the composite quantile regression loss function and the Wasserstein distance from the residuals to the set of quantiles, we establish a generalization of the composite quantile regression to the multiple-output settings. Theoretical convergence rates of the proposed estimator are derived both under the setting where the additive error possesses only a finite $\ell$-th moment (for $\ell > 2$) and where it exhibits a sub-Weibull tail. In doing so, we develop novel techniques for analyzing the M-estimation problem that involves Wasserstein-distance in the loss. Numerical studies confirm the practical effectiveness of our proposed procedure.
\end{abstract}

\noindent%

\begin{refsection}

\section{Introduction}\label{sec: intro}
The area of robust statistics has seen a revival of interest in recent years, %lots of references of robust mean estimation, robust covariance estimation, robust changepoints, robust differential privacy
both in Statistics and Computer Science. This is partly due to the fact that the massive surge in data volumes brings about a significant demand for efficient and precise analysis of heavy-tailed or partially corrupted data \citep{eklund2016cluster, wang2015high, szegedy2013intriguing}. Compared to earlier works in this area pioneered by \citet{tukey1963less} and \citet{huber1964,huber1965}, modern treatment of this topic focuses more on handling multivariate data. For instance, in the area of robust mean estimation, \citet{diakonikolas2020outlier,lugosi2021robust,depersin2022robust,minasyan2023statistically} have proposed various extensions of univariate robust mean procedures such as the trimmed mean estimator \citep{tukey1963less} and median of means estimator \citep{nemirovskij1983problem,jerrum1986random,alon1996space} to the multivariate setting. We witness a similar surge in research interest in the area of robust covariance estimation \citep{mendelson2020robust,abdalla2022covariance,minasyan2023statistically}. 

In this work, we focus on the topic of robust linear regression with potentially multivariate response variable, where a covariate-response pair $(X,Y)\in\mathbb{R}^p\times \mathbb{R}^d$ with joint distribution $P^{(X,Y)}$ is generated from
\begin{align}
Y = b^*X+\varepsilon, \label{mlm}
\end{align}
with regression coefficients $b^*\in\mathbb{R}^{d\times p}$, a zero-mean covariate vector $X \in \mathbb{R}^p$ and a noise vector $\varepsilon$ taking values in $\mathbb{R}^d$. Given independent and identically distributed (i.i.d.\!) covariate-response pairs $(X_1, Y_1), \ldots (X_n, Y_n)$ drawn from  $P^{(X, Y)}$, our goal is to estimate $b^*$. The contamination of the a linear model is mainly captured by two different mechanisms: heavy-tailed noise \citep{catoni2012challenging, lugosi2019mean} and outlier contamination \citep{szegedy2013intriguing, huber2004robust}. When $d=1$, both directions have thrived in recent years \citep{nguyen2013exact, fan2017estimation, sun2020adaptive, sasai2020robust, pensia2020robust, adomaityte2023high}. However, in the context of multiple-output linear regression, where $d > 1$, the literature is notably scant. In this work, we go beyond the case of the univariate response variable to the case of the multiple-output linear model under possibly heavy-tailed noise.  

One popular way to tackle the heavy-tailed error is based on the quantile regression \citep{koenker1978regression, wang2007robust, li20081, zou2008composite, wu2009variable, belloni2011}. In the case of univariate linear regression, although the ordinary least square (OLS) estimator is widely recognized as the best unbiased estimator when the random error follows a Gaussian distribution since it attains the Cramer--Rao lower bound, it may not perform well when the random error is heavy-tailed, as the mean squared error of the OLS estimator is proportional to the second moment of the random error term. This issue can be addressed by using the quantile regression estimator \citep{koenker1978regression}. Unlike the OLS estimator, which estimates the conditional mean function, the quantile regression estimator aims to estimate the conditional quantile function of $Y$ given $X$. Thanks to the robustness of quantiles, the quantile regression estimator is less affected by outliers or heavy-tailed distributions. However, the relative efficiency of the quantile regression estimator compared to the OLS estimator can be arbitrarily small based on their respective asymptotic variances. \cite{zou2008composite} proposed a solution to this issue through the composite quantile regression (CQR) method, whose loss function aggregates multiple quantile regression loss functions. Specifically, for $d=1$ and any $K \in \mathbb{N}$, the CQR estimator $\tilde b$ is obtained by the following optimization problem 
\begin{align}
    (\hat{q}_1, \ldots, \hat{q}_K, \tilde{b}) = \argmin_{q_1, \ldots, q_K \in \mathbb{R}, \ b \in \mathbb{R}^{d \times p} }\sum_{i = 1}^n \sum_{k = 1}^K \rho_{\tau_k}(Y_i - bX_i - q_k), \label{opt: cqr_est}
\end{align}
where $\rho_\tau(t)$ is the so-called check function defined as $\rho_\tau(t) = \max\{t, 0\} + (\tau - 1)t$ for any $t \in \mathbb{R}$, and $\tau_k = k/(K+1)$. \citet{zou2008composite} showed that the CQR estimator can achieve at least $70\%$ relative efficiency compared to the OLS estimator even for Gaussian noise. However, when $d \geq 2$, the CQR estimator $\tilde b$ does not have a natural extension due to the lack of a proper definition for multivariate rank/quantile and the corresponding multivariate check function. 

One of the key contributions of this study is the development of a multiple-output composite quantile regression (MCQR) estimator. The definition of our proposed estimator is closely related to the concept of the Monge--Kantorovich (MK) ranks/quantiles, which are multivariate generalization of ranks and quantiles from the view of optimal transport developed by \citet{chernozhukov2017monge} and \citet{hallin2021distribution}. Intuitively, the univariate cumulative distribution function (CDF) and the quantile function of any probability distribution $P^X$ can be viewed as optimal transport maps between $P^X$ and a reference distribution $U[0,1]$. This perspective allows for a natural extension of ranks and quantiles to multivariate distributions. Compared to many previous extensions based on Tukey's depth \citep{tukey1975mathematics}, MK-ranks/quantiles have several advantages, including the ability to capture more complex and possibly non-convex quantile contours and allowing for distribution-free inference in multivariate settings. Please refer to \citet{hallin2022measure} for a comprehensive introduction to the MK-ranks/quantiles.

A crucial observation in constructing our MCQR estimation is that the univariate CQR loss function can be equivalently described as the  \textit{Wasserstein product} between the empirical distribution of the residuals $(Y_i - bX_i: i=1,\ldots,n)$ and the uniform distribution $U[0,1]$. Here, the `Wasserstein product' between two distributions $P$ and $Q$ is the maximum of $\mathbb{E}(XY)$ over all couplings $(X,Y)$ with marginal distributions $X\sim P$ and $Y\sim Q$. When $Q$ is viewed as a reference distribution, this optimal coupling is exactly the same as in MK-quantiles. See~\eqref{eq: WassProd} for a formal definition and more detailed discussion. This alternative viewpoint allows us to circumvent the need of defining individual multivariate check functions and instead formulate the MCQR loss in terms of the MK-quantiles. It is worthwhile to note that while various previous studies in the literature have attempted to extend the concept of quantile regression to the multiple-output setting \citep{hallin2010multivariate, Kong2012-fm, hallin2015local, carlier2016vector, del2022nonparametric}, the majority have concentrated on estimating the quantile contours rather than focusing on the robust estimation of the regression coefficients. See Section \ref{sec: MCQR_method} for a more detailed discussion of our proposed method. 

Then in Section \ref{Sec: MCQR_Theory} we investigate the theoretical guarantees of the MCQR estimator. We first prove the consistency result when the random noise is only assumed to have finite $\ell$-th moment for some $\ell > 2$ (see Theomre \ref{thm: Consistency}). Then a faster convergence rate is established when we assume a noise distribution with a sub-Weibull tail (see Theorem \ref{thm: FasterRate}). We highlight that the MCQR procedure represents an M-estimation problem incorporating the Wasserstein distance within its loss function, for which the empirical process theory tools used in traditional M-estimators are not directly applicable. To the best of our knowledge, Theorem \ref{thm: Consistency} and Theorem \ref{thm: FasterRate} are the first results that establish the consistency and convergence rate of an M-estimation where the loss function involves the 2-Wasserstein distance. New theoretical tools were developed along the way, which we believe may be of independent interest in future research. Please refer to Section \ref{Sec: MCQR_Theory} for detailed descriptions of the Theorems and proof sketches. 

\subsection{Related works}\label{Sec:RelatedWorks}

Various definitions of multiple-output quantile regression have been proposed in the past, including the depth-based directional method \citep{hallin2010multivariate, Kong2012-fm, hallin2015local}, the M-quantile \citep{koltchinskii1997m}, the spatial quantile \citep{chaudhuri1996geometric, chakraborty2014spatial}, among others. As remarked above, unlike our work, all these approaches focus on estimating the quantile contours of the response variable. In addition, these definition of multivariate quantiles do not preserve the quintessential attributes of the univariate quantile, notably distribution-freeness and the Glivenko-Cantelli property \citep{hallin2021distribution}. Furthermore, their quantile contours are constrained to be convex, which hinders performance when data distribution exhibits non-convex level sets.

In contrast, \citet{chernozhukov2017monge} and \citet{hallin2021distribution} introduced a novel multivariate quantile/rank framework based on optimal transport. This framework adeptly captures level set non-convexities while retaining the distribution-freeness and the Glivenko-Cantelli property, hallmarks of the univariate rank/quantile \citep{chernozhukov2017monge, hallin2021distribution}. Several applications in multivariate statistics have been established successfully \citep{ deb2021multivariate,  del2022nonparametric, hallin2023efficient, shi2024distributiofree}. We refer to a comprehensive survey \cite{hallin2022measure} and references therein. Building upon this groundwork, \citet{carlier2016vector} and \citet{del2022nonparametric} proposed two notions of multiple-output quantile regression, though concentrating primarily on the estimation of conditional quantile functions rather than the regression coefficients themselves.

\subsection{Notation}
For $n\in\mathbb{N}$, write $[n]:=\{1,\ldots,n\}$. 
For any vector $v\in\mathbb{R}^d$, we write $\|v\| := (\sum_{j\in[d]} v_j^2)^{1/2}$. For any matrix $M\in\mathbb{R}^{p\times d}$, we define $\|M\|_{\mathrm{F}} := (\Tr(M^\top M))^{1/2}$. We denote $\mathcal{S}^{d-1}$ to be the unit sphere in $\mathbb{R}^d$. For any measurable function $f: X \to \mathbb{R}$, we denote $f^+(x):= \max\{f(x), 0\}$ as its positive part, and $f^-(x) := \max\{-f(x), 0\}$ as its negative part. We write $\mathcal{B}$ as the Borel $\sigma$-algebra of $\mathbb{R}^d$. Write $\mathcal{P}_{\ell}(\mathbb{R}^d)$ as the set of Borel probability measures defined on $(\mathbb{R}^d, \mathcal{B})$ with finite $\ell$-th order moments for $\ell \in \mathbb{N}$ and $\mathcal{P}_{ac}(\mathbb{R}^d)$ be the set of probability measures on the same space that are absolutely continuous with respect to the Lebesgue measure. For any random variable $X$ on $\mathbb{R}^d$, write $P^X$ for the associated probability measure and $P_n^X:= \frac{1}{n}\sum_{i = 1}^n \delta_{X_i}$ for the associated empirical distribution where $X_1,\ldots,X_n$ are $n$ independent copies of $X$ and $\delta_x$ denote the Dirac measure on $x$. 

\section{The MCQR construction}\label{sec: MCQR_method}
In this section, we present a generalization of the traditional CQR when the dimension of the response variable $d$ is greater than $1$. We start by revisiting the univariate CQR estimator, and showing that at the population level, it can be seen as the minimizer of the Wasserstein product between $P^{Y-bX}$ and the uniform reference distribution $U[0, 1]$, which allows a multivariate generalization. Moreover, we justify that the choice of the reference distribution does not affect the population minimizer in this problem, thus allowing us to select more natural reference distributions in multivariate settings.

\subsection{Univariate CQR revisited}
Since $q_1,\ldots,q_K$ in~\eqref{opt: cqr_est} have the interpretation of quantiles associated with $\tau_1,\ldots,\tau_K$, it is natural to further constrain the optimization by assuming $q_1\leq \cdots\leq q_K$. Let $\mathcal{M}$ denote the set of all increasing functions on $\mathbb{R}$, then \eqref{opt: cqr_est} with this additional constraint can be viewed as the empirical version of the following optimization problem %while discretizing a function $q$ at points $\tau_1, \ldots , \tau_K \in (0, 1)$ with $K \in \mathbb{N}$:
\begin{equation}
   \argmin\limits_{q \in\mathcal{M}, b \in \mathbb{R}^{1 \times p}}\mathbb{E} \Bigl\{\rho_T \bigl(Y- bX - q(T)\bigr) \Bigr\} = \argmin\limits_{q \in\mathcal{M}, b \in \mathbb{R}^{1 \times p}}\mathbb{E} \Bigl\{\int_0^1 \rho_\tau \bigl(Y- bX - q(\tau)\bigr) d \tau \Bigr\},\label{opt: cqr} 
\end{equation}
where $(X,Y)\sim P^{(X,Y)}$ and $T\sim U[0,1]$.
The following lemma indicates that, when $d = 1$, the true regression coefficient $b^*$ in~\eqref{mlm} and the quantile function $q^*_{\varepsilon}: \tau \mapsto \inf\{y\in\mathbb{R}: P^{\varepsilon}(-\infty, y] \geq \tau\}$ of $\varepsilon$ form a solution of \eqref{opt: cqr}. As we will see from Lemma~\ref{le: 1dmcqr} and Proposition~\ref{prop: unique}, this is actually the unique solution to the problem.
\begin{lemma}\label{le: justification}
Under the linear model \eqref{mlm}, we have  
\[
        (b^*, q_{\varepsilon}^*) \in \argmin_{b \in \mathbb{R}^{1 \times p}, q \in \mathcal{M}} \E\int_{0}^1 \rho_\tau (Y - bX - q(\tau)) \,d\tau .
\]
\end{lemma}

In fact, an inspection of the proof (see Appendix \ref{apx: LemmaJustification}) of the above lemma reveals that if $\tau_1,\ldots,\tau_K$ converges to a distribution $P^Z$ with support $\mathcal{Z}$ rather than to $U[0,1]$, then a similar result to Lemma~\ref{le: justification} holds provided that we modify the convex check functions $\rho_\tau:\mathbb{R}\to\mathbb{R}^+$ for $\tau\in\mathcal{Z}$ so that they satisfy $F_W^{-1} \circ F_Z(\tau)\in \argmin_{\theta}\E \rho_{\tau}(W- \theta)$ for all random variables $W$ with absolutely continuous distributions. However, generalizing the check functions beyond the univariate setting is difficult. While some attempts have been made \citep{chaudhuri1996geometric, koltchinskii1997m}, the resulting multivariate quantiles, defined through the minimizer of these generalized check functions, lack key properties of their univariate counterparts (see our discussion in Section~\ref{Sec:RelatedWorks}, as well as empirical comparisons in Section~\ref{sec: imp}). Instead, our work takes a different approach and generalizes the CQR population loss function as a whole rather than individual check functions. A key observation that allows us to achieve this is the following reformulation of the loss function of~\eqref{opt: cqr} in Lemma~\ref{le: 1dmcqr} below. 
To state the lemma,  we define the \emph{Wasserstein product} between $P, Q \in\mathcal{P}_2(\mathbb{R}^d)$ as
\begin{align}
     \wip{P}{Q} := \sup_{\gamma \in \mathcal{C}(P, Q)} \int \langle x , y \rangle d\gamma(x, y),  \label{eq: WassProd} 
\end{align}
where $\mathcal{C}(P, Q)$ denotes the set of all couplings between $P$ and $Q$, i.e. \!for any $\gamma \in \mathcal{C}(P, Q)$, and measureable subsets $A$, $B \subset \mathbb{R}^d$, we have $\gamma(A \times \mathbb{R}^d) = P(A)$ and $\gamma(\mathbb{R}^d \times B) = Q(B)$. The name `Wasserstein product' stems from its intrinsic link with the 2-Wasserstein distance: $\frac{1}{2}\mathcal{W}_2^2(P, Q) = \frac{1}{2}\int \|x\|^2 \, dP(x) + \frac{1}{2}\int \|y\|^2 dQ(y) - \wip{P}{Q}$. We will often slightly abuse notation to write $\wip{X}{Y}$ instead of  $\wip{P^X}{P^Y}$.

\begin{lemma}\label{le: 1dmcqr}
Suppose that $X\sim P^X$ is mean-zero with finite second moments. For $U\sim U[0, 1]$, and a fixed $b\in\mathbb{R}^{1\times p}$, we have 
\[
\inf_{q\in\mathcal{M}} \mathbb{E} \Bigl\{\int_0^1 \rho_\tau \bigl(Y- bX - q(\tau)\bigr) d \tau \Bigr\} + \frac{1}{2}\E Y = \wip{Y - bX}{U}.
\]
\end{lemma}
The proof is deferred to Appendix \ref{apx: Lemma1dmcqr}. Writing $\mathcal{L}(b; U) := \wip{Y-bX}{U}$, Lemma~\ref{le: 1dmcqr} and Equation~\eqref{opt: cqr} imply that, the optimizer in $b$ for the population CQR loss function in~\eqref{opt: cqr} is equal to  $\argmin_{b \in \mathbb{R}^{d \times p}} \mathcal{L}(b; U)$ when $d=1$.

\subsection{Multiple-output CQR via optimal transport}\label{sec: MCQR via OT}

With the help of Lemma~\ref{le: 1dmcqr}, we may regard $\mathcal{L}(b; U)$ as a generalized population CQR loss function for the multiple-output case ($d \geq 2$) for suitably chosen reference random vector $U$. The following proposition (see Appendix \ref{apx: PropUnique} for proof) verifies that under a mild condition this loss has a unique minimizer and that is independent of the specific choice of $U$ (see Appendix \ref{apx: ref_distribution} for an intuitive illustration).  

\begin{proposition}\label{prop: unique}
If $P^\varepsilon, P^U \in  \mathcal{P}_2(\mathbb{R}^d) \cap \mathcal{P}_{ac}(\mathbb{R}^d)$ and  $P^X$ is not a point mass, then $b^*$ is the unique minimizer of $\mathcal{L}(b; U)$.
\end{proposition}

There are various choices of the reference distribution of $U$, including the uniform distribution on the unit cube  \citep{chernozhukov2017monge, deb2021multivariate} and the spherical uniform distribution \citep{hallin2021distribution, del2022nonparametric}. In this paper, we opt for the standard multivariate normal distribution as the reference distribution, primarily motivated by its advantageous theoretical characteristics.  Moreover, we will also omit the specification of the reference distribution in the loss function and simply write it as $\mathcal{L}(b)$ throughout the rest of the paper.

Proposition~\ref{prop: unique} motivates the following natural estimator of $b^*$ based on the Wasserstein product of the empirical distributions.
\begin{definition}
Given i.i.d.\ covariate-response pairs $(X_1,Y_1),\ldots,(X_n,Y_n)$ generated as in \eqref{mlm} and  a reference distribution $P^U \in  \mathcal{P}_2(\mathbb{R}^d) \cap \mathcal{P}_{ac}(\mathbb{R}^d)$ and $U_1, \ldots, U_m \stackrel{i.i.d.}{\sim} P^U$, the \emph{MCQR estimator} for $b^*$ is defined as
\begin{align}
\hat{b} \in \argmin_{b \in \mathbb{R}^{d\times p}}\mathcal{L}_{n, m}(b), \ \text{where $\mathcal{L}_{n, m}(b):= \wip{P_n^{Y-bX}}{P_m^U}$} \label{opt: mcqr_est}.
\end{align}
\end{definition}
The optimization procedure above is an M-estimation problem. However, unlike classical M-estimation problems, the empirical loss function cannot be viewed as an empirical process of the population loss (in fact, $\mathbb{E}\wip{P_n^{Y-bX}}{P_m^U} \neq \wip{P^{Y-bX}}{P^U}$), which prevents us from applying traditional empirical process theory techniques to obtain the convergence rate results directly. Instead, a collection of new theoretical results is developed to better understand both the population and empirical version of the Wasserstein product loss. Please refer to Section \ref{Sec: MCQR_Theory} for more details. Secondly, it is worth noting that the empirical reference distribution $P_m^U$ is distinct from the distribution of $\tau_k$'s in~\eqref{opt: cqr_est} when $d=1$. Instead, we employ it as the reference distribution to redefine the distribution function and the quantile function (refer to Appendix \ref{apx: ref_distribution} for an example). Thus, even when $d=1$ with a uniform reference distribution, the plug-in estimator in~\eqref{opt: mcqr_est} does not reduce to the univariate CQR estimator~\eqref{opt: cqr_est}. This can also be seen from the proof of Lemma \ref{le: 1dmcqr}. Therefore, our proposed MCQR estimator~\eqref{opt: mcqr_est} is different from the univariate CQR estimator that is studied in \cite{zou2008composite} but shares the same loss function at the population level. See also Figure~\ref{fig: ParetoContam} and Figure~\ref{fig: GaussianContam} for an interesting difference in their robustness to contamination in one dimension.

\subsection{Solving MCQR via linear programming}
\label{Sec:LP}
We describe here how the optimization problem can be solved in practice. Given  $\{(X_i, Y_i)\}_{i = 1}^n \subset \mathbb{R}^{p} \times \mathbb{R}^{d}$ and $\{U_i\}_{i = 1}^m$, we define ${X} = (X_1, \ldots, X_n)^\top \in \mathbb{R}^{n \times p}$ and $Y = (Y_1, \ldots, Y_n)^\top \in \mathbb{R}^{n \times d}$ and  $U = (U_1, \ldots, U_m)^\top \in \mathbb{R}^{m \times d}$. Define 
\[
\mathcal{C}_{n,m} =\{A\in\mathbb{R}_+^{m\times n}: A\mathbf{1}_n = \mathbf{1}_m / m \text{ and } A^\top \mathbf{1}_m = \mathbf{1}_n / n \}.
\]
Every $\pi\in \mathcal{C}_{n,m}$ represents a coupling of $P^{(X,Y)}_n$ and $P^{U}_m$ in the sense that $\pi_{i,j}$ denotes the mass to be transported from $(X_i,Y_i)$ to $U_j$. Then by the definition of $\wip{\cdot}{\cdot}$, the optimization problem in~\eqref{opt: mcqr_est} can be written as
\begin{align*}
    \min_{b \in \mathbb{R}^{d \times p}}\max\limits_{\pi \in \mathcal{C}_{n, m}} \Tr\bigl( U^\top \pi(Y-Xb^\top)\bigr) &= \max\limits_{\pi \in \mathcal{C}_{n, m}} \min_{b \in \mathbb{R}^{d \times p}}\Tr\bigl( U^\top \pi(Y-Xb^\top)\bigr) \\ 
    & = \max_{\pi \in \mathcal{C}_{n, m}} \min\limits_{b\in \mathbb{R}^{d \times p}} \bigl\{ \Tr( U^\top \pi Y ) - \Tr(U^\top \pi Xb^\top)\bigr\},
\end{align*}
where the exchange of the minimum and maximum is allowed as the objective is linear \citep{v1928theorie}. The dual formulation on the right-hand side is easier to handle since its inner minimum is equal to $-\infty$ unless $U^\top \pi X = 0$. Hence, the dual problem of~\eqref{opt: mcqr_est} is
\begin{align*}
    \max_{\pi \in \mathcal{C}_{n,m}} & \quad \Tr( U^\top \pi Y )\\
    \mathrm{s.t.} & \quad U^\top \pi X = 0,
\end{align*}
which can be solved by standard linear programming solvers. After obtaining the dual optimizer $\hat \pi$, the MCQR estimator $\hat{b}$ is obtained via complementary slackness.

\section{Theoretical guarantees}\label{Sec: MCQR_Theory}
In this section, we investigate the theoretical performance of the proposed estimator when adopting a standard Gaussian reference distribution $U\sim \mathcal{N}(0,I_d)$. In Theorem \ref{thm: Consistency}, we provide a non-asymptotic bound for the estimation error when only assuming a finite $2+\delta$ moment condition on the random noise term. Furthermore, we demonstrate in Theorem \ref{thm: FasterRate} that in cases where the distributions of both the covariates and the noise exhibit a sub-Weibull tail, the MCQR estimator enjoys a faster rate of convergence to the truth. 

Given a positive definite matrix $\Sigma\in\mathbb{R}^{p\times p}$ and any matrix $A \in \mathbb{R}^{d \times p}$, we define the \emph{matrix Mahalanobis norm} of $A$ with respect to $\Sigma$ as $\|A\|_{\Sigma}:= \Tr^{1/2}(A \Sigma A^\top) = \|A\Sigma^{1/2}\|_{\mathrm{F}}$. We will assume throughout this section that $\E(XX^\top) = \Sigma$. %Writing $S:=\Sigma^{-1/2}X$, we assume that $S$ satisfies the following distributional assumption.
\begin{assumption}\label{assumption: ellipcitaldis}
$X$ follows an elliptical distribution, i.e., there exists independent random variable $R$ on $\mathbb{R}_+$ and random vector $Q \sim U(\mathcal{S}^{d-1})$ such that $X = \Sigma^{1/2}QR$.
\end{assumption}

Under this assumption on $X$, we first consider the case when the random noise $\varepsilon$ is only assumed to satisfy a finite moment condition.

\begin{theorem}\label{thm: Consistency}
Suppose $(X,Y), (X_1,Y_1),\ldots,(X_n,Y_n)$ are i.i.d.\ pairs generated according to~\eqref{mlm}, $U_1,\ldots,U_m\stackrel{\mathrm{i.i.d.}}{\sim} 
\mathcal{N}(0,I_d)$. Assume $m\geq n > 1$ and that  Assumption \ref{assumption: ellipcitaldis} holds. If $P^X, P^{\varepsilon} \in \mathcal{P}_\ell(\mathbb{R}^d)$ for $\ell > 2$ 
%and satisfying $\ell \not= 4$ when $d \leq 4$ and $\ell \not= \frac{d}{d-2}$ otherwise. Let $(X_1,Y_1),\ldots,(X_n,Y_n)$ be independent copies of $(X,Y)$, and $\hat{b}$ be 
then there exists $C > 0$ depending only on $\ell, d$ and $p$ such that with probability at least $1 - 4 (\log n)^{-1}$, the MCQR estimator defined in~\eqref{opt: mcqr_est} satisfies
\[
\| \hat{b} - b^*\|_{\Sigma}^2 \wedge 1 \leq C \bigl(n^{-\frac{1}{4}} + n^{-\frac{1}{ d \vee p}} + n^{\frac{2-\ell}{2 \ell}} \bigr)\log m.
\]
\end{theorem}
An immediate consequence of Theorem \ref{thm: Consistency} is that if taking $n$ and $m$ to be large enough such that 
\begin{align}
     C \bigl(n^{-\frac{1}{4}} + n^{-\frac{1}{ d \vee p}} + n^{\frac{2 -\ell}{2 \ell}} \bigr)\log m < 1, \label{Condition: nmSufficientlyLarge}
\end{align}
then we have 
\begin{align}
    \|\hat b - b^*\|_{\Sigma}^2 \leq  C\bigl(n^{-\frac{1}{4}} + n^{-\frac{1}{ d \vee p}} + n^{\frac{2 -\ell}{2 \ell}} \bigr)\log m \label{ErrorBound1} 
\end{align}
holds with probability at least $1-4(\log n)^{-1}$. We make a few remarks here. Firstly, to the best of our knowledge, this is the first consistency result for an M-estimator whose loss function involves a multivariate 2-Wasserstein distance term. \citet{bernton2019parameter} studied the convergence rate and asymptotic distribution of a minimum Wasserstein estimator, but their result is restricted to 1-Wasserstein distance in the univariate setting, for which explicit characterization of the optimal transport is available.  %The lack of explicit form for the derivative of $\wip{Y-bX}{U}$ to $b$ lies in the main challenge of the proof. 
In our setting, the traditional M-estimator/Z-estimator argument \citep[Chapter 3.2-3.3]{vanderVaartWellner1996} that derives consistency and rate of convergence of an M-estimator by analyzing the curvature of the loss function is infeasible. Instead, our proof relies on several new lemmas that reveal important properties of the Wasserstein product. %The second implication of bound (\ref{ErrorBound1}) is that $\hat b$ will lie in a unit ball centered at $b^*$ in the $\|\cdot\|_{\Sigma}$-norm on a high-probability event, which is the starting point of the proof of Theorem \ref{thm: FasterRate}. It enables us to establish a \textit{uniform} empirical Wasserstein distance convergence rate which plays a key role in the whole argument of Theorem \ref{thm: FasterRate}.

 % Although \citet{wang2021two, pmlr-v151-wang22f} developed a convergence rate that only depends on the dimension of the transformed space, they require a compact support assumption on both the target distribution and the source distribution while we do not.

To briefly sketch the proof of Theorem \ref{thm: Consistency}, we first introduce the following lemmas.

\begin{lemma}\label{le: wipLowerBound}
Let $Z$ and $\varepsilon$ be independent random vectors in $\mathbb{R}^d$ and $U\sim \mathcal{N}(0,I_d)$. If $P^{\varepsilon}$ and $P^Z$ are atomless probability measures with finite-second moments, then
\[
        \wip{Z+\varepsilon}{U}^2 \geq \wip{Z}{U}^2  + \wip{ \varepsilon}{U}^2. 
\]
\end{lemma}

This lemma is proved by constructing a sequence of couplings of the triple $(Z, \varepsilon, U)$ via the Slepian smart path interpolation \citep[see e.g.][Chapter 7.2.1]{vershynin2018high}. The best induced coupling of $(Z+ \varepsilon, U)$ provides the desired lower bound of $\wip{Z+\varepsilon}{U}$. See Appendix \ref{apx: wiplemma} for the proof. We remark that the lower bound in Lemma~\ref{le: wipLowerBound}
    is sharp, as can be seen from Lemma~\ref{le: wipLowerBoundSharp} in Appendix \ref{apx: Auxillary}. %\TY{write a lemma to say exists $Z$ and $\varepsilon$ such that equality holds. This can be done by setting $Z \sim N(0, \Sigma)$ and $\epsilon\sim N(0,\Gamma)$ such that $\Sigma = \mathrm{diag}(\sigma, 0,\ldots,0)$ and $\Gamma = \mathrm{diag}(\gamma, 0,\ldots,0)$. Then LHS = $\sigma$+$\gamma$ = RHS.}\XZ{Baically, we need to find a pair of $(\varepsilon, Z)$, such that $ \Bigl(\frac{\E\|Z\|^2}{2} + \frac{\E\|U\|^2}{2} - \frac{1}{2}\mathcal{W}_2^2(P^Z, P^\varepsilon)\Bigr)^2 + \Bigl(\frac{\E\|\varepsilon\|^2}{2} + \frac{\E\|U\|^2}{2} - \frac{1}{2}\mathcal{W}_2^2(P^U, P^\varepsilon)\Bigr)^2 \geq \Bigl(\frac{\E\|Z + \varepsilon\|^2}{2} + \frac{\E\|U\|^2}{2} - \frac{1}{2}\mathcal{W}_2^2(P^{Z+\varepsilon} , P^U)\Bigr)^2 $. Given the explicit form of 2-Wasserstein distance between Gaussian distribution, }

\begin{lemma}\label{le: innerprodWassControl}
    Let $X_1$, $X_2$, $Y_1$, $Y_2$ be random elements taking values in a common normed space $\mathcal{X}$. We have
    \[
        \bigl|\wip{X_1}{X_2} - \wip{Y_1}{Y_2}\bigr| \leq   \bigl(\E\|Y_2\|^2\bigr)^{1/2} \mathcal{W}_2(P^{X_1}, P^{Y_1})  + \bigl(\E\|X_1\|^2\bigr)^{1/2}\mathcal{W}_2(P^{X_2}, P^{Y_2}) .
    \]
\end{lemma}

This lemma links $\mathcal{W}_2(P^{X_1}, P^{X_2})$, $\mathcal{W}_2(P^{Y_1}, P^{Y_2})$ with $\mathcal{W}_2(P^{X_1}, P^{Y_1})$, $\mathcal{W}_2(P^{X_2}, P^{Y_2})$. This is useful when transforming a two-sample problem into two one-sample problems. Please refer to Appendix \ref{apx: innerprodWass} for the proof.

\begin{figure}
  \centering
  \subfigure[The upper and lower bound constructed in the Proof of Theorem \ref{thm: Consistency}.]{
  \includegraphics[scale = 0.3]{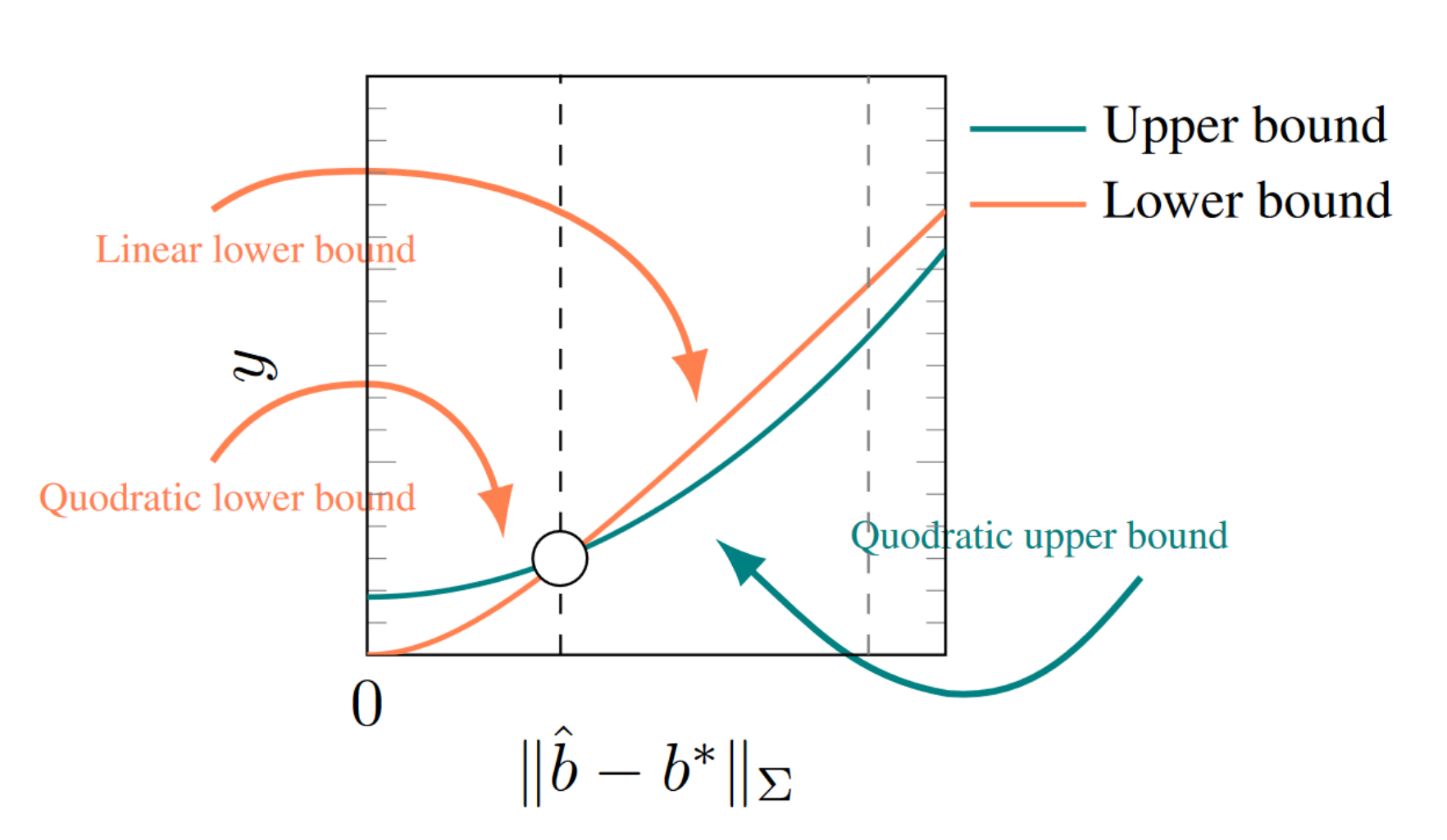}
  \label{fig: proofthm5}
  }\hspace{0.3cm}
    \subfigure[The upper and lower bound constructed in the proof of Theorem \ref{thm: FasterRate}.]{
    \includegraphics[scale = 0.3]{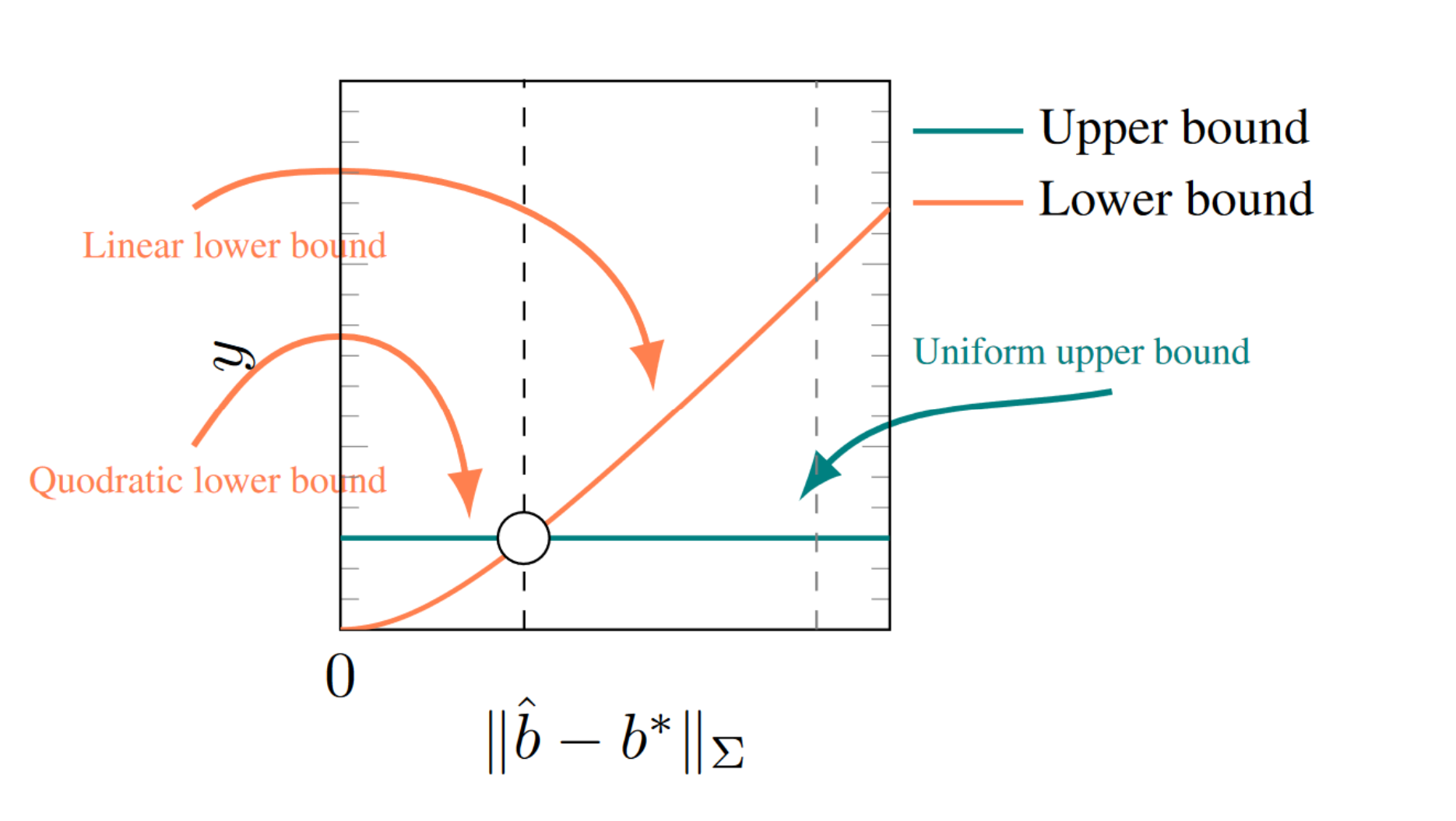}
  \label{fig: proofthm8}
    }
    \caption{Illustration of proofs.}
\end{figure}

\begin{proofsketch}[\ref{thm: Consistency}]
    We start with the basic inequality:
\begin{align}
    \mathcal{L}(\hat{{b}}) - \mathcal{L}({{b}^*}) \leq \mathcal{L}(\hat{{b}}) - \mathcal{L}_{n,m}(\hat{{b}}) + \mathcal{L}_{n,m}({{b}^*}) - \mathcal{L}({{b}^*})  \label{ineq: basic}.
\end{align}
The proof strategy involves establishing a lower bound for the left-hand side of (\ref{ineq: basic}) with respect to $\|\hat b - b^*\|_{\Sigma}$ and an upper bound for the right-hand side of (\ref{ineq: basic}) in terms of $\|\hat b - b^*\|_{\Sigma}$. Then by solving the resulting inequality, we can derive an expression bounding $\|\hat b - b^*\|_{\Sigma}$. 

For a lower bound of the left-hand side of (\ref{ineq: basic}), since for any $b \in \mathbb{R}^{d\times p}$, we have $\mathcal{L}({b}) - \mathcal{L}({{b}^*}) = \wip{{(b^*-b)X + \varepsilon}}{{U}} -\wip{{\varepsilon}}{{U}}$, by applying Lemma \ref{le: wipLowerBound} and the explicit form for $\wip{(b^* - b)X}{U}$ we can show that 
\begin{align}
    \mathcal{L}(b) - \mathcal{L}(b^*) \geq \sqrt{r^2 + \|b^* -b\|_{\Sigma}^2} -r, \label{ineq: Lowerbound} 
\end{align}
where $r := \wip{\varepsilon}{U}$ is a constant. This lower bound grows quadratically in $\|\hat b - b^*\|_{\Sigma}$ when $\|\hat b - b^*\|_{\Sigma}$ is close to zero and linearly when $\|\hat b - b^*\|_{\Sigma}$ is large (see Figure \ref{fig: proofthm5} for an illustration).

To upper bound the right-hand side of (\ref{ineq: basic}), by applying Lemma \ref{le: innerprodWassControl} we have for each $b \in \mathbb{R}^{d \times p}$,
\begin{align}
    |\mathcal{L}(b) - \mathcal{L}_{n,m}(b)|\leq \biggl(\frac{1}{m}\sum_{i=1}^m&\|U_i\|^2\biggr)^{1/2} \mathcal{W}_2(P^{Y -bX}, P_n^{Y -bX}) \notag \\ 
    &+ (\E\|Y-bX\|^2)^{1/2}\mathcal{W}_2(P^{U}, P_m^{U})  \label{ineq: Upperbound}.
\end{align}
Here $\mathcal{W}_2(P^{Y -bX}, P_n^{Y -bX})$ and $\mathcal{W}_2(P^{U}, P_m^{U})$ are one-sample empirical Wasserstein distance, and the state-of-art convergence rate can be applied \citep[see e.g.][]{fournier2015rate} (the actual proof is more involved in the sense that we need to establish the same result uniformly over $b$). Then a direct calculation on the right-hand side of (\ref{ineq: Upperbound}) leads to a quadratic upper bound in terms of $\|b^* - b\|_{\Sigma}$. The result follows by combining the upper bound with the lower bound (\ref{ineq: Lowerbound}). See Appendix \ref{apx: ConsisThm} for a complete proof.  
\end{proofsketch}

Before we state a faster convergence rate result, we first introduce the following assumptions.
\begin{assumption}\label{assumption: sw}
    For some $\sigma_1, \sigma_2 >0$ and $\alpha, \beta \in (0, 2]$, it holds that the distribution of $\Sigma^{-1/2}X$ is  $\sw{\sigma_1}{\alpha}$ and $P^\varepsilon$ is $\sw{\sigma_2}{\beta}$, in the sense that  
    \begin{align}
    \E \exp\biggl\{{\frac{1}{2}(\|\Sigma^{-1/2}X\|/\sigma_1})^{\alpha}\biggr\} \leq 2 \quad \text{and} \quad \E \exp\biggl\{{\frac{1}{2}(\|\varepsilon\|/\sigma_2})^{\beta}\biggr\} \leq 2 \label{def: sub-weibull}
\end{align}
\end{assumption}
\begin{assumption}\label{assumption: anticoncen}
    Suppose $P^\varepsilon \in \mathcal{P}_{ac}(\mathbb{R}^d)$. For some $\gamma_1, \gamma_2>0$, the density function of $\varepsilon$, write as $f_{\varepsilon}$, satisfies the following anti-concentration property
    \begin{align}
        f_{\varepsilon}(e) \geq \gamma_1 \exp{(-\gamma_2\|e\|^2)}, \quad \text{for $\|e\| \geq 1$} \label{ineq: AntiConcentrationofvarepsilon}.
    \end{align}
\end{assumption}
On the one hand, Assumption \ref{assumption: anticoncen} immediately implies the following anti-concentration bound
\begin{align}
    \Prob(\|\varepsilon\| \geq r) \geq \frac{\pi^{d/2}\bigl((r + 1)^d - r^d\bigr)}{\Gamma(\frac{d}{2} + 1)} \gamma_1 \exp(-2\gamma_2 r^2 - 2\gamma_2), \quad \text{for $r \geq 1$}\notag.
\end{align}
This indicates that the random noise $\varepsilon$ possesses a heavier tail than the sub-gaussian tail outside the unit ball. On the other hand, by proposition \ref{prop: swequivalent}\ref{prop: sw1}, the sub-Weibull assumption implies that $\Prob(\|\varepsilon\|\geq r) \leq 2 e^{-\frac{1}{2}(r/\sigma_2 )^{\beta}}$. The anti-concentration condition in (\ref{ineq: AntiConcentrationofvarepsilon}) is a relaxation of the so-called $(\gamma_1, \gamma_2)$-regularity defined in \cite{polyanskiy2016wasserstein}. The merit of employing this relaxation becomes apparent when examining Lemma \ref{le: AnticoncentrationforSumofTwoRVs}, where it is demonstrated that the convolution of two independent probability densities adhering to (\ref{ineq: AntiConcentrationofvarepsilon}) continues to satisfy the anti-concentration inequality. In contrast, the convolution of two independent regular densities may not be regular. 

Equipped with these assumptions, we are ready to state an improved convergence rate. 
\begin{theorem}\label{thm: FasterRate}
    Under the same setup of Theorem \ref{thm: Consistency} and suppose that Assumptions~\ref{assumption: sw} and~\ref{assumption: anticoncen} are satisfied. For $m,n$ large enough such that (\ref{Condition: nmSufficientlyLarge}) is satisfied, there exists some constant $M>0$ depending only on $d, \alpha, \beta, \sigma_1, \sigma_2, \gamma_1, \gamma_2$ such that with probability at least $1 - 33(\log n)^{-1}$, we have
    \begin{equation}
        \|b^* - \hat b\|_{\Sigma}^2 \leq M  \bigl((p/n)^{1/2}+n^{-2/d}\bigr) (\log m)^{\frac{8}{2\wedge \alpha\wedge\beta}}. \label{ErrorBound2} 
    \end{equation}
\end{theorem}

When $d >4$, up to a factor of the logarithm, the empirical Wasserstein distance estimation error $n^{-2/d}$ is the dominant term. This is derived from a uniform empirical Wasserstein distance control (see~\eqref{ineq: basicinequality_uniform} and Proposition \ref{Prop:W2distdiff}), and its minimax optimality has been established in \cite{singh2018minimax}. Compared to (\ref{ErrorBound1}), this improved bound in~\eqref{ErrorBound2} removes the dependence on $p$ in the exponent. Moreover, unlike the convergence rate result established for the projected Wasserstein distance in \citet{wang2021two, pmlr-v151-wang22f}, our argument does not require the distribution of $\varepsilon$ to have compact support. When $d \leq 4$, the parametric rate $(p/n)^{1/2}$ dominates the estimation error. However, this does not translate into a the root-$n$ consistency even when $d = 1$. We conjecture that this is likely due to an artifact of our proof. Specifically, due to a lack of effective tools to analyze the curvation of the loss function that incorporates the Wasserstein distance, we were unable to obtain concentration results for $\frac{\partial}{\partial b} (\mathcal{L}(b)- \mathcal{L}_{n,m}(b))$ uniformly over $b$ in a similar way that we have done for $\mathcal{L}(b)- \mathcal{L}_{n,m}(b)$. Exploration along this direction remains an area for future work. We briefly sketch the proof below. See Appendix \ref{apx: FasterRate} for a complete proof.

% When $d \leq 4$, the parametric rate $(p/n)^{1/4}$ dominates estimation error, however, it does not recover the optimal root-$n$ convergence rate obtained in \cite{zou2008composite} even when $d = p = 1$. We make two remarks on this gap: 1) as mentioned in section \ref{sec: MCQR via OT}, our composite CQR estimator (\ref{opt: mcqr_est}) is different from the univariate CQR estimator defined in \citet{zou2008composite} even in the case of $d = 1$. Thus it is possible for them to have different convergence rates; 2) due to the lack of explicit form for the derivative of $\mathcal{W}_2^2(P^{Y-bX}, P^U)$ to $b$, the explicit form for the derivative of $\wip{Y-bX}{U}$ to $b$ is also unavailable. Hence, obtaining the asymptotic normality of an M-estimator through the Z-estimator argument \citep[Chapter 3.3]{vanderVaartWellner1996} by analyzing the curvature of the loss function is infeasible. Instead, we continue with the proof idea of Theorem \ref{thm: Consistency} but provide a stronger control on the right-hand side of (\ref{ineq: basic}). \XZ{I think we should say more about the optimality of $n^{-2/d}$, and remove the discussion on the sub-optimality. }

\begin{proofsketch}[\ref{thm: FasterRate}]
Assume the setting of Theomem \ref{thm: Consistency}, error bound (\ref{ErrorBound1}) implies that on a high probability event, $\hat b$ will lie in a bounded ball centered at $b^*$, denoted by $\mathcal{B}$. Thus the basic inequality (\ref{ineq: basic}) indicates the following uniform bound 
\begin{align}
    \mathcal{L}(\hat b) - \mathcal{L}(b^*) &\leq 2 \sup_{b \in \mathcal{B}}|\mathcal{L}(b) - \mathcal{L}_{n,m}(b)| \notag \\ 
    &\leq \Bigl|\frac{1}{m}\sum_{i = 1}^m \|U_i\|^2 - \E\|U\|^2 \Bigr| + \sup_{b \in \mathcal{B}}\Bigl|\frac{1}{n}\sum_{i = 1}^n \|Y_i -bX_i\|^2 - \E \|Y-bX\|^2 \Bigr| \notag \\
    &\quad+ \sup_{b \in \mathcal{B}}\Bigl|\mathcal{W}_2^2(P^{Y-bX}, P^U) - \mathcal{W}_2^2(P_n^{Y-bX}, P_m^U) \Bigr| \label{ineq: basicinequality_uniform}.
\end{align}
Utilizing the same lower bound for the left-hand side as in (\ref{ineq: Lowerbound}), it remains to derive an upper bound for the right-hand side of the above inequality. While the initial two terms of (\ref{ineq: basicinequality_uniform}) can be effectively controlled through the application of statistical concentration arguments, as elucidated in Lemma \ref{le: ULLN}, achieving control over the last term demands much more effort. Motivated by the duality argument presented in \citet[Theorem 12]{manole2021sharp}, we establish a non-asymptotic \textit{uniform} error bound for the empirical 2-Wasserstein distance (Proposition \ref{Prop:W2distdiff} in Appendix \ref{apx: FasterRate}; see also Figure \ref{fig: proofthm8} for an illustration), which forms the key ingredient of the proof.
\end{proofsketch}

\section{Numerical experiments}\label{sec: imp}
% We can do two sets of experiments, the first showcasing how our procedure works under heavy-tailed noise or noise of arbitrary correlation/support. We do a 4-panel figure: a) Gaussian b) mvt2 c) pareto copula d) banana-shaped error
% the second experiment looks at how robust the procedure is to contamination: a) N(0, I), N(1000, I), increasing contamination fraction epsilon from 0 to 0.5, p = 7 and d = 1; b) same as a but p = 10 and d = 4; c) pareto contaminated with shifted and scaled pareto p = 7, d = 1; d) same but p = 10 d = 4. 

In this section, we compare the empirical performance of MCQR  with other robust regression estimators. The MCQR estimator is obtained by solving the linear programming problem in Section~\ref{Sec:LP}. The competitors used in the simulation studies include the ordinary least squares estimator (LS), the spatial quantile regression (SpQR) with zero quantile level \citep{chaudhuri1996geometric}, and coordinate-wise CQR (CoorCQR), i.e. independently applying CQR to each component of the response variable. We refer readers to Appendix \ref{Apx: SimulationDetails} for more details about SpQR. 

\begin{figure}[htbp!]
\centering
\subfigure[Gaussian noise]{\includegraphics[width=0.45\textwidth]{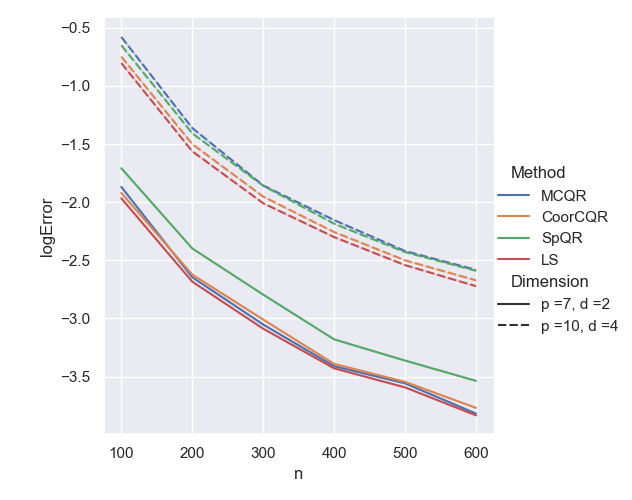}\label{fig: Gauss}}
\subfigure[multivariate $t_2$ noise]{\includegraphics[width=0.45\textwidth]{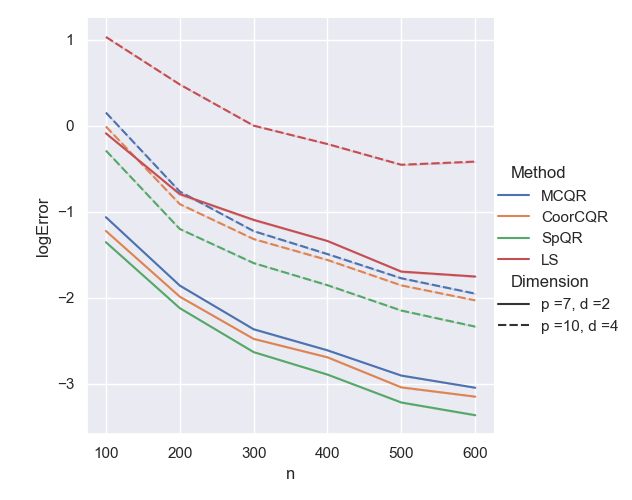}\label{fig: t_2}}
\subfigure[Pareto copula noise]{\includegraphics[width=0.45\textwidth]{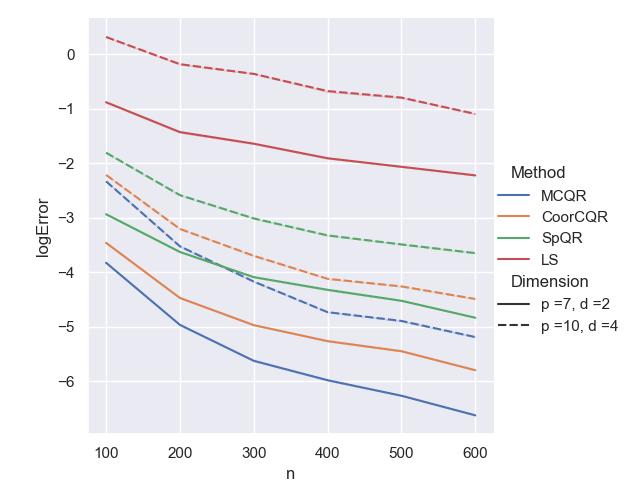}\label{fig: Pareto}}
\subfigure[Banana-shaped noise]{\includegraphics[width=0.45\textwidth]{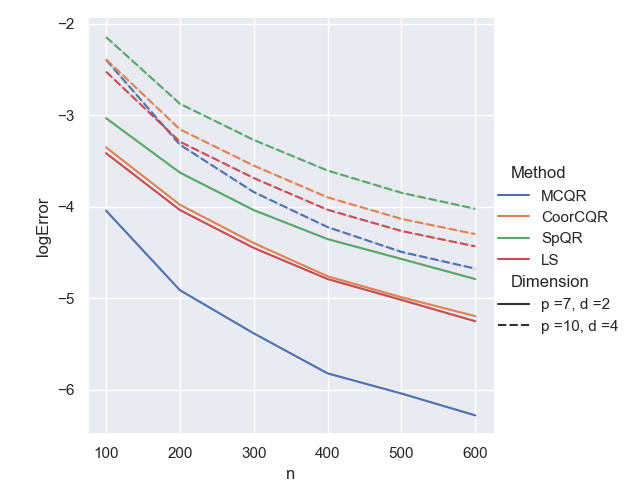}\label{fig: Banana}}

\caption{Logarithmic average loss, measured in matrix Mahalanobis norm, of the regression coefficient estimated by MCQR, CoorCQR, SpQR and LS for data generated according to the mechanism described in Section~\ref{sec: imp} for various sample size $n$, covariate dimension $p$ and response dimension $d$ and four different noise distributions (panels (a) to (d)).}

\label{fig: Experiment}
\end{figure}

% \subfigure[Gaussian noise]{\includegraphics[width=0.45\textwidth]{case1.png}\label{fig: Gauss}}
% \subfigure[multivariate $t_2$ noise]{\includegraphics[width=0.45\textwidth]{case2.png}\label{fig: t_2}}
% \subfigure[Pareto copula noise]{\includegraphics[width=0.45\textwidth]{case6.png}\label{fig: Pareto}}
% \subfigure[Banana-shaped noise]{\includegraphics[width=0.45\textwidth]{case3.png}\label{fig: Banana}}

In each experiment, we draw i.i.d.\ data $(X_1,Y_1),\ldots,(X_n,Y_n)$ according to model~\eqref{mlm}, where the regression coefficients $b^* \in \mathbb{R}^{d \times p}$ has independent $\mathcal{N}(5, 5)$ entries and is kept fixed for all repetitions. Covariates $X_i \in \mathbb{R}^p$, $i=1,\ldots,n$, are drawn from $N(0,\Sigma)$ with a Toeplitz covariance matrix $\Sigma =(2^{-|i-j|})_{i,j} \in \mathbb{R}^{p \times p}$. The noise $\varepsilon$ is generated from one of the following distributions:
\begin{enumerate}[label=(1\alph*), noitemsep]
    \item $\varepsilon\sim \mathcal{N}(0,I_d)$
    \item $\varepsilon \sim t_2(0, I_d)$ follows a multivariate $t_2$ distribution
    \item $\varepsilon$ has each marginal distributed with $\mathrm{Pareto}(-2, 2, 1)$ \footnote{the Pareto distribution $\mathrm{Pareto}(k, \alpha, s)$ has density function $f(x)  \varpropto \frac{\alpha s^{\alpha + 1}}{(x - k)^{\alpha + 1}}$ for all $x\geq 1+k$, with shape parameter $\alpha >0$, location parameter $k \in \mathbb{R}$ and scale parameter $s > 0$. Here $\mathrm{Pareto}(-2, 2, 1)$ has mean 0.} and the same copula as $\mathcal{N}(0,\Sigma')$, where $\Sigma' = (0.9^{|i-j|})_{i,j}\in\mathbb{R}^{d\times d}$
    \item $\varepsilon$ follows a centered Banana-shaped distribution, i.e.\ $\varepsilon_i \stackrel{d}{=}(B_{d-1}, \|B_{d-1}\|^2 - \frac{2}{d+2}) + 0.3 B_d$, where $B_d$ is uniformly distributed in the unit ball in $\mathbb{R}^d$ 
\end{enumerate}

Figure~\ref{fig: Experiment} reports the average matrix Mahalanobis norm error (estimated over $100$ Monte Carlo repetitions) of MCQR, LS, SpQR and CoorCQR over the four noise distributions mentioned above for $n\in\{100,200,\ldots,600\}$ and $(d,p)\in\{(2,7), (4,10)\}$. We see that MCQR has done well over all settings considered here. In contrast, LS estimator performs the best under Gaussian noise but has poor performance under heavy-tailed noise or noise with non-convex support. CoorCQR and SpQR have relatively good performance in panels (a) and (b) when the noise is spherically symmetric but their performance deteriorated when the noise exhibits strong cross-sectional dependence in panels (c) and (d). 

While our theoretical results have mostly concerned with heavy-tailed noise, we also investigate the empirical performance of MCQR in the presence of outlier contamination. Here, we consider two cases of $\epsilon$-contaminated noise, for some $\epsilon \in (0, 1)$:
\begin{enumerate}[label=(2\alph*), noitemsep]
    \item $\varepsilon \sim (1 - \epsilon) P_1+ \epsilon P_2$; here $P_1$ is a Pareto copula with $\mathrm{Pareto}(-\frac{10}{9}, 10, 1)$ marginals and copula generated by $\mathcal{N}(0, \Sigma')$ as in case (1c) and $P_2$ is a heavier-tailed location-shifted Pareto copula with marginals distributed as $\mathrm{Pareto}(10, 2, 10)$.
    \item $\varepsilon \sim (1-\epsilon)\mathcal{N}(0, I_d) + \epsilon \mathcal{N}(100, I_d)$
\end{enumerate}

\begin{figure}[htbp!]
\centering
\subfigure[Pareto contamination]{\includegraphics[width=0.45\textwidth]{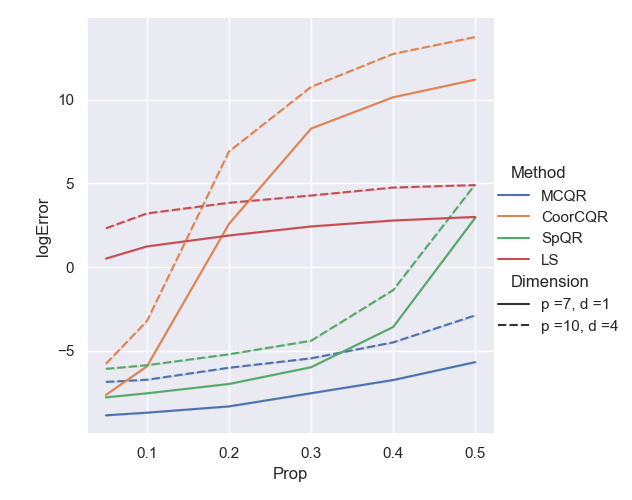}\label{fig: ParetoContam}}
\subfigure[Gaussian contamination]{\includegraphics[width=0.45\textwidth]{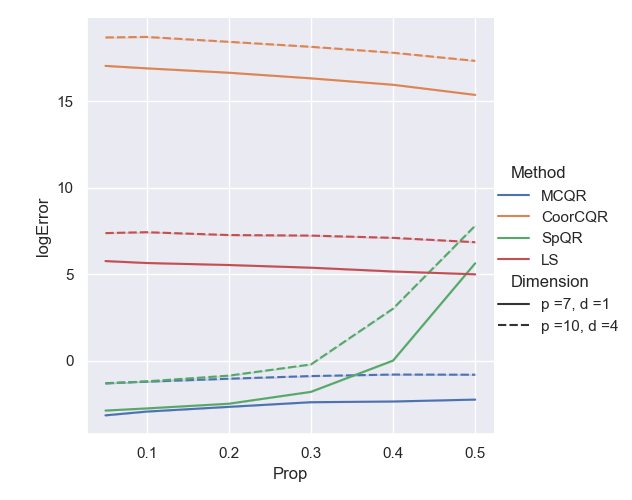}\label{fig: GaussianContam}}

\caption{Logarithmic average estimation loss, measured in matrix Mahalanobis norm, of the regression coefficient estimated by MCQR, CoorCQR, SpQR and LS for data generated according to the mechanism described in Section~\ref{sec: imp} for various outlier contamination proportion (from $0.05$ to $0.5$), covariate dimension $p$ and response dimension $d$ and two different noise contamination models. We fix $n=200$.}

\label{fig: Contamination}
\end{figure}

% \subfigure[Pareto contamination]{\includegraphics[width=0.45\textwidth]{case_outliers_1.png}\label{fig: ParetoContam}}
% \subfigure[Gaussian contamination]{\includegraphics[width=0.45\textwidth]{case_outliers_2.png}\label{fig: GaussianContam}}

Figure~\ref{fig: Contamination} shows the performance of the four procedures for increasing levels of contamination proportion $\epsilon$. We observe that MCQR is generally more robust than other competitors when we add additional outliers to the random error. Interestingly, we see that in the case where $d = 1$, the CoorCQR, which reduces to the univariate CQR, shows a lack of robustness against the outlier contamination, while the 1-dimensional version of MCQR maintains its robustness even with a high proportion of contamination.

%\appendixbibliographystyle
%\bibliographystyle{apalike}
%\bibliography{ref.bib}
% \bibliograpystyle{apa}
\printbibliography
\end{refsection}

\newpage
\appendix
\section*{Appendices}
We presents the proofs of all main results in Appendix \ref{apx: proofs}. Specifically, Appendix \ref{apx: LemmaJustification} -  \ref{apx: PropUnique} contain proof of the theoretical results in Section \ref{sec: MCQR_method}. Then the proof of Lemma \ref{le: wipLowerBound} and Lemma \ref{le: innerprodWassControl} are included in Appendix \ref{apx: wiplemma} and \ref{apx: innerprodWass}, respectively. The consistency Theorem \ref{thm: Consistency} is proved in Appendix \ref{apx: ConsisThm}, while the proof  regarding the convergence rate, as specified in Theorem \ref{thm: FasterRate}, is showed in \ref{apx: FasterRate}. All auxillary results are included in Appendix \ref{apx: Auxillary}.

Appendix \ref{apx: ref_distribution} provides an example of a check function under a $U[-1, 1]$ reference distribution. A brief introduction about spatial quantile is provided in Appendix \ref{Apx: SimulationDetails}

\begin{refsection}

\section{Proof of main results}\label{apx: proofs}

We first record here some notations and several classical results on optimal transport theory that will be used throughout our theoretical analysis.
\subsection{Preliminaries on optimal transport theory}

Define the rescaled squared $\ell_2$-distance as $L_2(x, y):= \frac{1}{2}\|x-y\|^2$ for any $x, y \in \mathbb{R}^d$. In this notation, for two distributions $P$ and $Q$ on $\mathbb{R}^d$, we have 
\begin{align}
    \frac{1}{2}\mathcal{W}_2^2(P, Q) = \inf_{\gamma \in \mathcal{C}(P, Q)}\int L_2(x, y) d \gamma(x, y) =:  I_2(P, Q) \label{opt: w2}.
\end{align}

Our proof depends on the following Kantorovich duality \citep[see e.g.,][Theorem 1.3]{villani2021topics}
\begin{equation}
I_2(P, Q) = \sup_{\varphi, \psi \in \Phi_2} J_{P, Q} (\varphi, \psi),\label{opt: KantoDual}
\end{equation}
where $\Phi_2:= \{(\varphi, \psi) \in L^1(P)\times L^1(Q): \varphi(x) + \psi(y) \leq L_2(x, y)\}$ and 
\[
   J_{P, Q}(\varphi, \psi):= \int \varphi(x) d P(x) + \int \psi(y) dQ(y). 
\]
By taking advantage of the particular form of $L_2$, we also have for $\tilde{\Phi}:= \{(\varphi, \psi) \in L^1(P) \times L^1(Q): \varphi(x) + \psi(y) \geq x^T y\}$ that 
\begin{align}
    \int \frac{\|x\|^2}{2} d  P(x) + \int \frac{\|y\|^2}{2} d  Q(y) - \sup_{\varphi, \psi \in \Phi_2} J_{P, Q}(\varphi, \psi) = \inf_{\varphi, \psi \in \tilde{\Phi}} J_{P, Q}(\varphi, \psi) := \tilde{I}_2(P, Q).  \label{opt: L2EquivalentForm} 
\end{align}
Thus solve the problem of \eqref{opt: KantoDual} degenerates to solve the problem of $\tilde{I}_2(P, Q)$. 

For any $\varphi \in L^1(P)$, define its \textit{Legendre transform} as $\varphi^*(y):= \sup_{x \in \mathbb{R}^d}(x^T y - \varphi(x))$. Then it can be shown that $\varphi^*$ is a \textit{convex lower semi-continuous (l.s.c.)} function. This definition immediately implies that for any $(\varphi, \psi) \in \tilde{\Phi}$, $\psi(y) \geq \varphi^*(y), \ \forall y \in \mathbb{R}^d.$ Thus we have $J_{P, Q}(\varphi, \psi) \geq J_{P, Q}(\varphi, \varphi^*)$. Similarily, we have $\varphi(x) \geq \sup_{y \in \mathbb{R}^d}\bigl(x^T y - \varphi^*(y)\bigr) = \varphi^{**}(x), \ \forall x \in \mathbb{R}^d,$ which further implies that $J_{P, Q}(\varphi, \varphi^*) \geq J_{P, Q}(\varphi^{**}, \varphi^*)$. In the end, we deduced that 
\[\inf_{\varphi, \psi \in \tilde{\Phi}}J_{P, Q}(\varphi, \psi) \geq \inf_{\varphi \in L^1(P)} J_{P, Q}(\varphi^{**}, \varphi^*) \geq \inf_{\text{$\varphi$ is convex l.s.c.}} J_{P, Q}(\varphi^*, \varphi). \]
In fact, it can be shown \Citep[see e.g.][Theorem 2.9]{villani2021topics} that the equality above holds, i.e. \!there  exists a convex l.s.c. \!function $\varphi_0$ such that the \textit{conjugate pair} $(\varphi_0, \varphi_0^*)$ is the optimal solution to $\tilde{I}_2(P, Q)$. Now we are ready to state a fundemental theorem for the optimal transport theory with $L_2$ loss function. 

\begin{theorem}\label{thm: K-SOptimalityCriterion}\citep[Theorem 2.12 and Remark 2.13(iii)]{villani2021topics}
Let $P$ and $Q$ be probability measures on $\mathbb{R}^d$, with finite second moment. We consider the Kantorovich dual problem associated with the rescaled squared $\ell_2$-distance $L_2$. Then $\gamma \in \mathcal{C}(P, Q)$ is optimal if and only if there exists a convex l.s.c. \!function $\varphi_0$ such that 
\[
\mathrm{Supp}(\gamma) \subset \partial \varphi_0,
\]
or equivalently, for $\gamma$-almost all $(x, y)$,
\[
y \in \partial \varphi_0(x).
\]
Moreover, there exists a conjugate pair $(\varphi_0, \varphi_0^*)$ that is a minimizer of $\tilde{I}_2(P, Q)$. Thus $(\|\cdot\|^2/2 - \varphi_0, \|\cdot\|^2/2 - \varphi_0^*)$ solves the Kantorovich dual problem $I_2(P, Q)$.
\end{theorem}

The $1$-Wasserstein distance satisfies the following Kantorovich--Rubinstein duality. 

\begin{theorem}[Kantorovich--Rubinstein theorem]\label{thm: kant-Rubin}
    Suppose $\mathcal{X}$ is a subset of $\mathbb{R}^d$, define the diameter of $\mathcal{X}$ as $\mathrm{diam}(\mathcal{X}) := \sup_{x, y \in \mathcal{X}}\|x - y\|$. Let $\mathrm{Lip}(\mathcal{X})$ denote the space of all Lipschitz function on $\mathcal{X}$ and for any $f$ within this space define
    \[
    \lip{f}{\mathcal{X}} := \max\Bigl\{ \sup_{\substack{x, y \in \mathcal{X} \\ x \not= y}}\frac{|f(x) - f(y)|}{\|x - y\|} , \ \frac{\|f\|_{\infty}} {\mathrm{diam}(\mathcal{X})} \Bigr\}.
    \]
    Then 
    \begin{align}
        \mathcal{W}_1(P, Q) = \sup\Biggl\{ \int f(x) d P(x) - \int f(y) dQ(y): f \in \mathrm{L}^1(|P-Q|), \ f \in \mathrm{Lip}_1(\mathcal{X}) \Biggr\}, \label{eq: Kant-Rubin} 
    \end{align}
    where $\mathrm{Lip}_1(\mathcal{X}) := \{f: \lip{f}{\mathcal{X}}\leq 1\}$.
\end{theorem}
In particular, the $1$-Wasserstein distance can be seen as a special case of a integral probability metric (defined below) with respect to the $\mathrm{Lip}_1$ function class.
\begin{definition}[Integral Probability Metrics]
    Given probability measures $P$ and $Q$ as before, the integral probability metrics (IPMs) with respect to function class $\mathcal{F}$ is defined as 
    \begin{equation}
        \ipm{P}{Q}{\mathcal{F}} = \sup_{f \in \mathcal{F}}\Bigl\{ \int f(x) d P(x) - \int f(y) dQ(y)\Bigr\}. \label{Eq:IPM}
    \end{equation}
\end{definition}
\subsection{Additional notation}
Suppose $T$ is a map from a measurable space $X$, equipped with a measure $\mu$, to an arbitrary space $Y$, we denote by $T\#\mu$ as the push-forward of $\mu$ by $T$. Specifically, $(T\#\mu)(A) = \mu(T^{-1}(A))$ for any measurable set $A$.

Suppose $X_1, \ldots, X_n$ are random samples from some probability distribution $P$. Then given any function class $\mathcal{F}$, define the Rademacher complexity of $\mathcal{F}$ as 
\begin{equation}
\mathcal{R}_n(\mathcal{F}, P) := \E\Bigl(\sup_{f \in \mathcal{F}}\frac{1}{n}\sum_{i = 1}^n\xi_i f(X_i)\Bigr), \label{Eq:RademacherComplexity}    
\end{equation}
where $\xi_i$'s are independent Rademacher random variables, independent from $X_1,\ldots,X_n$. The $p$-dimensional closed ball in centered at $x \in \mathbb{R}^p$ with radius $r > 0$ is denoted by $\mathcal{B}_{x, r}^{p}: = \{y\in \mathbb{R}^p: \|y\| \leq r \}$ and we omit $r$ when $r = 1$: $\mathcal{B}_{x, 1}^{p}:=\mathcal{B}_x^p$. The matrix operator norm is denoted by $\|\cdot\|_{\ope}$, so that $\|A\|_{\ope}:=\sup_{x: \|x\| = 1}\|Ax\|$.

\subsection{Proof for Lemma \ref{le: justification}}\label{apx: LemmaJustification}

\begin{proof}
    For any fixed $\tau \in (0, 1)$, by the definition of check function $\rho_\tau$ we have 
    \[
    q_Y(\tau) \in \argmin_{\theta}\E \rho_\tau(Y - \theta),
    \]
    where $q_Y(\cdot)$ is the quantile function of $Y$. Thus under the linear model \eqref{mlm} we have for any $x \in \mathbb{R}^p$,
    \begin{align}
(b^*, q_{\varepsilon}^*(\tau)) \in \argmin_{b \in \mathbb{R}^{1 \times d}, q \in \mathbb{R}} \E [\rho_\tau (Y- bX - q) \mid X = x]. \label{opt: Conditionaljustfication}
    \end{align}
    For any $b \in \mathbb{R}^{1 \times p}$ and $q \in \mathbb{R}$, define $ g(x; b, q):=\E [\rho_\tau (Y- bX - q) \mid X = x]$, then \eqref{opt: Conditionaljustfication} implies that 
    \[
     g (x; b^*,q_{\varepsilon}^*(\tau)) \leq  g (x; b, q),
    \]
    thus
    \[
    \int_{\mathbb{R}^p}  g (x; b^*,q_{\varepsilon}^*(\tau)) \, d x \leq \int_{\mathbb{R}^p}  g (x; b, q) \, dx.
    \]
    Then by the Fubini Theorem and the Law of iterated expectation, we have
    \begin{align}
    \E[\rho_\tau (Y- b^* X - q_\varepsilon^*(\tau))] \leq \E[\rho_\tau (Y- b X - q)]. \label{opt: Unconjust} 
    \end{align}
    Because the quantile function $q_\varepsilon^* \in \mathcal{M}$, thus \eqref{opt: Unconjust} implies that for any $q(\cdot) \in \mathcal{M}$
    \begin{align}
         \int_0^1 \E[\rho_\tau (Y- b^* X - q_\varepsilon^*(\tau))] \, d \tau \leq \int_{0}^1 \E[\rho_\tau (Y- b X - q(\tau))]  d \tau \notag.
    \end{align}
    Therefore the result follows by applying the Fubini Theorem once again. 
\end{proof}

\subsection{Proof for Lemma \ref{le: 1dmcqr}}\label{apx: Lemma1dmcqr}

\begin{proof}
    Let $\mathcal{C}$ denote the class of convex functions on $[0,1]$. By the definition of the check function $\rho_\tau$ and the fact that $X$ is mean-zero, we have 
    \begin{align}
        \inf_{q\in\mathcal{M}} \mathbb{E} \Bigl\{\int_0^1 \rho_\tau \bigl(Y- bX & - q(\tau)\bigr) d \tau \Bigr\} +  \frac{1}{2}\E Y \notag\\
        &=
            \inf_{q\in\mathcal{M}}\Bigl\{\mathbb{E}\int_0^1(Y - q(\tau) - bX)^+ d\tau +  \int_0^1(1-\tau)q(\tau)d\tau\Bigr\}  \notag \\
            &=
            \inf_{q\in\mathcal{M}}\Bigl\{\mathbb{E} \max_{t\in[0,1]} \int_0^t (Y - q(\tau) - bX) \, d\tau +  \int_0^1 \int_{\tau}^1 q(\tau)\,du\,d\tau\Bigr\}   \notag \\
            &=\inf_{\phi \in \mathcal{C}}\Bigl\{\mathbb{E}\max\limits_{t \in [0, 1]}(t(Y - bX) -\phi(t)) + \mathbb{E}\phi(U)\Bigr\}\notag\\
            &= \inf_{\phi\in\mathcal{C}} \mathbb{E} \bigl\{\phi^*(Y-bX) + \mathbb{E}\phi(U)\bigr\},
            \label{eq: CQRtoWip}
    \end{align}
    where $\phi^*(t) := \max_{t\in[0,1]} \{ut - \phi(u)\}$ is the Legendre conjugate of $\phi:[0,1]\to\mathbb{R}$ and we used Fubini's theorem and a change of variable $q \mapsto \phi\in\mathcal{C}$ defined by $\phi(t) = \int_{0}^t q(\tau)\,d\tau$ in the penultimate step.

    Let $\phi_0$ be the optimizer of \eqref{eq: CQRtoWip} and $\phi_0^*$ its Legendre conjugate, then by~\citet[Theorem~2.9]{villani2021topics}, we have 
    \begin{align*}
    \mathbb{E}\phi_0^*(Y-bX) + \mathbb{E}\phi_0(U) &= \inf_{\phi\in\mathcal{C}} \bigl\{\mathbb{E}\phi^*(Y-bX) + \mathbb{E}\phi(U)\bigr\} \\ &= \inf_{\phi,\psi\in\mathcal{C}:\phi(x)+\psi(y) \geq xy} \bigl\{\mathbb{E}\psi(Y-bX) + \mathbb{E}\phi(U)\bigr\}.
    \end{align*}
    Then by the arguments in~\citet[Sec~2.1.2]{villani2021topics}, the pair $(\tilde\phi_0, \tilde\psi_0)$ defined by $\tilde\phi_0(u) = u^2/2 - \phi_0(u)$ and $\tilde\psi_0(y) = y^2/2 - \phi_0^*(y)$ is the optimizer of the Kantorovich dual formulation of the optimal transport problem between $P^{Y-bX}$ and $P^U$, i.e.  
    \begin{align}
      \mathbb{E}\tilde\psi_0(Y-bX) + \mathbb{E}\tilde\phi_0(U) = \sup_{\substack{\tilde\phi, \tilde\psi \in L^1(\mathbb{R})\\ \tilde\phi(x)+\tilde\psi(y)\leq (x-y)^2/2}} \mathbb{E}\tilde\psi(Y-bX) + \mathbb{E}\tilde\phi(U).  \label{eq: KantDual}
    \end{align}
    By the strong duality theorem \citep[Theorem 1.3]{villani2021topics}, we have 
    \begin{align}
       \frac{1}{2}\mathcal{W}_2^2\bigl(P^{Y-bX}, P^{U} \bigr) &= \mathbb{E}\tilde\psi_0(Y-bX) + \mathbb{E}\tilde\phi_0(U) \notag \\
       &= \mathbb{E} \biggl\{\frac{1}{2} (Y - bX)^2 - \phi_0^*(Y-bX)\biggr\} + \mathbb{E} \biggl\{\frac{1}{2} U^2 -  \phi_0(U) \biggl\} \label{eq: WassKanto},
    \end{align}
    which together with the definition of $\wip{\cdot}{\cdot}$ implies that 
    \[
    \wip{P^{Y-bX}}{P^U} = \mathbb{E} \phi_0^*(Y - bX) + \mathbb{E}\phi_0(U).
    \]
    The result follows by combining the above identity with the optimality of $\phi_0$ in \eqref{eq: CQRtoWip}.
\end{proof}
\subsection{Proof for Proposition \ref{prop: unique}}\label{apx: PropUnique}
\begin{proof}
    By Brenier's Theorem\citep[Theorem 2.12 (ii)]{villani2009optimal}, there is a unique (invertible) optimal transport map $\phi:\mathbb{R}^d\to\mathbb{R}^d$ from $P^U$ to $P^\varepsilon$, which induces a coupling $P^{(U,\varepsilon)}:=(\phi\otimes \Id)\# P^U \in \mathcal{C}(P^U, P^\varepsilon)$. Then $P^{(U,\varepsilon)}\otimes P^{(b^*-b)X}$ is a joint distribution of $(U,\varepsilon,(b^*-b)X)$, which induces a joint distribution $P^{(U, Y-bX)} \in \mathcal{C}(P^U, P^{Y-bX})$ through the map $(u,e,z) \mapsto (u,e+z)$. Observe that the squared $L_2$ transport cost associated with $P^{(U,Y-bX)}$ is
    \begin{align}
        \int \|u - v\|_2^2 \, d P^{(U,Y-bX)}(u, v )
        &= 
        \int \| u - (e + z) \|_2^2  \, d(P^{(U,\varepsilon)}\otimes P^{(b^*-b)X})(u, e, z)\notag\\
        &= \int \|\phi(u) - u\|_2^2 \,d P^{U}(u) + \int \|z\|_2^2 \, d P^{(b^* - b)X}(z) \notag \\
        &= \mathcal{W}_2^2(P^U,P^{\varepsilon}) + \E\|(b^* - b)X\|_2^2. \label{eq: costlambda}  
    \end{align}
    Therefore, we have 
    \begin{align}
        \mathcal{L}(b; U) - \mathcal{L}(b^*; U) & = -\mathcal{W}_2^2(P^U, P^{Y-bX}) + \mathcal{W}_2^2(P^U,P^{\varepsilon}) + \E\|(b^* - b)X\|_2^2 \notag\\
        & = \int \|u - v\|_2^2 \, d P^{(U,Y-bX)}(u, v ) \notag \\
        & \quad - \inf_{Q\in\mathcal{C}(P^{U},P^{Y-bX})} \int \|u - v\|_2^2 \, d Q(u,v) \geq 0.
        \label{eq: unique}
    \end{align}
    This implies that $b^* \in \argmin \mathcal{L}(b; U)$. To prove the uniqueness, by Brenier's Theorem, since $P^U \in \mathcal{P}_{\mathrm{ac}}(\mathbb{R}^d)$, the optimal transport map from $P^U$ to $P^{Y-bX}$ is unique, thus the equality can only be achieved in~\eqref{eq: unique} if $P^{(U,Y-bX)}$ is the optimal coupling. In such a case, by the Knott-Smith optimality criterion \citep[Theorem 2.12(i)]{villani2021topics}, there exists a unique convex lower semi-continuous function $h:\mathbb{R}^d \to \mathbb{R}^d$ such that $\mathrm{Supp}(P^{(U,Y-bX)}) \subset \mathrm{Graph}(\nabla h)$ in the sense that, for any $(u, v)\in \mathrm{Supp}(P^{(U,Y-bX)})$, we have $v = \nabla h(u)$. Define an event $A = \{\nabla h(\phi^{-1}(\varepsilon))) = \varepsilon + (b^* - b)X\}$. Then
    \begin{align}
        \Prob(A) &= \Prob\bigl((\phi^{-1}(\varepsilon), \varepsilon + (b^* - b)X ) \in \{(u, v): \nabla h(u) = v\}\bigr) \notag \\
        & = P^{(U, Y-bX)} \{(u, v): \nabla h(u) = v\}  = 1 \notag.
    \end{align}
    This implies that $\varepsilon + (b^* - b)X = \nabla h (\phi^{-1}(\varepsilon))$ almost surely. Because $X$ is independent of $\varepsilon$, and is not a point mass, the only way to make this equality hold  is when $b=b^*$ as desired.
    \end{proof}
\subsection{Proof for Theorem \ref{thm: Consistency}}\label{apx: ConsisThm}
For notation simplicity, write $S := \Sigma^{-1/2}X$ and $S_i := \Sigma^{-1/2}X_i$ for $i \in [n]$ throughout the rest of the paper. 
\begin{proof}
    By the definition of $\hat b$ in~\eqref{opt: mcqr_est}, we have the following basic inequality:
    \begin{align}
    \mathcal{L}(\hat{{b}}) - \mathcal{L}({{b}^*}) \leq \mathcal{L}(\hat{{b}}) - \mathcal{L}_{n,m}(\hat{{b}}) + \mathcal{L}_{n,m}({{b}^*}) - \mathcal{L}({{b}^*})  \label{ineq: rbasic}.
    \end{align}
    By the explicit formula for the 2-Wasserstein distance between two elliptical distributions \citep[see][Theorem 2.1]{gelbrich1990formula}, we have 
    \begin{align}
    \wip{P^{(b^* - b)X}}{P^U} &= \frac{1}{2}\Bigl\{\E \|(b^* - b)X\|^2+ \E\|U\|^2 - \mathcal{W}_2^2(P^{(b^* - b)X}, P^U) \Bigr\}\nonumber\\
    &=\frac{1}{2}\Bigl\{\E \|(b^* - b)X\|^2+ \E\|U\|^2 - \bigl\|\bigl((b^* - b)\Sigma(b^* - b)^T\bigr)^{1/2} - I_d\bigr\|_{\F}^2\Bigr\}\nonumber\\
    &= \Tr\bigl\{\bigl((b^* - b)\Sigma(b^* - b)^T\bigr)^{1/2}\bigr\}\\
    &\geq \Tr^{1/2}\bigl\{(b^* - b)\Sigma(b^* - b)^T\bigr\} = \|b^* -b\|_{\Sigma}.\label{Eq:EllipticalWIP}
    \end{align}
    Hence, writing $r := \wip{P^{\varepsilon}}{P^U}$, we have by Lemma \ref{le: wipLowerBound} that for any $b\in\mathbb{R}^{d\times p}$,
    \begin{align}
        \mathcal{L}({b}) - \mathcal{L}({{b}^*}) &= \wip{P^{(b^*-b)X + \varepsilon}}{P^{U}} -\wip{P^{\varepsilon}}{P^{U}} \notag\\
        %&= \wip{P^{\varepsilon} \ast P^{(b^* - b)X}}{P^U} - %\wip{P^{\varepsilon}}{P^U} \notag \\
        & \geq \sqrt{r^2 + \wip{P^{(b^* - b)X}}{P^U}^2} -  r \geq \sqrt{r^2 + \|b^* -b\|_{\Sigma}^2} -r. 
        \label{ineq: lb}
    \end{align}
    %where the last inequality follows by Lemma \ref{le: eleinequality}.
    
    On the other hand, by Lemma~\ref{le: innerprodWassControl}, we have
    \begin{align}
        |\mathcal{L}(b) - \mathcal{L}_{n,m}(b)| &= \Bigl|\wip{P^{Y-bX}}{P^U} - \wip{P_n^{Y-bX}}{P_m^U}\Bigr| \notag \\
        &\leq \alpha_m \mathcal{W}_2(P^{Y -bX}, P_n^{Y -bX})  + (\E\|Y-bX\|^2)^{1/2}\mathcal{W}_2(P^{U}, P_m^{U})  \label{eq: ub},
    \end{align}
    where $\alpha_m := \bigl(\frac{1}{m}\sum_{i=1}^m\|U_i\|^2\bigr)^{1/2}$.
    We control the two terms on the right-hand side of~\eqref{eq: ub} separately. For the first term, suppose $P_1$ is the optimal coupling between $P^{S}$ and $P_n^{S}$, and $P_2$ is the optimal coupling between $P^{\varepsilon}$ and $P_n^{\varepsilon}$. Since $P_1\otimes P_2$ induces a coupling between $P^{Y-bX}$ and $P_n^{Y-bX}$ through the relation $Y-bX = (b^*-b)\Sigma^{1/2}S +\varepsilon$, we have  
    \begin{align*}
        \mathcal{W}_2^2(P^{Y-bX},& P_n^{Y-bX}) \\
        &\leq \int \|(b^* - b)\Sigma^{1/2}s_1 + e_1 - (b^* - b)\Sigma^{1/2}s_2 - e_2\|^2 d (P_1 \otimes P_2)(s_1, s_2, e_1, e_2) \notag \\
        &\leq \int \|b^* -b\|_{\Sigma}^2 \|s_1-s_2\|^2 dP_1(s_1,s_2)  + \int \|e_1 - e_2\|^2 d P_2(e_1, e_2) \notag \\
        & = \|b^* -b\|_{\Sigma}^2 \mathcal{W}_2^2(P^S, P^S_n) + \mathcal{W}_2^2(P^{\varepsilon}, P_n^{\varepsilon}).
    \end{align*}
    Thus,
    \begin{equation}
        \mathcal{W}_2(P^{Y-bX}, P_n^{Y-bX}) \leq  \|b^* -b\|_{\Sigma} \mathcal{W}_2(P^S, P^S_n) + \mathcal{W}_2(P^\varepsilon, P_n^{\varepsilon}) =: I_n(\|b^*-b\|_{\Sigma}).
        \label{ineq: ubW2}
    \end{equation}
    For the second term on the right-hand side of~\eqref{eq: ub},  define $s^2:= \E\|\varepsilon\|^2$, we have 
    \begin{align}
        (\E\|Y-bX\|^2)^{1/2} &= (\E\|(b^* -b)X + \varepsilon\|^2)^{1/2} \leq  \{2\mathbb{E}\|(b^*-b)X\|^2 + 2\mathbb{E}\|\varepsilon\|^2\}^{1/2}\notag\\
        &= \bigl\{2\|b^*-b\|_{\Sigma}^2 + 2s^2\bigr\}^{1/2}.
        %&\leq \begin{cases}\frac{\|b^* -b\|_{\Sigma}^2}{\sqrt{2s}} + \sqrt{2s} ,& \text{if $\|b^* -b\|_{\Sigma} < 1$} \\ \sqrt{2}(\|b^* -b\|_{\Sigma} - 1) + \frac{1}{\sqrt{2s}} + \sqrt{2s},& \text{otherwise,}\end{cases} 
        \label{ineq: Y-bX}
    \end{align} 
    Combining~\eqref{eq: ub},~\eqref{ineq: ubW2} and~\eqref{ineq: Y-bX}, we obtain that
    \begin{align}
        |\mathcal{L}(b) - \mathcal{L}_{n,m}(b)| \leq  \alpha_m  I_n(\|b^* - b\|_{\Sigma})  +  \bigl\{2\|b^*- b\|_{\Sigma}^2 + 2s^2\bigr\}^{1/2}\mathcal{W}_2(P^{U}, P_m^{U}) . \label{ineq: ub} 
    \end{align}
    Since~\eqref{ineq: lb} and~\eqref{ineq: ub} holds for arbitrary $b\in\mathbb{R}^{d\times p}$, we have by~\eqref{ineq: rbasic} that 
    \begin{align}
    \bigl\{r^2 + \|b^* -\hat b\|_{\Sigma}^2\bigr\}^{1/2}-r &\leq \alpha_m I_n(\|b^* - \hat{b}\|_{\Sigma})  +   \bigl\{2\|b^*-\hat b\|_{\Sigma}^2 + 2s^2\bigr\}^{1/2}\mathcal{W}_2(P^{U}, P_m^{U}) \notag \\
    &\quad + \alpha_mI_n(0) + s\sqrt{2}\mathcal{W}_2(P^U, P_m^U). \notag 
    % & \leq (\|b^* - \hat{b}\|_{\Sigma} \vee 1)\bigl(\alpha_m \mathcal{W}_2(P^S, P_n^S) + (\sqrt{2} + 2s\sqrt{2})\mathcal{W}_2(P^U, P_m^U) + 2\alpha_m\mathcal{W}_2(P^\varepsilon, P_n^{\varepsilon})\bigr).
    \end{align}
    We apply Lemma \ref{le: eleinequality} to the left-hand side of the above and combine with the fact that $r^2 \leq s^2 d$, we deduce that for some constant $C>0$ only depending on $d$, the following inequality holds: 
    \begin{align}
       &\frac{(2\|b^* - \hat{b}\|_{\Sigma} - 1) \wedge \|b^* - \hat{b}\|_{\Sigma}^2}{(\|b^* - \hat{b}\|_{\Sigma} \vee 1)} \notag \\
        &\qquad\leq C(2 + 2s)\bigl(\alpha_m \mathcal{W}_2(P^S, P_n^S) + (\sqrt{2} + 2s\sqrt{2})\mathcal{W}_2(P^U, P_m^U) + 2\alpha_m\mathcal{W}_2(P^\varepsilon, P_n^{\varepsilon})\bigr)\label{ineq: case1ulb}.
    \end{align}
    Thus we only need to control the right-hand side of the above.

    Note by Markov's inequality, $E_0^{(m)}:= \{\alpha_m \leq \sqrt{d \log m}\}$ holds with probability at least $1 - (\log m)^{-1}$. Similarily, by the convergence rate of empirical 2-Wasserstein distance in Theorem \ref{thm: WassRate} implies that there exists constants $C_1 > 0$ depending only on $p$ and $\ell$ and $C_2$, $C_3 >0$ depending only on $d$, $\ell$ such that for all $m,n > 1$, events $E_1^{(n)} := \{ \mathcal{W}_2(P^S, P_n^S) \leq  C_1 \tau_n^{1/2}(p, \ell) \log^{1/2}n \}$, $E_2^{(n)}: = \{\mathcal{W}_2(P^\varepsilon, P_n^\varepsilon) \leq C_2 \tau_n^{1/2}(d, \ell) \log^{1/2}n  \}$ and $E_3^{(m)}: = \{\mathcal{W}_2(P^U, P_m^U) \leq C_3 \tau_m^{1/2}(d, \ell) \log^{1/2}m \}$ hold with probability at least $1 - (\log n)^{-1}$, $1 - (\log n)^{-1}$, $1 - (\log m)^{-1}$, respectively. Therefore, for all $n>1$ and $m >n$, let $E^{(n,m)} := E_0^{(m)} \bigcap E_1^{(n)}\bigcap E_2^{(n)} \bigcap E_3^{(m)}$, we have $\Prob(E^{(n,m)}) \geq 1 - 4(\log n)^{-1}$.

    Note
    \begin{align}
        \frac{(2\|b^* - \hat{b}\|_{\Sigma} - 1) \wedge \|b^* - \hat{b}\|_{\Sigma}^2}{(\|b^* - \hat{b}\|_{\Sigma} \vee 1)} \geq \|b^* - \hat{b}\|_{\Sigma}^2 \wedge 1 \notag. 
    \end{align}
    Then combining this with (\ref{ineq: case1ulb}), and working on the event $E^{(n,m)}$, there exists some constant $ \tilde M >0$ depending only on $d, \ell, p$ such that
    \begin{align}
        \|b^* - \hat{b}\|_{\Sigma}^2 \wedge 1 &\leq \tilde M (1 + s)\bigl( \tau_{n}^{1/2}(p, \ell) + s \tau_{n}^{1/2}(d, \ell)\bigr)\log^{1/2} m \notag \\
        &\leq \tilde M \bigl(n^{-1/4} + n^{-\frac{1}{ d \vee p}} + n^{\frac{2 -\ell}{2 \ell}} \bigr)\log m, \notag 
    \end{align} 
    where a positive constant depending on $d$ is absorbed in $\tilde M$ in the final inequality, while we stick with notation $\tilde M$ for simplicity. 
    \end{proof}

\subsection{Proof for Lemma \ref{le: wipLowerBound}}\label{apx: wiplemma}
\begin{proof}
By the Brenier's Theorem \citep[Theorem 2.12 (ii)]{villani2009optimal}, there exists optimal transport maps $\phi, \psi: \mathbb{R}^d\to\mathbb{R}^d$ such that $\phi\# P^{\varepsilon} = P^U$ and $\psi\# P^Z = P^U$. Now, for any fixed $t \in [0, 1]$, we define $M_t(z, e): = \sqrt{1-t}\psi(z)+\sqrt{t}\phi(e)$, for all $z , e \in \mathbb{R}^d$. Since $M_t(Z,\varepsilon) \stackrel{\mathrm{d}}= U$, there exists a coupling $P^{(Z,\varepsilon,U)}\in \mathcal{C}(P^{Z}\otimes P^{\varepsilon}, P^U)$ whose associated transport map is $M_t$ (more specifically, $P^{(Z,\varepsilon,U)} = (\mathrm{Id}\otimes M_t)\#(P^Z\otimes P^\varepsilon)$).
Thus, we have  
\begin{align*}
    \llangle Z+\varepsilon, U\rrangle_{\mathcal{W}_2} &\geq \int \langle z+e, u \rangle  d  P^{(Z,\varepsilon,U)}(z,e, u)   \\
    & = \int \bigl\langle  z+e , \sqrt{1-t}\psi(z)+\sqrt{t}\phi(e)\bigr\rangle d (P^Z \otimes P^\varepsilon)(z, e)    \\
    & = \sqrt{1-t} \int \langle z, \psi(z)\rangle d P^Z(z) +\sqrt{t}\int \langle e , \phi(e)   \rangle d P^\varepsilon(e) \\
    & = \sqrt{1-t} \wip{Z}{U} + \sqrt{t} \wip{\varepsilon}{U},
\end{align*}
where in the penultimate step we used the fact that $\varepsilon$ is independent from $Z$. Now, taking $t = \frac{\wip{\varepsilon}{U}^2}{\wip{\varepsilon}{U}^2 + \wip{Z}{U}^2}$, we have 
\[
\wip{Z+\varepsilon}{U}^2 \geq {\wip{Z}{U}^2 + \wip{\varepsilon}{U}^2}
\]
as desired.
\end{proof}

\subsection{Proof for Lemma \ref{le: innerprodWassControl}}\label{apx: innerprodWass}
\begin{proof}
Let $\mathcal{X}_1,\mathcal{X}_2,\mathcal{Y}_1,\mathcal{Y}_2$ denote four copies of $\mathcal{X}$. By Lemma \ref{le: pmcomp}, there exists a distribution $\eta$ on $\mathcal{X}_1\times \mathcal{X}_2\times \mathcal{Y}_1\times \mathcal{Y}_2$ with marginals $P^{X_1}$, $P^{X_2}$, $P^{Y_1}$, $P^{Y_2}$, such that $\eta|_{\mathcal{X}_1 \times \mathcal{X}_2}$, $\eta|_{\mathcal{X}_2\times\mathcal{Y}_2}$, $\eta|_{\mathcal{X}_1\times\mathcal{Y}_1}$ are optimal couplings between $X_1$ and $X_2$, $X_2$ and $Y_2$, and $X_1$ and $Y_1$ respectively. Then we have 
        \begin{align*}
           \wip{X_1}{X_2} &- \wip{Y_1}{Y_2}       \\
           &= \!\sup_{\mu \in \mathcal{C}(P^{X_1}, P^{X_2})} \int\langle x_1, x_2 \rangle\, d \mu (x_1, x_2) -  \!\!\sup_{\nu \in \mathcal{C}(P^{Y_1}, P^{Y_2})}\int\langle y_1, y_2 \rangle \,d \nu (y_1, y_2)  \\
            &\leq \int \langle x_1, x_2 \rangle \, d \eta|_{\mathcal{X}_1 \times \mathcal{X}_2}(x_1,x_2) - \int \langle y_1, y_2 \rangle\, d \eta|_{\mathcal{Y}_1\times\mathcal{Y}_2}(y_1,y_2)  \\
            & \leq \int \langle x_1 , x_2 - y_2 \rangle  - \langle y_1 - x_1 , y_2 \rangle \, d \eta (x_1,x_2,y_1,y_2)    \\
            & \leq \Bigl(\int \|x_2 - y_2\|^2 d \eta|_{\mathcal{X}_2\times\mathcal{Y}_2}(x_2,y_2) \Bigr)^{1/2} \Bigl(\int \|x_1\|^2 d \eta|_{\mathcal{X}_1}(x_1)\Bigr)^{1/2}  \\
            & \qquad \quad + \Bigl(\int \|x_1 - y_1\|^2 d \eta|_{\mathcal{X}_1\times\mathcal{Y}_1}(x_1,y_1) \Bigr)^{1/2} \Bigl(\int \|y_2\|^2 d \eta|_{\mathcal{Y}_2}(y_2)\Bigr)^{1/2}  \\
            & = \mathcal{W}_2(P^{X_2}, P^{Y_2}) \cdot \bigl(\E\|X_1\|^2\bigr)^{1/2} +  \mathcal{W}_2(P^{X_1}, P^{Y_1}) \cdot  \bigl(\E\|Y_2\|^2\bigr)^{1/2},
        \end{align*}
        where we used the Cauchy--Schwarz inequality in the final inequality. 
Similarly, we can find  $\tilde \eta$ such that $\tilde\eta|_{\mathcal{Y}_1 \times \mathcal{Y}_2}$, $\tilde \eta|_{\mathcal{X}_2\times\mathcal{Y}_2}$, $\tilde \eta|_{\mathcal{X}_1\times\mathcal{Y}_1}$ are the corresponding optimal couplings between $Y_1$ and $Y_2$, $X_2$ and $Y_2$, and $X_1$ and $Y_1$ respectively. Then,
\begin{align*}
           \wip{Y_1}{Y_2} - \wip{X_1}{X_2}       
            &\leq \int \langle y_1, y_2 \rangle\, d \tilde \eta|_{\mathcal{Y}_1\times\mathcal{Y}_2}(y_1,y_2) - \int \langle x_1, x_2 \rangle \, d \tilde \eta|_{\mathcal{X}_1 \times \mathcal{X}_2}(x_1,x_2)  \\
            & \leq \int \langle y_1 - x_1 , y_2 \rangle - \langle x_1 , x_2 - y_2 \rangle  \, d \tilde\eta (x_1,x_2,y_1,y_2)    \\
            & \leq  \mathcal{W}_2(P^{X_1}, P^{Y_1}) \cdot  \bigl(\E\|Y_2\|^2\bigr)^{1/2} + \mathcal{W}_2(P^{X_2}, P^{Y_2}) \cdot \bigl(\E\|X_1\|^2\bigr)^{1/2}.
\end{align*}
Combining the above two bounds, we get the desried results. 
\end{proof}

    \begin{lemma}\label{le: pmcomp}
    For $L \in \mathbb{N}$, write $V = \{1,\ldots,L\}$. Let  $(\mathcal{X}_i,\Omega_i,\nu_i)$, $i\in V$ be $L$ probability spaces. Suppose that for some $E \subseteq V\times V$, and for each $(i,j)\in E$, we have a pre-specified joint probability measure $\xi_{i,j}$ on $(\mathcal{X}_i\times \mathcal{X}_j, \Omega_i\otimes \Omega_j)$ such that $\xi_{i,j}|_{\mathcal{X}_i} = \nu_i$ and $\xi_{i,j}|_{\mathcal{X}_j} = \nu_j$. If the simple undirected graph $G = (V,E)$ is acyclic, then there exists a joint probability measure $\rho$ on $\bigl(\prod_{i=1}^L \mathcal{X}_i, \bigotimes_{i=1}^L \Omega_i\bigr)$ such that $\rho|_{\mathcal{X}_i} = \nu_i$ for all $i\in V$ and $\rho|_{\mathcal{X}_i\times \mathcal{X}_j} = \xi_{i,j}$ for all $(i,j)\in E$.
    \end{lemma}
    \begin{proof}
    We assume first that $G$ is connected. Then, there exists a traversal of all the vertices in $G$ such that apart from the first vertex in the traversal, each vertex has exactly one edge connected to an earlier vertex. This can be done by using e.g.\ depth-first search or breadth first search, after arbitrarily assigning a root node, and each node is connected only to its parent node when first visited. Hence, without loss of generality, we may relabel the nodes so that this traversal is given by the ordering $1,2,\ldots, L$. We now prove by induction that for any $\ell\in \{1,\ldots, L\}$, there exists a measure $\rho_{1,\ldots,\ell}$ on $\mathcal{X}_1\times\cdots\times \mathcal{X}_\ell$ such that $\rho_{1,\ldots,\ell}|_{\mathcal{X}_i} = \nu_i$ for all $i\in \{1,\ldots,\ell\}$ and $\rho_{1,\ldots,\ell}|_{\mathcal{X}_i\times\mathcal{X}_j} = \xi_{i,j}$ for all $(i,j)\in E\cap \{1,\ldots,\ell\}^2$. 

    The base case of the induction is trivially true as we can take $\rho_1 = \nu_1$. Now assume that we have successfully constructed $\rho_{1,\ldots,\ell-1}$ for some $\ell \in \{2,\ldots, L\}$. Let $\ell'$ be the only neighbour of $\ell$ in $\{1,\ldots,\ell-1\}$ (the existence and uniqueness of $\ell'$ is guaranteed by the traversal ordering of the vertices in the previous paragraph). By the Disintegration Theorem \citep[see e.g.][]{graf1989classification}, there exists a probability measure $\xi_{\ell\mid \ell'}(\cdot \mid x_{\ell'})$ on $\mathcal{X}_\ell$ such that $d\xi_{\ell\mid \ell'}(x_\ell \mid x_{\ell'}) d\nu_{\ell'}(x_{\ell'}) = d\xi_{\ell',\ell}(x_{\ell'},x_\ell)$. Now, we define
    \[
    d\rho_{1,\ldots,\ell}(x_1,\ldots,x_\ell) = d\rho_{1,\ldots,\ell-1}(x_1,\ldots,x_{\ell -1}) d\xi_{\ell\mid\ell'}(x_\ell\mid x_{\ell'}).
    \]
    To see that $\rho_{1,\ldots,\ell}$ satisfies the required conditions, we check that for any $B \in \Omega_i$, $\rho_{1,\ldots,\ell}|_{\mathcal{X}_i}(B) = \rho_{1,\ldots,\ell-1}|_{\mathcal{X}_i}(B) = \nu_i(B)$ if $i \leq \ell-1$ and 
    \begin{align*}
    \rho_{1,\ldots,\ell}|_{\mathcal{X}_\ell}(B) &= \rho_{1,\ldots,\ell}(\mathcal{X}_1\times\cdots\times \mathcal{X}_{\ell-1}\times B)= \int_{\mathcal{X}_{\ell'}} \int_{B} d\xi_{\ell\mid\ell'}(x_\ell\mid x_{\ell'})  d\rho_{1,\ldots,\ell-1}|_{\mathcal{X}_{\ell'}}(x_{\ell'})\\
    &=\int_{\mathcal{X}_{\ell'}} \int_{B} d\xi_{\ell\mid\ell'}(x_\ell\mid x_{\ell'})  d\nu_{\ell'}(x_{\ell'}) = \xi_{\ell',\ell}(\mathcal{X}_{\ell'}\times B) = \nu_{\ell}(B),
    \end{align*}
    if $i = \ell$. Moreover, if $(i,j)\in E\cap\{1,\ldots,\ell\}^2$, then for $A\in\Omega_i$ and $B\in\Omega_j$, we either have $(i,j) \in E\cap\{1,\ldots,\ell-1\}^2$, in which case $\rho_{1,\ldots,\ell}|_{\mathcal{X}_i\times \mathcal{X}_j}(A\times B) = \rho_{1,\ldots,\ell-1}|_{\mathcal{X}_i\times \mathcal{X}_j}(A\times B) = \xi_{i,j}(A\times B)$, or $(i,j) = (\ell', \ell)$ (or $(\ell, \ell')$ which can be handled symmetrically), in which case, 
    \begin{align*}
    \rho_{1,\ldots,\ell}|_{\mathcal{X}_{\ell'}\times \mathcal{X}_\ell}(A\times B) &= \int_{A} \int_{B} d\xi_{\ell\mid\ell'}(x_\ell\mid x_{\ell'})  d\rho_{1,\ldots,\ell-1}|_{\mathcal{X}_{\ell'}}(x_{\ell'})\\
    &=\int_{A} \int_{B} d\xi_{\ell\mid\ell'}(x_\ell\mid x_{\ell'})  d\nu_{\ell'}(x_{\ell'}) = \xi_{\ell',\ell}(A\times B).
    \end{align*}
    This completes the induction. In particular, $\rho_{1,\ldots,L}$ satisfies the desired properties of $\rho$ in the lemma. 
    \end{proof}
    \subsection{Proof for Theorem \ref{thm: FasterRate}}\label{apx: FasterRate}
    Define event $\Theta := \{\|\hat b - b^*\|_{\Sigma} < 1\}$, then in the regime of (\ref{Condition: nmSufficientlyLarge}) we have $\Prob(\Theta) \geq 1 - 4(\log n)^{-1}$. We henceforth work on the event $\Theta$ throughout the proof. Write the linear transformation $A(b)  = (b^* - b)X + \varepsilon$ for any $b \in \mathbb{R}^{d \times p}$.

Our proof strategy for Theorem~\ref{thm: FasterRate} is to use the fact that $b^*$ maximizes $\mathcal{L}$ and $\hat b$ maximizes $\mathcal{L}_n$ to bound $\mathcal{L}(\hat b) - \mathcal{L}(b^*)$ by $|\mathcal{L}(b^*) - \mathcal{L}_n(b^*)| + |\mathcal{L}(\hat b) - \mathcal{L}_n(\hat b)|$. Write $\mathcal{B}:= \{b \in \mathbb{R}^{d \times p}: \|b - b^*\|_{\Sigma} < 1\}$. Then on the event $\Theta$, the key to control the latter is to establish a bound on 
\[
\sup_{b\in\mathcal{B}}\Bigl|\mathcal{W}_2^2(P^{A(b)}, P^U) - \mathcal{W}_2^2(P_n^{A(b)}, P_m^U) \Bigr|
\]
in Proposition~\ref{Prop:W2distdiff}. The proof of Proposition~\ref{Prop:W2distdiff} relies on rewriting the Wasserstein distances using the Kantorovich dual formulation. Specifically, writing $\tilde{\Phi}_b := \{(f, g) \in L^1(P_n^{A(b)}) \times L^1(P_m^U): v^T u \leq f(v) + g(u), \ \forall (v, u) \in \mathrm{Supp}(P_n^{A(b)}) \times \mathrm{Supp}(P_m^U) \}$, then for any fixed $ b \in  \mathcal{B}$, by Theorem \ref{thm: K-SOptimalityCriterion} and Lemma \ref{le: ConjugateUpperBound}, there exists a conjugate pair $(\tilde{\varphi}_{b;n, m}, \tilde{\varphi}_{b;n, m}^*)$ such that 
    \begin{align}
        (\tilde{\varphi}_{b;n, m}^*, \tilde{\varphi}_{b;n, m} ) &= \argmin_{(f, g) \in \tilde{\Phi}_b}  \int f \, d P_n^{A(b)} + \int g \, d P_m^U,\label{Eq:varphipair}\\
        \frac{1}{2}\mathcal{W}_2^2(P_n^{A(b)}, P_m^U) &= \int \|v\|^2/2 - \tilde\varphi_{b;n,m}^*(v) \, dP_n^{A(b)}(v) + \int \|u\|^2/2 - \tilde\varphi_{b;n,m}(u) \, dP_m^U(u),\nonumber
    \end{align}
    and 
    \begin{align}
            \|u\|^2/2 \leq \tilde{\varphi}_{b;n,m}(u) \leq \|u\|^2/2 + L_{b;n,m}, \quad
            \|v\|^2/2 - L_{b;n,m} \leq \tilde{\varphi}_{b;n,m}^*(v) \leq \|v\|^2/2 ,\label{ineq: BoundedConjugate}  
            \end{align}  
    where $L_{b;n,m}:= \max\{L_2(A(b)_i, U_j): 1\leq i \leq n, \ 1\leq j \leq m\}$.
    
    Before stating Proposition~\ref{Prop:W2distdiff}, we first establish two results on extensions of $\tilde\varphi_{b;n,m}$ and $\tilde\varphi^*_{b;n,m}$ onto the entire $\mathbb{R}^d$, which  will form the core of the argument in the proof of Proposition~\ref{Prop:W2distdiff}.    
    
    \begin{proposition}\label{le: ExtProperties}
    Let $\tilde\varphi$ and $\tilde\varphi^*$ be defined as in~\eqref{Eq:varphipair} and set $L_{b;n,m}:=\max_{i\in[n],j\in[m]} L_2(A(b)_i, U_j)$. Let $\zeta_{b;n,m}$, $\varphi_{b;n,m}$ and $\varphi^*_{b;n,m}$ be defined such that for all $v \in \mathbb{R}^d$, 
    \begin{align*}
    \zeta_{b; n,m}(v) &:= \sup_{u \in \mathrm{Supp}(P_n^{A(b)})}\bigl\{v^T u - \tilde{\varphi}_{b;n,m}(u)\bigr\} \vee \bigl(\frac{\|v\|^2}{2} - L_{b;n,m} \bigr),\\
    \varphi_{b;n,m}(v) &:= \sup_{u\in \mathbb{R}^d}\bigl\{v^T u - \zeta_{b;n,m}(u) \bigr\},\\
    \varphi_{b;n,m}^*(v) &:= \sup_{u\in \mathbb{R}^d}\bigl\{v^T u - \varphi_{b;n,m}(u) \bigr\}.
    \end{align*}
    Then we have 
        \begin{enumerate}[label = (\roman*)]
            \item for any $(u, v) \in \mathbb{R}^d \times \mathbb{R}^d$, $v^T u \leq \varphi_{b;n,m}(u) + \varphi_{b;n,m}^*(v)$;
            \item $\varphi_{b;n,m}(u) = \tilde{\varphi}_{b;n,m}(u)$ for $u \in \mathrm{Supp}(P_n^{A(b)})$ and $\varphi_{b;n,m}^*(v) = \tilde{\varphi}_{b;n,m}^*(v)$ for $v \in \mathrm{Supp}(P_m^U)$;
            \item for $u, v \in \mathbb{R}^d$, $-L_{b;n,m} \leq \frac{\|u\|^2}{2} - \varphi_{b;n,m}(u)  \leq 0$ and $0 \leq \frac{\|v\|^2}{2}-\varphi^*_{b;n,m}(v) \leq  L_{b;n,m}$; \label{suble: Int}
            \item Let $\pi_{b;n,m} \in\mathcal{C}(P_n^{A(b)}, P_m^U)$ be the optimal coupling between $P_n^{A(b)}$ and $P_m^U$. Then for any $(u,v)\in \mathrm{Supp}(\pi_{b;n,m})$, we have $v\in \partial \varphi_{b;n,m}(u)$ and $u \in \partial \varphi_{b;n,m}^*(v)$.
        \end{enumerate}
    \end{proposition}
    \begin{proof}
        Note (i) is immediately followed by the definition of $\varphi_{b;n,m}$ and $\varphi_{b;n,m}^*$. For part (ii), note for any $u \in \mathrm{Supp}(P_m^U) $
        \begin{align}
            \varphi_{b;n,m}(u) \leq \sup_{v\in \mathbb{R}^d}\bigl\{ v^T u - v^T u + \tilde{\varphi}_{b;n,m}(u) \bigr\} = \tilde{\varphi}_{b;n,m}(u) \label{ineq: phi} .
        \end{align}
        For any $v \in \mathrm{Supp}(P_n^{A(b)})$,
        \begin{align}
            \varphi_{b;n,m}^*(v) &\leq \sup_{u \in \mathbb{R}^d} \bigl\{v^T u - v^T u + \zeta_{b;n,m}(v)\bigr\} \notag \\ &= \zeta_{b;n,m}(v) \leq \tilde{\varphi}_{b;n,m}^*(v) \vee \bigl(\frac{\|v\|^2}{2} - \|c\|_{\infty} \bigr) \leq \tilde{\varphi}_{b;n,m}^*(v) \label{ineq: phistar}.
        \end{align}
        Assume any of \eqref{ineq: phi} or \eqref{ineq: phistar} holds strictly, then because $P_n^{A(b)}$ and $P_m^{U}$ are finitely support it follows that 
        \begin{align}
            \int \varphi_{b;n,m}(u) dP_m^U(u) + \int \varphi_{b;n,m}^*(v) d P_n^{A(b)}(v) < \int \tilde{\varphi}_{b;n,m}(u) dP_m^U(u) + \int \tilde{\varphi}_{b;n,m}^*(v) d P_n^{A(b)}(v), \notag 
        \end{align}
        which contradicts to the optimality of $(\tilde{\varphi}_{b;n,m}, \tilde{\varphi}_{b;n,m}^*)$. This completes the proof for (ii). 
    
        For part (iii), by the bounded property \eqref{ineq: BoundedConjugate} and preceding constructions we have for $u \in \mathbb{R}^d$
        \begin{align}
            \|u\|^2/2 - \varphi_{b;n,m}(u)  \geq \inf_{v\in \mathbb{R}^d} \bigl\{L_2(u, v) -  L_{b;n,m}\bigr\} = - L_{b;n,m}. \label{ineq: lowervarphi} 
        \end{align}
        Moreover, we have
        \begin{align}
            \|u\|^2/2 - \varphi_{b;n,m}(u) &\leq -\bigl(\|u\|^2/2 - \zeta_{b;n,m}(u)\bigr) \notag \\
            &= -\inf_{u^\prime \in \mathrm{Supp}(P_n^{A(b)})}\bigl(L(u, u^\prime) - (\|u^\prime\|^2/2 - \tilde{\varphi}_{b;n,m}(u^\prime))\bigr) \wedge L_{b;n,m} \leq 0 \label{ineq: uppervarphi},
        \end{align}
        where the last step follows by the fact that $\|u^\prime\|^2/2 - \tilde{\varphi}_{b;n,m}(u^\prime)\leq 0$, for all $u^\prime \in \mathrm{Supp}(P_n^{A(b)})$. Here, we proved that $ -L_{b;n,m} \leq \frac{\|u\|^2}{2} - \varphi_{b;n,m}(u)  \leq 0 $ and the result holds. For any $v \in \mathbb{R}^d$, by \eqref{ineq: uppervarphi} we have 
        \begin{align}
            \|v\|^2/2 - \varphi^*_{b;n,m}(v) = \inf_{u \in \mathbb{R}^d}\Bigl(L_2(u, v) - (\|u\|^2/2 - \varphi_{b;n,m}(u))\Bigr) \geq 0. \label{Lowerbound: varphi^*}  
        \end{align}
        Moreover, by \eqref{ineq: lowervarphi} it follows that
        \begin{align}
            \|v\|^2/2 - \varphi^*_{b;n,m}(v) \leq -(\|v\|^2/2 - \varphi_{b;n,m}(v)) \leq L_{b;n,m}.\label{Upperbound: varphi^*} 
        \end{align}
        Thus we have $0 \leq \frac{\|v\|^2}{2}-\varphi^*_{b;n,m}(v) \leq  L_{b;n,m}$ as desired. 
    
        To prove (iv), note (ii) implies that 
        \begin{align}
            \int (\varphi_{b;n,m}(u) + \varphi_{b;n,m}^*(v) - v^T u )d \pi_{b; n,m}(u, v) = 0. \notag 
        \end{align}
        Furthermore, part (i) implies that the integrand of the above is nonnegative. Thus it follows that 
        \begin{align}
            \varphi_{b;n,m}(u) + \varphi_{b;n,m}^*(v)  = v^T u, \quad \forall (u, v) \in \mathrm{Supp}(\pi_{b; n,m}). \notag 
        \end{align}
        Then the conclusion follows by \cite[Proposition 2.4]{villani2021topics}.
    \end{proof}

    Now we argue that for all $b \in \mathcal{B}$, $\varphi_{b;n,m}^*$ (and similarly, $\varphi_{b;n,m}$) is a piecewise Lipschitz function on a high probability event that does not depend on $b$. The following lemma plays a key role in the argument. It implies that the local Lipschitz constant of $\varphi_{b;n,m}^*$ is largely driven by the magnitude of the subdifferential of $\varphi_{b;n,m}^*$. The proof is analogous to \Citet[Lemma 10]{manole2021sharp}, but for the sake of completeness, we provide it here. 

    \begin{lemma}\label{le: BoundedSubdiff}
    Suppose $P$ and $Q$ are two distributions on $\mathbb{R}^d$. Let $(\varphi_0, \varphi_0^*)$ be the conjugate pair that solves $\tilde I_2(P, Q)$ (see \eqref{opt: L2EquivalentForm}). Then for any $r \geq 1$, $\varphi_0: \mathcal{B}_{0, r}^d \to \mathbb{R}$ and $\varphi_0^*: \mathcal{B}_{0, r}^d \to \mathbb{R}$ are Lipschitz continuous with parameters $L_0$ and $L_0^*$ respectively, where
         % For any $r \geq 1$, $\varphi_{b;n,m}^*: \mathcal{B}_{0, r}^{d} \to \mathbb{R}$ and $\varphi_{b;n,m}: \mathcal{B}_{0, r}^{d} \to \mathbb{R}$ are convex and Lipschitz functions with Lipschitz constant $L_{r, b}$ and $L_{r,b}^\prime$, where 
    \begin{gather*}
    % L_{r, b} := \sup\bigl\{\|y\|: y \in \partial  \varphi_{b;n,m}^* \bigl( \mathcal{B}_{0, r}^{d} \bigr)\bigr\}, \quad \text{and} \quad
    % L_{r, b}^\prime := \sup\bigl\{\|z\|: z \in \partial  \varphi_{b;n,m} \bigl( \mathcal{B}_{0, r}^{d} \bigr)\bigr\}.
    L_0 := \sup\bigl\{\|y\|: y \in \partial \varphi_0\bigl(\mathcal{B}_{0, r}^d\bigr) \bigr\} \quad, \text{and } \quad L_0^* := \sup\bigl\{\|z\|: z \in \partial\varphi_0^*(\mathcal{B}_{0, r}^d)\bigr\}
    \end{gather*}
\end{lemma}
\begin{proof}
    We focus on $\varphi_0$ and the same argument can be used for $\varphi_0^*$. Firstly, by \citet[Proposition 2.4]{villani2021topics}, for any $v \in \mathcal{B}_{0, r}^{d}$, $\varphi_0$ admits the following representation
    \begin{align}
        \varphi_0(v) = \sup_{u \in \partial\varphi_0 (v)}\bigl\{u^T v - \varphi_0^*(u)\bigr\}. \notag 
    \end{align}
    Thus, there exists a sequence of $u_k \in \partial\varphi_0 (v) $ such that 
    \begin{align}
         \varphi_0(v) \leq u_k^T v - \varphi_0^*(u_k) + \frac{1}{k}, \quad \text{for} \ k= 1, 2, \ldots. \notag 
    \end{align}
    Then for any $v^\prime \in \mathcal{B}^{d}_{0, r}$, we have 
    \begin{align}
        \varphi_0(v) - \varphi_0(v^\prime) &\leq u_k^T v- \varphi_0^*(u_k) + \frac{1}{k} - u_k^T v^\prime +\varphi_0^* (u_k) \notag \\
        &= u_k^T(v-v^\prime) + \frac{1}{k} \leq L_{0} \|v - v^\prime\| + \frac{1}{k}, \notag 
    \end{align}
    and the Lipschitz property follows by letting $k \to +\infty$.
\end{proof}

For all $j \geq 0$, define $L_j := [-3^j, 3^j]^{d}$ and let  $P_j := L_j \setminus L_{j - 1}$.  We note that each $P_j$ can be further partitioned into $N:=3^{d}-1$ cubes, say $\{P_{j,k}\}_{k=1,\ldots,N}$, that are each congruent to $L_{j-1}$. We note that all elements of $P_j$ has norm bounded by  $\ell_j := \sup_{z \in {P}_j}\|z\| = 3^j\sqrt{d} $. 

For any $I \subset \mathbb{R}^d$, we write  $\mathcal{C}(I)$ for the set of all the convex function on $I$. We define
$\mathcal{C}_{m, u}(I) := \{f\in \mathcal{C}(I): \exists m, u >0, \ \mathrm{s.t.} \ |f(x) - f(y)| \leq m\|x-y\|, \ |f(x)| \leq u, \ \forall x, y \in I\}$ to be the class of $m$-Lipschitz convex functions on $I$ bounded in value by $u$. Given a sequence $M $ and $U$,  define
    \[
    \mathcal{C}_{M,U} : =\bigl\{f: \mathbb{R}^p \times \mathbb{R}^d \to \mathbb{R}: f|_{P_{j,k}} \in \mathcal{C}_{M_j, U_j}(P_{j,k}), \ j\geq 0, \ 1\leq k \leq N  \bigr\}.
    \]
We now prove that for suitable choices of $M$, $U$ and $R,T $, $\varphi_{b;n,m}^* - \varphi_{b;n,m}^*(0)\in \mathcal{C}_{M,U}$ and $\varphi_{b;n,m} - \varphi_{b;n,m}(0) \in \mathcal{C}_{R, T}$ on a high probability event that does not depend on $b$. Recalling that we write $S= \Sigma^{-1/2}X$ and $S_i = \Sigma^{-1/2}X_i$ for $i \in [n]$. 

Let's first discuss the concentration property of $P^U$ and $P^{A(b)}$ and their empirical counterparts $P_m^U$ and $P_n^{A(b)}$. In fact, due to the Gaussian assumption, $P^U$ is a $\sw{\sqrt{2d}}{2}$ distribution. Moreover, by the sub-Weibull assumptions on $S$ and $\varepsilon$, there exists a constant $\sigma>0$ depends on $\sigma_1, \sigma_2$ such that $\|S\| + \|\varepsilon\| \sim \sw{\sigma}{\alpha \wedge \beta}$. Thus by noting that $\|A(b)\| \leq \|S\| + \|\varepsilon\|$ for all $b \in \mathcal{B}$, $P^{A(b)}$ is a $\sw{\sigma}{\alpha \wedge \beta}$ random vector as well. However, the concentration of the corresponding empirical measures introduces extra randomness on the sub-Weibull parameters, as defined here   
\[E_{1, m} = \int \exp\Bigl(\frac{\|u\|^2}{4d}\Bigr) \, dP_m^U, \quad \text{and} \quad E_{b; 2, n} = \int \exp\Bigl(\frac{\|v\|^{\alpha \wedge \beta}}{4 \sigma^{\alpha \wedge \beta}}\Bigr) \, dP_n^{A(b)}.\]
The following lemma constructs the sub-Weibull properties of $P_m^U$ and $P_n^{A(b)}$. 

\begin{lemma}\label{le: SWpropertyofEmpiricalPUandPAb}
    Define $E_{2, n} := \sup_{b \in \mathcal{B}} E_{b;2,n}$. Then for any fixed $n, m \geq 1$ we have that $P_m^U$ is $\sw{(2dE_{1, m})^{1/2}}{2}$ and $P_{n}^{A(b)}$ is $\sw{\sigma (2E_{2,n})^{1/(\alpha \wedge \beta)}}{\alpha \wedge \beta}$, where $E_{1, m} \leq 2 + \sqrt{\frac{\log m}{m}}$ with probability at least $1 - 2(\log m)^{-1}$ and $E_{2, n } \leq 2 + \sqrt{\frac{\log n}{n}}$ with probability at least $1 - 2(\log n)^{-1}$.
\end{lemma}
\begin{proof}
    We only need to note that $E_{1, m} \geq 1$, and Jensen's inequality yields that
    \begin{align}
        \int \exp\Bigl(\frac{\|u\|^2}{4d E_{1, m}}\Bigr) \, dP_m^U \leq E_{1, m}^{\frac{1}{E_{1, m}}} \leq 2. \notag 
    \end{align}
    One the other hand, for each fixed $b \in \mathcal{B}$, a similar calculation can be applied to $P_n^{A(b)}$ and obtain that $P_n^{A(b)} \sim \sw{\sigma(2 E_{b;2,n})^{1/(\alpha \wedge \beta)}}{\alpha \wedge \beta}$. Thus by noting that 
    \[
    \int \exp\Bigl(\frac{\|v\|^{\alpha \wedge \beta}}{ 4E_{b;2,n}\sigma^{\alpha \wedge \beta}}\Bigr) \, dP_n^{A(b)}(v) \geq \int \exp\Bigl(\frac{\|v\|^{\alpha \wedge \beta}}{ 4E_{2,n}\sigma^{\alpha \wedge \beta}}\Bigr) \, dP_n^{A(b)}(v) 
    \]
    we have $P_n^{A(b)} \sim \sw{\sigma (2 E_{2, n})^{1/(\alpha \wedge \beta)}}{\alpha \wedge \beta}$

    Now we control the sub-Weibull parameters. Define ${\Gamma}_1:= \bigl\{E_{1, m}  \leq 2 + \sqrt{\frac{\log m}{m}} \bigr\}$, then by the Chebyshev's inequality we have
    \begin{align}
        \Prob({\Gamma}_1^c)  \leq \Prob \Bigl(|E_{1, m} - \E E_{1, m}| \geq \sqrt{\frac{\log m}{m} }\Bigr) \leq  \frac{m\mathrm{Var}(E_{1, m})}{\log m} \leq \frac{2}{\log m}\notag.  
    \end{align}
    To control $E_{2,n}$, we first note 
    \begin{align*}
        \E E_{2, n} &= \E\sup_{b \in \mathcal{B}}\exp\Bigl(\frac{1}{4}\bigl(\frac{\|A(b)\|}{{\sigma}}\bigr)^{\alpha \wedge \beta}\Bigr) \leq \E \exp\Bigl(\frac{1}{4}\bigl(\frac{\|S\| + \|\varepsilon\|}{{\sigma}}\bigr)^{\alpha \wedge\beta} \Bigr) \leq 2. \notag 
    \end{align*}
    Then define $\Gamma_2:= \Bigl\{E_{2, n} \leq 2 + \sqrt{\frac{\log n}{n}}\Bigr\}$, then we have 
    \begin{align*}
        \Prob(\Gamma_2^c) &\leq \Prob \Bigl(E_{2, n} - \E \exp\Bigl(\frac{1}{4}\bigl(\frac{\|S\| + \|\varepsilon\|}{{\sigma}}\bigr)^{\alpha \wedge\beta} \Bigr)  \geq \sqrt{\frac{\log n}{n} }\Bigr)  \\
        &\leq \Prob\Bigl(\frac{1}{n}\sum_{i = 1}^n\exp\Bigl(\frac{1}{4}\bigl(\frac{\|S_i\|+ \|\varepsilon_i\|}{ \sigma}\bigr)^{\alpha \wedge \beta}\Bigr) -\E \exp\Bigl(\frac{1}{4}\bigl(\frac{\|S\| + \|\varepsilon\|}{{\sigma}}\bigr)^{\alpha \wedge\beta} \Bigr)  \geq \sqrt{\frac{\log n}{n}}\Bigr) \\
        &\leq \frac{2}{\log n},
    \end{align*}
    where the final inequality is obtained by Chebyshev's inequality. 
\end{proof}

\begin{proposition}\label{prop: PiecewiseLipvarphi}
Let $J_n = \floor*{\frac{1}{2}\log_3 \Bigl(\frac{\log n}{16\gamma_2d}\Bigr)}$, $I_m = \floor*{ \frac{1}{2}\log_3\Bigl(\frac{\log m}{8d}\Bigr) }$ and $l_{n,m} =  (\log m )\vee (\log n)^{2/(\alpha \wedge \beta)}$. Then there exist an event $\Upsilon$ with probability at least $1-12 (\log n)^{-1}$ and constants $C_i^\prime, C_i^\prime, \tilde{C}_i, \tilde{C}_i>0$ depends on $d, \gamma_1, \gamma_2, \sigma_1, \sigma_2, \alpha, \beta$ such that on $\Upsilon$, for all $b\in\mathcal{B}$, we have $\varphi^*_{b;n,m} - \varphi^*_{b;n,m} (0)\in \mathcal{C}_{M,U}$ and $\varphi_{b;n,m} - \varphi_{b;n,m}(0) \in \mathcal{C}_{R, T}$ where $M$ and $U$ are chosen as
    \begin{gather}\label{MandU}
    M_j = \begin{cases}
    C_0^\prime \ell_j, &\quad  0\leq j\leq J_n \\
    C_1^\prime l_{n,m}\ell_j, &\quad j > J_n
    \end{cases}, \quad 
    U_j = \begin{cases}
    \tilde{C}_0 \ell_j^{3}, &\quad  0\leq j\leq J_n \\
    \tilde{C}_1 l_{n,m}\ell_j^{3}, &\quad j > J_n
    \end{cases},  
    \end{gather}
    and $R$ and $T$ are chosen as  
        \begin{gather}\label{RandT}
        R_i = \begin{cases}
            C_2^\prime\ell_i, & \quad 0 \leq i \leq I_m \\
            C_3^\prime l_{n,m} \ell_i  ,& \quad i > I_m
        \end{cases}, \quad 
        T_i = \begin{cases}
            \tilde C_2  \ell_i^{3},  &\quad 0 \leq i \leq I_m \\
            \tilde C_3  l_{n,m} \ell_i^{3}, &\quad i > I_m
        \end{cases}.  
    \end{gather}
\end{proposition}
\begin{proof}
    Note Lemma~\ref{le: BoundedSubdiff} implies that in order to quantify the Lipschitz constant of $\varphi_{b;n,m}^*$ on $P_{j,k}$, we only need to bound the magnitude of $\sup\{\|y\|:y\in \partial \varphi_{b;n,m}^*(P_{j,k})\}$. To this end, we first note that $\partial \varphi_{b;n,m}^*(v) = \partial^c (\|\cdot\|^2/2 - \varphi_{b;n,m}^*)(v)$ and $\|\cdot\|^2/2 - \varphi_{b;n,m}^*$ is obviously a c-concave function. Thus by Lemma \ref{le: ExtProperties}(iv) and Lemma \ref{le: SWpropertyofEmpiricalPUandPAb}, we can apply \citet[Theorem 11]{manole2021sharp} to obtain\footnote{We remark that the bound given below uses the probability mass on $\mathcal{B}_{w,3}^d$ whereas the original formulation in \citet[Theorem~11]{manole2021sharp} has $\mathcal{B}_{w,1}^d$ instead. We have used a slightly different radius here for the convenience of the subsequent argument. The exact radius is unimportant in the argument used in that theorem and the same proof will work verbatim with radius changed to 3. } that there exists a constant $C_0 > 0$ depends on $d$ such that for any $v \in P_{j,k}$ and $y \in \partial\varphi_{b;n,m}^*(v)$, we have 
    \begin{equation}
    \label{Eq:ynorm}
     \|y\| \leq  C_0 (2dE_{1, m})^{1/2}\Bigl\{(\|v\| +1) \vee \sup\limits_{w:\|v - w\|\leq 2}\Bigl[\log \Bigl(\frac{1}{P_n^{A(b)}(\mathcal{B}^d_{w,3})}\Bigr)\Bigr]^{1/2}\Bigr\}.
    \end{equation}
    Thus to upper bound the magnitude of $\partial\varphi_{b;n,m}^*(v)$ we only need to prove an anticoncentration bound for $P^{A(b)}_n$.
    
    We first note that from \eqref{ineq: AntiConcentrationofvarepsilon}, for any $0 \leq j \leq J_n$, $v \in P_j$ and $w$ such that $\|w - v\|\leq 2$, we have
    \begin{align}
        P^{\varepsilon}(\mathcal{B}_{w, 2}^d) \geq \int_{\mathcal{B}_{w, 2}^d\setminus \mathcal{B}_{0}^d} \gamma_1 \exp(-\gamma_2\|e\|^2) \, d e &\geq \frac{\pi^{d/2}(2^d - 1)}{\Gamma(d/2 + 1)}\gamma_1\exp\big(-2\gamma_2\|z\|^2 -50\gamma_2\bigr) \notag  \\ 
        &\geq 2 K_1 \exp(-2 \gamma_2 \ell_j^2), \label{ineq: LowerboundVarepsilononball}
    \end{align}
    where $K_1 \in (0, 1)$ is a constant depending on $d, \gamma_1, \gamma_2$. Observe that the right-hand side does not depend on $z$ or $w$, hence, we may take infimum over $v\in P_j$ and $w$ such that $\|w-v\|\leq 2$ and have the same lower bound. Hence, we have
    \begin{align*}
        P^\varepsilon \otimes  P^S (\mathcal{B}_{w, 2}^d \times \mathcal{B}_{0}^p) = P^\varepsilon(\mathcal{B}_{w, 2}^d)P^S(\mathcal{B}_{0}^p) \geq 2K_1^\prime \exp(-2 \gamma_2\ell_j^2),
    \end{align*}
    for some $K_1^\prime \in (0, 1)$ depends on $d, \gamma_1, \gamma_2, \sigma_1$ and $\alpha$, where the sub-Weibull assumption on $S$ has been exploited in the final inequality.

    On the other hand, let $\mathcal{B}^d:=\{\mathcal{B}^d_{a,r}: a\in\mathbb{R}^d, r > 0\}$ be the set of all balls in $\mathbb{R}^d$. Let $\tilde u=\sqrt\frac{160d\log n}{n}$ and define 
    \[
        \Upsilon_{1} := \Bigl\{ \sup_{B\in \mathcal{B}^d}|P_n^\varepsilon \otimes  P_n^S (B \times \mathcal{B}_{0,}^p) - P^\varepsilon \otimes  P^S (B \times \mathcal{B}_{0}^p) | < \tilde u \Bigr\}.
    \]
    Thus, since $\tilde u \lesssim K_1^\prime n^{-1/8} \leq K_1^\prime e^{-2 \gamma_2\ell_j^2}$ for $0 \leq j \leq J_n$, working on $\Upsilon_{1}$ we have $P_n^\varepsilon \otimes  P_n^S (\mathcal{B}_{w, 2}^d \times \mathcal{B}_{0}^d) \geq K_1^\prime \exp(-2 \gamma_2 \ell_j^2)$. Thus consider the event  
    \begin{align}
       \Upsilon_{2}:= \bigcap_{j = 0}^{J_{n}}\Bigl\{\inf_{v \in P_j}\inf_{w: \|v - w\|\leq 2} P_n^\varepsilon \otimes  P_n^S (\mathcal{B}_{w, 2}^d \times \mathcal{B}_{0}^p) \geq K_1^\prime\exp\bigl(-2\gamma_2\ell_j^2\bigr)\Bigr\} \notag,
    \end{align}
    we have $\Upsilon_{1} \subset \Upsilon_{2}$. Note the Vapnik--Chervonenkis (VC) dimension of $\mathcal{B}^d$ is no more than $d + 2$ \citep[See e.g.][Corollary 13.2]{devroye2013probabilistic}, by the VC-inequality {\citep[see][Theorem 2]{vapnikuniform}} we have
    \begin{align}
         \Prob(\Upsilon_{1}^c) \lesssim   n^{d+2}\exp(-n \tilde{u}^2/32) \leq n^{2-4d} \leq n^{-2}, \label{ineq: BoundUpsilon1n}
    \end{align}
    whence $\Prob(\Upsilon_{2}) \geq 1 - n^{-2}$.
    Thus working on $\Upsilon_{2} \cap \Gamma_1$, by $\|A(b)_i\| \leq \|S_i\| + \|\varepsilon_i\|$ for all $b \in \mathcal{B}$ and $i \in [n]$, we have $P_n^{A(b)}(\mathcal{B}_{w, 3}^d) \geq K_1^\prime \exp(-2 \gamma_2 \ell_j^2)$, and combining this with \eqref{Eq:ynorm}, we conclude that for any $1 \leq k \leq N$ and $0 \leq j \leq J_n$, there exits some sufficiently large constant $C_0^\prime>0$ depends on $d, \gamma_1, \gamma_2, \sigma_1, \alpha$ such that 
    \begin{align}
        \sup_{y \in \partial \varphi_{b;n,m}^*(P_{j,k})}\|y\| &\leq C_0 (2d E_{1, m})^{1/2}\Bigl(\ell_j + 1 +  \sqrt{2}\ell_j\gamma_2^{1/2} + \sqrt{\log(1/K_1^\prime)}\Bigr) \notag \\ 
        & \leq C_0^\prime E_{1, m}^{1/2}\ell_j \leq C_0^\prime \Bigl(2 + \sqrt{\frac{\log m}{m}}\Bigr)^{1/2}\ell_j \lesssim C_0^\prime \ell_j : = M_j.
        \label{ineq: varphi^*j<=J_n}
    \end{align} 
    
    When $j > J_n$, by Lemma \ref{le: ExtProperties}(iii) and \citet[Proposition 16]{manole2021sharp}, we only need to bound $L_{b;n,m}$. Note $L_{b;n,m} \leq  L_{n,m}:=2\max_{i \in [n]}\|\Sigma^{-1/2}X_i\|^2 + 2\max_{i \in [n]}\|\varepsilon_i\|^2 + 2\max_{j \in [m]}\|U_j\|^2$. Define $r_{n,m}:= 2\sigma_1^2 (4 \log n)^{2/\alpha}+ 2\sigma_2^2 (4 \log n)^{2/\beta} +  8d \log m$ and consider the event $\Upsilon_{3}:= \{L_{n,m} < r_{n,m}\}$. By part(i) of Proposition \ref{Prop:SW} and union bound, it follows that
    \begin{align}
    \Prob\bigl(\Upsilon_{3}^c \bigr) &\leq \Prob\bigl(\max_{i \in [n]}\|\Sigma^{-1/2}X_i\|^2 \geq \sigma_1^2 (4 \log n)^{2/\alpha} \bigr) + \Prob\bigl(\max_{i \in [n]}\|\varepsilon_i\|^2 \geq \sigma_2^2 (4 \log n)^{2/\beta} \bigr) \notag \\
    &\quad + \Prob\bigl(\max_{j \in [m]}\|U_j\|^2 \geq 8d\log m \bigr) \leq 4n^{-1} + 2m^{-1} \label{ineq: HighProbabilityofPhinm}.
    \end{align}
    Therefore on the event $\Upsilon_{3}$, by \Citet[Proposition 16]{manole2021sharp} we have that there exists a universal constant $C_1>0$ and a sufficiently large $C_1^\prime>0$ depends on $\sigma_1, \sigma_2$ such that for any $1\leq k \leq N$,
    \begin{align}
    \sup_{y \in \partial \varphi_{b;n,m}^*(P_{j,k})}\|y\|  \leq C_1(\ell_j + r_{n,m}) \leq C_1^\prime l_{n,m}\ell_j =: M_j, \quad \text{for all $j > J_n$}. \label{ineq: varphi^*j>J_n}
    \end{align}
    
    Putting \eqref{ineq: varphi^*j<=J_n} and \eqref{ineq: varphi^*j>J_n} together, for some constants $\tilde C_0, \tilde C_1 >0$ depend on $d, \gamma_1, \gamma_2, \sigma_1, \sigma_2, \alpha$, we have $\varphi^*_{b;n,m} -\varphi^*_{b;n,m}(0)\in \mathcal{C}_{M, U}$ on the event $\Upsilon^\prime := \Upsilon_{2}\cap \Gamma_1 \cap \Upsilon_{3}$, where $M = (M_j)_{j \geq 0}$ and $U = (U_j)_{j \geq 0}$ are chosen as 
    \begin{gather}
    M_j = \begin{cases}
    C_0^\prime \ell_j, &\quad  0\leq j\leq J_n \\
    C_1^\prime l_{n,m}\ell_j, &\quad j > J_n
    \end{cases}, \quad 
    U_j = \begin{cases}
    \tilde{C}_0 \ell_j^{3}, &\quad  0\leq j\leq J_n \\
    \tilde{C}_1 l_{n,m}\ell_j^{3}, &\quad j > J_n
    \end{cases}, \notag 
    \end{gather}
    as desired. 
    
    A similar argument can be applied to study the Lipschitz property of $\varphi_{b;n,m}$. Since $U\sim \mathcal{N}(0, I_d)$, for all $i \leq I_m$, $u \in P_i$ and all $w$ such that $\|w - u\|\leq 2$, we have 
    \[ P^U(\mathcal{B}_{w}^d) = \int_{\mathcal{B}_{w}^d} (2 \pi)^{-d/2} \exp(-\|y\|^2/2) \, d y \geq 2K_2e^{-\ell_i^2},  \] 
    where $K_2 \in (0, 1)$ is a constant depends only on $d$. Let $\tilde v = \sqrt{\frac{160 d \log m}{m}}$, and define 
    \begin{align*}
        \Upsilon_{4} : =\Bigl\{ \sup_{B \in \mathcal{B}^d}|P_m^U(B) - P^U(B)| < \tilde v \Bigr\},  \text{ and }
        \Upsilon_{5}:=\bigcap_{i = 0}^{I_m}\Bigl\{\inf_{u \in P_i}\inf_{w: \|u - w\|\leq 2}P_n^U(\mathcal{B}_{w}^d) \geq K_2 e^{-\ell_i^2} \Bigr\}.
    \end{align*}
     Then since $\tilde v \leq m^{-1/8} \leq K_2 e^{-\ell_{I_m}^2}$ we have $\Upsilon_{4} \subset \Upsilon_{5}$. Furthermore, by leveraging the VC-inequality again, we can deduce that $\Prob(\Upsilon_{4}^c)  \leq m^{-2}$, which implies that $\Prob(\Upsilon_{5}) \geq 1 - m^{-2}$. On the event $\Upsilon_{5} \cap \Gamma_2$, by applying \citet[Theorem 11]{manole2021sharp} and Lemma \ref{le: BoundedSubdiff} again we obtain that for $0 \leq i \leq I_m$, there exists constants $C_2>0$ depends on $d, \alpha, \beta$ and $ C_2^\prime>0$ depends on $d, \sigma, \alpha, \beta$ such that 
    \begin{align}
        \sup_{z \in \partial \varphi_{b;n,m}(u)}\|z\| &\leq  C_2\sigma (2E_{2, n})^{1/(\alpha \wedge \beta)} (2\ell_i + 1 + \sqrt{\log (1/K_2)}) \notag \\
        &\leq C_2^\prime E_{2, n}^{1/(\alpha \wedge \beta)}\ell_i \leq C_2^\prime \Bigl(2 + \sqrt{\frac{\log n}{n}}\Bigr)^{1/(\alpha \wedge \beta)} \ell_i \lesssim C_2^\prime \ell_i:= R_i. \label{ineq: LipConstantforileqI_m}
    \end{align}
    
    When $i > I_m$, since we still have $|\|u\|^2/2 - \varphi_{b;n,m}|  \leq L_{b;n,m} \leq L_{n,m} $ by Lemma \ref{le: ExtProperties}(iii), working on the event $\Upsilon_{3}$, there exists an absolute constant $C_3 >0$, and $C_3^\prime>0$ depends on $\sigma_1, \sigma_2$ such that for $1 \leq k \leq N$.
    \begin{align}
        \sup_{z \in \partial \varphi_{b;n,m}(P_{i, k})}\|z\| \leq C_3(\ell_i + r_{n,m}) \leq C_3^\prime l_{n,m}\ell_i:= R_i \quad\text{for $i > I_m$}. \label{ineq: LipConstantfori>I_m} 
    \end{align}
    Thus combine \eqref{ineq: LipConstantforileqI_m} and \eqref{ineq: LipConstantfori>I_m} we can deduce that there exists constants $\tilde C_2, \tilde C_3>0$ depend on $d, \alpha, \beta, \sigma_1, \sigma_2$ such that $\varphi_{b;n,m}-\varphi_{b;n,m}(0) \in \mathcal{C}_{R, T}$ on the event $\Upsilon = \Upsilon^\prime \cap \Upsilon_{5} \cap \Gamma_2$, where 
    \begin{gather}
        R_i = \begin{cases}
            C_2^\prime\ell_i, & \quad 0 \leq i \leq I_m \\
            C_3^\prime l_{n,m} \ell_i  ,& \quad i > I_m
        \end{cases}, \quad 
        T_i = \begin{cases}
            \tilde C_2  \ell_i^{3},  &\quad 0 \leq i \leq I_m \\
            \tilde C_3  l_{n,m} \ell_i^{3}, &\quad i > I_m
        \end{cases}. \notag 
    \end{gather}
    Finally, combining the controls on the probability of $\Upsilon_2, \Upsilon_3, \Upsilon_5, \Gamma_1$ and $\Gamma_2$, we arrive at the the upper bound $\Prob(\Upsilon) \geq 1 - n^{-2} - 4n^{-1} - 2 m^{-1} - m^{-2} -2(\log m)^{-1} - 2(\log n)^{-1} \geq 1 - 12(\log n)^{-1}$ when $m\geq n$, which completes the proof. 
\end{proof}

Now we are ready to introduce the core proposition in the proof. 

\begin{proposition}
\label{Prop:W2distdiff}
There exists a constant $C>0$ depending on $d, \alpha, \beta, \sigma_1, \sigma_2, \gamma_1, \gamma_2$ such that for any fixed $n,m \geq 1$,  the following inequality holds 
\begin{align}
\sup_{b\in\mathcal{B}} \bigl|\mathcal{W}_2^2(P^{A(b)},P^U) - \mathcal{W}_2^2(P_n^{A(b)},P_m^U)\bigr| \leq C (\log m)^{\frac{8}{2\wedge \alpha\wedge\beta}} \biggl(\sqrt\frac{p}{n}+\frac{1}{n^{2/d}}\biggr)\label{ineq: ControlonEmpiricalWassersteinDistance}
\end{align} 
with probability at least $1-29(\log n)^{-1}$.
\end{proposition}

\begin{proof}
    Note part (i) and part (iii) of Lemma \ref{le: ExtProperties} implies that $(\|\cdot\|^2 - \varphi^*_{b;n,m}, \|\cdot\|^2 - \varphi_{b;n,m})$ is a feasible pair to the duality of the Kantorovich problem between $P^{A(b)}$ and $P^U$. This yields that  
    \begin{align*}
        \frac{1}{2}\mathcal{W}_2^2(P^{A(b)}, P^U) &\geq \int \biggl(\frac{\|v\|^2}{2} - \varphi_{b;n,m}^* (v)\biggr) dP^{A(b)}(v) + \int \biggl(\frac{\|u\|^2}{2} - \varphi_{b;n,m} (u)\biggr) dP^{U}(u) \\
        &= \int \frac{\|v\|^2}{2} dP_n^{A(b)}(v) + \int \frac{\|u\|^2}{2} dP_m^{U}(u) \\
        &\quad + \int \frac{\|v\|^2}{2} (dP^{(A(b)} - dP_n^{A(b)})(v)  +  \int \frac{\|u\|^2}{2} (dP^U - dP_m^{U})(u) \\
        & \quad - \Bigl\{\int \varphi_{b;n,m}^* (v) dP_n^{A(b)}(v) + \varphi_{b;n,m} (u) dP_m^U(u)\Bigr\} \\
        & \quad - \Bigl\{\int \varphi_{b;n,m}^* (v) d \bigl(P^{A(b)} - P_n^{A(b)}\bigr)(v) + \int \varphi_{b;n,m} (u) d \bigl(P^{U} - P_m^U\bigr)(u)\Bigr\}.
    \end{align*}
       By the definition of $(\varphi_{b;n,m}, \varphi_{b;n,m}^*)$, we have $\mathcal{W}_2^2(P_n^{A(b)}, P_m^U) = \int \bigl(\frac{\|v\|^2}{2} - \varphi_{b;n,m}^* (v) \bigr) dP_n^{A(b)}(v) + \int \bigl( \frac{\|u\|^2}{2} -  \varphi_{b;n,m} (u) \bigr) dP_m^{U}(u) $. Consequently, from the above display, we deduce that 
    \begin{align}
       \frac{1}{2}\mathcal{W}_2^2(P_n^{A(b)}, P_m^U)& - \frac{1}{2}\mathcal{W}_2^2(P^{A(b)}, P^U)\nonumber \\
       &\leq \underbrace{\int \varphi_{b;n,m}^* (v) d \bigl(P^{A(b)} - P_n^{A(b)}\bigr)(v) + \int \varphi_{b;n,m} (u) d \bigl(P^{U} - P_m^U\bigr)(u)}_{=: E_{b;n,m}} \notag \\
       & \qquad + \underbrace{\int \frac{\|v\|^2}{2} d \bigl(P_n^{A(b)} - P^{A(b)}\bigr)(v) + \int \frac{\|u\|^2}{2} d \bigl(P_m^{U} - P^U\bigr)(u)}_{=:F_{b;n,m}}. \label{Eq:EmpiricalPopulationWassersteinDistUpper} 
    \end{align}

    On the other hand, define $\Psi_b := \{(f, g) \in L^1(P^{A(b)}) \times L^1(P^U): v^T u \leq f(v) + g(u), \ \forall (v, u) \in \mathrm{Supp}(P^{A(b)}) \times \mathrm{Supp}(P^U) \}$, then Theorem \ref{thm: K-SOptimalityCriterion} implies that for any $b \in \mathcal{B}$ there exists a conjugate pair $(\psi_{b}^*, \psi_b)$ such that 
    \begin{align*}
        (\psi_{b}^*, \psi_{b}) &= \argmin_{f, g \in \Psi_b} \int f \, dP^{A(b)} + \int g \, d P^U, \\
        \frac{1}{2}\mathcal{W}_2^2(P^{A(b)}, P^U) &= \int \|v\|^2/2 - \psi_b^*(v) \, dP^{A(b)}(v) + \int \|u\|^2/2 - \psi_b(u) \, d P^U(u). 
    \end{align*}
    Since $\Psi_b \subseteq \tilde \Phi_b$, $(\|v\|^2/2 - \psi_b^*(v), \|u\|^2/2 - \psi_b(u))$ is a feasible solution for the duality between $P_n^{A(b)}$ and $P_m^U$. Therefore, we can rerun the previous derivation and obtain that 
    \begin{align}
        \frac{1}{2}\mathcal{W}_2^2(P_n^{A(b)}, P_m^U) &- \frac{1}{2}\mathcal{W}_2^2(P^{A(b)}, P^U)\nonumber \\
        &\geq \underbrace{\int \psi_{b}^* (v) d \bigl(P^{A(b)} - P_n^{A(b)}\bigr)(v) + \int \psi_{b} (u) d \bigl(P^{U} - P_m^U\bigr)(u)}_{=: G_{b;n,m} } \notag  \\
       & \qquad + \int \frac{\|v\|^2}{2} d \bigl(P_n^{A(b)} - P^{A(b)}\bigr)(v) + \int \frac{\|u\|^2}{2} d \bigl(P_m^{U} - P^U\bigr)(u). \label{Eq:EmpiricalPopulationWassersteinDistLower} 
     \end{align} 
     Write the first two terms and the last two terms of \eqref{Eq:EmpiricalPopulationWassersteinDistUpper} as $E_{n,m}$ and $F_{n,m}$ respectively, and write the first two terms of \eqref{Eq:EmpiricalPopulationWassersteinDistLower} as $G_{n,m}$. Then combining \eqref{Eq:EmpiricalPopulationWassersteinDistUpper} and \eqref{Eq:EmpiricalPopulationWassersteinDistLower}, for $\vartheta_k$ defined in~\eqref{ineq: vartheta_k}, we have 
     \begin{align}
     \sup_{b \in \mathcal{B}}\Bigl|\frac{1}{2}\mathcal{W}_2^2(P_n^{A(b)}, P_m^U) - \frac{1}{2}\mathcal{W}_2^2&(P^{A(b)}, P^U)\Bigr| \leq \sup_{b \in \mathcal{B}}|E_{b;n,m}| + 2\sup_{b \in \mathcal{B}} |F_{b;n,m}| + \sup_{b \in \mathcal{B}}|G_{b;n,m}| \notag \\
     &\lesssim  (\log m)^{\frac{6}{2\wedge \alpha\wedge\beta}}\bigl(\vartheta_n + \sqrt{\frac{p}{n}} + \sqrt{\frac{\log n}{n}} + \vartheta_m + \sqrt{\frac{\log m}{m}}\bigr),\notag \\
     &\leq (\log m)^{\frac{8}{2\wedge\alpha\wedge \beta}} \biggl(\sqrt\frac{p}{n}+\frac{1}{n^{2/d}}\biggr).
     \end{align}
     with probability at least $1 - 29(\log n)^{-1}$, where we have used Lemmas~\ref{le: Controlvarphi*}, \ref{le: ULLN} and \ref{le: BoundforGbnm} to bound each of the three terms in the penultimate inequality. 
\end{proof}

\begin{lemma}\label{le: Controlvarphi*}
There exists $C>0$, depending only on $d, \alpha, \beta, \gamma_2, \sigma_1, \sigma_2$, and an event $\Omega$ with probability at least $1 - 18(\log n)^{-1}$, such that on $\Omega$, for any $b \in \mathcal{B}$,  we have 
    \begin{gather*}
        \biggl|\int \varphi_{b;n,m}^* (v) \, d \bigl(P^{A(b)} - P_n^{A(b)}\bigr)(v)\biggr| \leq C(\log m)^{\frac{6}{2 \wedge \alpha \wedge \beta}}\biggr( {\vartheta}_n + \sqrt\frac{p}{n} + \sqrt{\frac{2 \log n}{n}}\biggr)\\
        \biggl|\int \varphi_{b;n,m} (u) \, d \bigl(P^{U} - P_m^U\bigr)(u) \Bigr|\leq C (\log m)^{\frac{6}{2 \wedge \alpha \wedge \beta}} \biggl(\vartheta_m + \sqrt{\frac{2 \log m}{m}} \biggr),
    \end{gather*}
    where $\vartheta_n$ is defined as \eqref{ineq: vartheta_k}.
\end{lemma}
\begin{proof}
    We note that the value of the integrals on the left-hand side of both inequalities will not change if we add any constant to the functions $\phi^*_{b;n,m}$ and $\phi_{b;n,m}$. Hence, we may assume without loss of generality throughout this proof that $\phi^*_{b;n,m}(0) = \phi_{b;n,m}(0) = 0$. 
    
    % Recalling that $A(b)$ is a $\sw{\sigma}{\alpha \wedge \beta}$ random vector for some $\sigma>0$ only depends on $\sigma_1$ and $\sigma_2$. Moreover, d

    Note that due to the sub-Weibull assumptions on $\varepsilon$ and $S$, and combining with Proposition \ref{prop: swsum}(ii), we have $(\varepsilon, S) \sim \sw{\rho}{\alpha \wedge \beta}$ for $\rho>0$ depending only on $\sigma_1$ and $\sigma_2$. Then let $\kappa =  \rho (4\log n)^{1/(\alpha \wedge \beta)}$ and $\Omega_{1} := \{\max_{1 \leq i \leq n}\|(\varepsilon_i, S_i)\| \leq \kappa\}$, and by Proposition \ref{prop: swequivalent}(i), we have
    \begin{gather*}
        \Prob(\Omega_{1}^c) \leq n \Prob(\|(\varepsilon, S)\| \geq \kappa) \leq 2n \exp\biggl\{-\frac{1}{2}(\kappa/\rho)^{\alpha \wedge \beta}\biggr\} \leq  \frac{2}{n}. 
    \end{gather*} 

    For any $b \in \mathcal{B}$, define the linear projection $T_b:  \mathbb{R}^p\times \mathbb{R}^{d}  \to \mathbb{R}^d$ such that 
    \begin{align}
        T_b(s, e) := (b^* - b)\Sigma^{1/2}s + e. \label{Map: T_b}
    \end{align}
    Write $E_b = \{T_b(s, e) \in \mathbb{R}^d: (s, e) \in \mathcal{B}_{0, \kappa}^{d + p}\}$. Working on the event $\Omega_1$ and observing that $\|T_b\|_{\ope} \leq 1$ for any $b \in \mathcal{B}$, we have
       \begin{align}
        \int_{\mathbb{R}^d \setminus E_b} \varphi_{b;n,m}^* (v) &\,d \bigl(P^{A(b)} - P_n^{A(b)}\bigr)(v)   =\int_{\mathbb{R}^{d+p} \setminus \mathcal{B}_{0, \kappa}^{d+p}} \varphi_{b;n,m}^* \circ T_b(e,s) \,  d P^{\varepsilon}\otimes P^S(e,s) \notag \\
        &\stackrel{(a)}{\leq} \int_{\mathbb{R}^{d+p} \setminus \mathcal{B}_{0, \kappa}^{d+p}} \biggl(\frac{\|T_b(e,s)\|^2}{2} + r_{n,m}\biggr) \, d (P^\varepsilon\otimes P^S)(e,s) \notag \\
        &\leq \int_{\mathbb{R}^{d+p} \setminus \mathcal{B}_{0, \kappa}^{d+p}} \biggl(\frac{\|(e,s)\|^2}{2} + r_{n,m}\biggr) \, d (P^\varepsilon\otimes P^S)(e,s) \notag \\
        &\stackrel{(b)}{\leq} C_{4} e^{-\frac{1}{4}(\frac{\kappa}{\rho})^{\alpha \wedge \beta}} + \frac{2r_{n,m}}{n^2} \lesssim \frac{C_4}{n}, \label{ineq: varphiRestricted}
    \end{align}
    where we use part \ref{suble: Int} of Proposition \ref{le: ExtProperties} to obtain (a) and Lemma \ref{Lem:TruncatedMoment} to obtain (b) and $C_4>0$ is a constant only depending on $d, \sigma_1, \sigma_2, \alpha, \beta$. %Thus we can now focus on $\varphi^*_{b;n,m}$ restricts on $E_b$ with only a loss of order $n^{-1}$. 

    % Working on $\Omega_{1}$, there exists a constant $C_4>0$, depending only on $\sigma, \alpha, \beta$ , such that  
    % \begin{align}
    %     \int_{\mathbb{R}^d \setminus \mathcal{B}_{0, \kappa}^d} \varphi_{b;n,m}^* (v) &\,d \bigl(P^{A(b)} - P_n^{A(b)}\bigr)(v)   =\int_{\mathbb{R}^d \setminus \mathcal{B}_{0, \kappa_{n, m}}^d} \varphi_{b;n,m}^* (v)\,  d P^{A(b)}(v) \notag \\
    %     \stackrel{(a)}{\leq}& \int_{\mathbb{R}^d \setminus \mathcal{B}_{0, \kappa_{n, m}}^d} \biggl(\frac{\|v\|^2}{2} + r_{n,m}\biggr) \, d P^{A(b)}(v) 
    %     \stackrel{(b)}{\leq} C_{4} e^{-\frac{1}{4}(\frac{\kappa_{n, m}}{\sigma})^{\alpha \wedge \beta}} + \frac{2r_{n,m}}{n^2} \lesssim \frac{C_4}{n}, \label{ineq: varphiRestricted}
    % \end{align}
    % where we use part \ref{suble: Int} of Proposition \ref{le: ExtProperties} to obtain (a) and Lemma \ref{Lem:TruncatedMoment} to obtain (b). 

    On the other hand, for $\mathcal{X} \subseteq \mathbb{R}^d$, we define $\mathrm{Lip}_{1, 1}(\mathcal{X}) := \{f \in \mathrm{Lip}_1(\mathcal{X}): \sup_{x\in\mathcal{X}} |f(x)| \leq 1\}$ to be the class of $1$-Lipschitz functions on $\mathcal{X}$ uniformly bounded by $1$. Consider the following function class
    \begin{align}
          \mathcal{F} := \Bigl\{(s,e)\mapsto (\varphi \circ T_b)(s,e) \one_{\mathcal{B}_{0}^{d + p}}(s,e) : b \in \mathcal{B}, \ \varphi \in \mathrm{Lip}_{1, 1}(\mathcal{B}_{0}^d) \Bigr\}. \label{eq: Fb}
     \end{align}

    % $j_{n,m} := (J_n + 1) + (I_m + 1) + \lceil \log_3 (\rho(4 \log n)^{1/(\alpha \wedge \beta)}/ d^{1/2})\rceil$
    
    Let $j_{n} = (J_n + 1) + \lceil \log_3 (\rho(4 \log n)^{1/(\alpha \wedge \beta)}/ d^{1/2})\rceil$. Then we have $3^{j_n} \sqrt{d} \geq \kappa$, which implies that $\mathcal{B}_{0, \kappa}^d \subseteq \bigcup_{j =0}^{j_{n}} \bigcup_{k = 1}^N P_{j, k}$ for $P_{j,k}$ defined before Lemma~\ref{le: SWpropertyofEmpiricalPUandPAb}. Let $\Upsilon$ be the event with probability $1-12(\log n)^{-1}$ on which Proposition~\ref{prop: PiecewiseLipvarphi} holds. Then, from Proposition \ref{prop: PiecewiseLipvarphi}, we have $ \varphi^*_{b;n,m}|_{\mathcal{B}^d_{0, \kappa}}$ is Lipschitz continuous with parameter $M_{j_n}$ and upper bound $U_{j_{n}}$, for $M_j$ and $U_j$ are defined in \eqref{MandU}. Specifically, since $j_n > J_n$, from \eqref{MandU}, there exists $C_5 >0$, depending only on $d, \alpha, \beta, \sigma_1, \sigma_2, \gamma_2$, such that $M_{j_{n}} \vee U_{j_{n}} \leq C_5 (\log m)^{\frac{5}{2 \wedge \alpha \wedge \beta}}$. Whence, observing that 
    % we have $(C_5 (\log m)^{\frac{5}{2 \wedge \alpha \wedge \beta}})^{-1}(\varphi^*_{b;n,m} - \varphi^*_{b;n,m}(0))|_{\mathcal{B}^d_{0, \kappa_{n, m}}} \in \mathrm{Lip}_{1, 1}(\mathcal{B}_{0, \kappa}^d)$. Define linear function $\omega: \mathcal{B}_{0, 1}^d \to \mathbb{R}^d$ as $\omega (v) = \kappa v$ for any $v \in \mathbb{R}^d$. Then 
   \[
      \frac{\varphi^*_{b;n,m}\bigl(\kappa T_b(\cdot, \cdot)\bigr)\one_{\mathcal{B}_{0}^{d + p}}(\cdot, \cdot)}{C_5 (\log m)^{\frac{5}{2 \wedge \alpha \wedge \beta}}}  \in \mathcal{F} ,
   \]
   we deduce that 
       \begin{align}
        \int_{E_b}\frac{\varphi_{b;n,m}^*(v)}{C_5 \kappa (\log m)^{\frac{5}{2 \wedge \alpha \wedge \beta}} }&\, d \bigl(P^{A(b)} - P_n^{A(b)}\bigr)(v)  \notag \\
        & =\int_{\mathcal{B}_{0, \kappa}^{d+p}} \frac{\varphi_{b;n,m}^*( T_b(s,e))}{C_5 \kappa(\log m)^{\frac{5}{2 \wedge \alpha \wedge \beta}} }\, d( P^\varepsilon\otimes P^S - P^\varepsilon_n \otimes P^S_n)(e, s) \notag    \\
        & = \int_{\mathcal{B}_{0, 1}^{d+p}} \frac{\varphi_{b;n,m}^* (\kappa T_b(s,e))}{C_5 (\log m)^{\frac{5}{2 \wedge \alpha \wedge \beta}} }\, d( P^\varepsilon\otimes P^S - P^\varepsilon_n \otimes P^S_n)(e, s) \notag    \\
        & \leq \sup_{f \in \mathcal{F}}\Bigl\{ \int f(s, e) \, d( P^\varepsilon \otimes P^S  - P_n^\varepsilon  \otimes P_n^S)(e, s) \Bigr\}. \label{ineq: BoundonIntofvarphi*} 
    \end{align} 

    % \begin{align}
    %     \int_{\mathcal{B}_{0, \kappa}^d}\frac{\varphi_{b;n,m}^*(v)}{C_5 \kappa(\log m)^{\frac{5}{2 \wedge \alpha \wedge \beta}} }&\, d \bigl(P^{A(b)} - P_n^{A(b)}\bigr)(v)  \notag \\
    %     & =\int_{E_b} \frac{\varphi_{b;n,m}^*(v)}{C_5 \kappa(\log m)^{\frac{5}{2 \wedge \alpha \wedge \beta}} }\circ T_b(s,e)\, d( P^\varepsilon\otimes P^S - P^\varepsilon_n \otimes P^S_n)(e, s) \notag    \\
    %     & \leq \sup_{f \in \mathcal{F}}\Bigl\{ \int f(s, e) \, d( P^\varepsilon \otimes P^S  - P_n^\varepsilon  \otimes P_n^S)(e, s) \Bigr\}. \label{ineq: BoundonIntofvarphi*} 
    % \end{align} 
    By Lemma \ref{le: RademacherComplexityofR_n} and \citet[Theorem 4.10]{wainwright2019high},  there exists an event $\Omega_2$ with probability at least $1 - n^{-1}$, on which for some constant $C'>0$, depending only on $d$, we have 
    \begin{equation}
    \label{Eq:Omega2}
    \sup_{f \in \mathcal{F}}\Bigl| \int f \, d (P^\varepsilon \otimes P^S  - P_n^\varepsilon  \otimes P_n^S) \Bigr| \leq 2C'\Bigl(\vartheta_n + \sqrt{\frac{p}{n}}\Bigr)+ \sqrt{\frac{2\log n}{n}}.
    \end{equation}
    Combining~\eqref{ineq: varphiRestricted}, \eqref{ineq: BoundonIntofvarphi*} and~\eqref{Eq:Omega2}, we have on event $\Upsilon\cap\Omega_1 \cap \Omega_2 $ that
    \begin{align*}
        \int \varphi_{b;n,m}^*(v) \, d \bigl(P^{A(b)} - P_n^{A(b)}\bigr)(v) &\leq C_5\kappa (\log m)^{\frac{5}{2 \wedge \alpha \wedge \beta}}\biggl(2C'\Bigl(\vartheta_n + \sqrt{\frac{p}{n}}\Bigr)+ \sqrt{\frac{2\log n}{n}}\biggr)+ \frac{C_{4}}{n} \\
        &\leq C_5^\prime(\log m)^{\frac{6}{2 \wedge \alpha \wedge \beta}}\biggr( {\vartheta}_n + \sqrt\frac{p}{n} + \sqrt{\frac{2 \log n}{n}}\biggr), 
    \end{align*}
    for some $C_5^\prime >0$ depending only on $d, \alpha, \beta, \sigma_1, \sigma_2, \gamma_2$. A symmetric argument shows that on $\Upsilon\cap\Omega_1\cap\Omega_2$,  $\int -\varphi^*_{b;n,m}d(P^{A(b)}-P_n^{A(b)})$ can be controlled by the same upper bound. This establishes the first claim of the lemma. 

    A similar argument is applied to obtain the bound for the empirical process of $\varphi_{b;n,m}$. Let $\gamma = 2\sqrt{2d\log m}$, and define $\Omega_{3}:= \{\max_{1 \leq i \leq m} \|U_i\| \leq \gamma\}$. Then by a union bound we have $\Prob(\Omega_{3}^c) \leq m \Prob\bigl(\|U_1\| \geq \gamma\bigr) \leq 2m \exp(-\frac{1}{2}\frac{\gamma^2}{2d}) \leq \frac{2}{m}$. Working on $\Omega_{3}$ we deduce that for some absolute constant $C_6>0$,

    \begin{align}
        \int_{\mathbb{R}^d \setminus \mathcal{B}_{0, \gamma}^d} \varphi_{b;n,m}(u)\, d(P^U - P_m^U)(u) &= \int_{\mathbb{R}^d \setminus \mathcal{B}_{0, \gamma}^d} \varphi_{b;n,m}(u) \, dP^U(u) \notag \\
        &\stackrel{(c)}{\leq} \int_{\mathbb{R}^d \setminus \mathcal{B}_{0, \gamma}^d} \frac{\|u\|^2}{2} \, dP^U(u)  \stackrel{(d)}{\leq} \frac{C_6}{m^2}. \label{ineq: controlvarphiint}
    \end{align}
    In the above, we use part \ref{suble: Int} in the Proposition \ref{le: ExtProperties} to obtain (c) and Lemma \ref{Lem:TruncatedMoment} in inequality (d).
    
    Define 
    \begin{align}
        \mathcal{H} = \{g \one_{\mathcal{B}_{0}^d}: g \in \mathrm{Lip}_{1, 1}(\mathcal{B}_{0}^d)\}. \label{eq: mathcalH}
    \end{align}
    Let $i_{m} := (I_m + 1) + \ceil*{ \frac{1}{2}\log_3( 8 d \log m ) }$. Observe that $3^{i_{m}} \sqrt{d} \geq \gamma$ thus we have $\mathcal{B}^d_{0, \gamma} \subset \bigcup_{i = 0}^{i_{m}} \bigcup_{k=1}^N P_{i,k}$. Since $\varphi_{b;n,m} \in \mathcal{C}_{R, T}$ on $\Upsilon$ according to Proposition \ref{prop: PiecewiseLipvarphi}, we have that $\varphi_{b;n,m} |_{\mathcal{B}_{0, \gamma}^d}$ is bounded and Lipschitz continuous with upper bound $T_{i_{m}}$ and Lipshictz constant $R_{i_{m}}$ as defined in \eqref{RandT}. Moreover, by the explicit display of \eqref{RandT}, there exists a constant $C_{7}$ depends on $d, \sigma_1, \sigma_2, \alpha, \beta, \gamma_2$ such that $R_{i_{m}} \vee T_{i_{m}} \leq C_7 (\log m)^{\frac{5}{2 \wedge \alpha \wedge \beta}}$. Therefore, on $\Upsilon$, we have 
    \[
    \frac{\varphi_{b;n,m}(\langle \gamma, \cdot \rangle ) }{C_7(\log m)^{\frac{5}{2 \wedge \alpha \wedge \beta}}}\one_{\mathcal{B}_{0, 1}^d}(\cdot ) \in \mathcal{H},
     \]
     and consequently,
    \begin{align}
        \frac{1}{C_7\gamma (\log m)^{\frac{5}{2 \wedge \alpha \wedge \beta}}}\int_{\mathcal{B}_{0, \gamma}^d} \varphi_{b;n,m}(u)  \, d(P^U - P_m^U)(u) 
        &\leq \sup_{h \in \mathcal{H}}\Bigl\{\int h(u) \, d(P^U - P_m^U)(u)\Bigr\} \label{ineq: UpperboundIntvarphi}.
    \end{align}
    %     \begin{align}
    %     \int \varphi_{b;n,m}(u) \, d(P^U - P_m^U)(u) &= \int \varphi_{b;n,m}(u) - \varphi_{b;n,m}(0) \, d(P^U - P_m^U)(u)  \notag \\
    %     &= \int_{\mathcal{B}_{0, \gamma_{n,m}}^d} \varphi_{b;n,m}(u) - \varphi_{b;n,m}(0) \, d(P^U - P_m^U)(u) + \frac{C_6}{m} \notag \\
    %     &= \gamma_{n,m}\int_{\mathcal{B}_{0, 1}^d} \varphi_{b;n,m} \circ \zeta(u) d(P^U - P_m^U)(u) + \frac{C_6}{m} \notag \\
    %     &\leq C_7 (\log m)^{\frac{5}{2 \wedge \alpha \wedge \beta}} \gamma_{n,m}\sup_{h \in \mathcal{H}} \Bigl\{\int_{\mathcal{B}_{0, 1}^d} h \, d(P^U - P^U_m)(u) \Bigr\} + \frac{C_6}{m} \label{ineq: UpperboundIntvarphi}.
    % \end{align}
    Then applying Lemma \ref{le: RademacherComplexityofR_n} and \citet[Theorem 4.10]{wainwright2019high}, we derive that there exists an event $\Omega_4$ with probability at least $1 -m^{-1}$ such that on this event we have
    \begin{align} \label{Omega4}
    \sup_{h \in \mathcal{H}}\biggl|\int h\, d(P^U -P_m^U) \biggr| \leq 2\vartheta_m +\sqrt{\frac{2\log m}{m}}. 
    \end{align}
    Consequently, combining~\eqref{ineq: controlvarphiint}, \eqref{ineq: UpperboundIntvarphi} and \eqref{Omega4}, and working on the event $\Upsilon\cap \Omega_3 \cap \Omega_4$, we obtain
    \begin{align}
        \int \varphi_{b;n,m}(u) \, d(P^U - P^U_m) &\leq C_7 (\log m)^{\frac{5}{2 \wedge \alpha \wedge \beta}} \gamma\Bigl(2\vartheta_m + \sqrt{\frac{2 \log m}{m}}\Bigr) + \frac{C_6}{m},\notag \\
        &\leq C_7^\prime (\log m)^{\frac{6}{2 \wedge \alpha \wedge \beta}} \Bigl(\vartheta_m + \sqrt{\frac{2 \log m}{m}} \Bigr). \notag 
    \end{align}
    for some $C_7^\prime$ depends on $d, \sigma_1, \sigma_2, \alpha, \beta, \gamma_2$. A symmetric argument can be applied to establish the upper bound for $\int -\varphi_{b;n,m}(u) \, d(P^U - P^U_m)$ and the second claim follows. Finally, the proof is complete by observing that $\Prob(\Upsilon\cap  \Omega_1\cap\Omega_2\cap \Omega_3\cap\Omega_4)\geq 1 - 18({\log n})^{-1}$. 
    \end{proof}
    
\begin{lemma}\label{le: RademacherComplexityofR_n}
Suppose that $\mathcal{T}$ be a subset of linear maps from $\mathbb{R}^p$ to $\mathbb{R}^d$ whose operator norms are bounded by 1 and let $\mathcal{L}$ be a subset of $\{g: \text{$g\in\mathrm{Lip}_{1,1}(\mathcal{B}_{0}^d) $, $g(0)=0$ and $g$ is convex}\}$.  Define $\mathcal{F}:= \{(g\circ h)\mathbbm{1}_{\mathcal{B}_{0}^p}: h\in\mathcal{T}, g\in \mathcal{L}\}$. Let $P \in \mathcal{P}_2(\mathbb{R}^p)$. Then exists $C>0$, depending only on $d$, such that
\[
\mathcal{R}_n(\mathcal{F}, P) 
%\leq C M \biggl(\vartheta_n  + \int_0^2 N(\delta, \mathcal{T}, \|\cdot\|_{\mathrm{op}}) \,d\delta \biggr) 
\leq C\biggl(\vartheta_n + \sqrt\frac{p}{n}\biggr),
\]
where 
    \begin{align}
        {\vartheta}_k:= \begin{cases}
            k^{-2/d},  \quad &\text{if $d \geq 5$,}  \\
            k^{-1/2}\log k, \quad  &\text{if $d = 4$,} \\
            k^{-1/2}, \quad &\text{if $d \leq 3$.}
        \end{cases} \label{ineq: vartheta_k} 
    \end{align}
for $k\in\mathbb{N}$.
\end{lemma}
\begin{proof}
For any fixed $\delta\in(0,1)$, let $\mathcal{G}$ be a $\delta$-covering set of $\mathcal{L}$ with respect to $\|\cdot\|_{L^\infty(\mathcal{B}_{0}^d)}$. By \citet[Remark 1 and Theorem 6]{bronshtein1976varepsilon}, we have $N_0 := |\mathcal{G}| \leq e^{C_8 (4/\delta)^{d/2}}$ for some $C_8 > 0$, depending only on $d$. Similarly, let $\mathcal{H}$ be a $\delta$-covering set of $\mathcal{T}$ with respect to $\|\cdot\|_{\mathrm{op}}$. By \citet[Lemma~5.7]{wainwright2019high}, we have $N_1:=|\mathcal{H}| \leq (1+2/\delta)^{dp}$. Now, given any $f = g\circ h \in \mathcal{F}$, we can find $g'\in\mathcal{G}$ and $h'\in\mathcal{H}$ such that $\|g'-g\|_{L^\infty(\mathcal{B}_{0}^d)} \leq \delta$ and $\|h'-h\|_{\mathrm{op}} \leq \delta$. Consequently, for $X\sim P$, we have 
\begin{align*}
    \|(g \circ h &- g' \circ h')\one_{\mathcal{B}_0^p}\|_{L^2(P)} \leq \|(g \circ h - g' \circ h)\one_{\mathcal{B}_0^p}\|_{L^2(P)} + \|(g'\circ h  - g'\circ h')\one_{\mathcal{B}_0^p}\|_{L^2(P)} \\
    &= \Bigl\{\mathbb{E} \Bigl|(g-g')\circ h(X)\one_{\{\|X\|\leq 1\}}\Bigr|^2\Bigr\}^{1/2}  + \Bigl\{\mathbb{E} \Bigl|g'\circ (h-h')(X) \one_{\{\|X\|\leq 1\}}\Bigr|^2\Bigr\}^{1/2} \leq 2\delta.
    \end{align*}
    which implies 
    \begin{align}
        \log N(2\delta, \mathcal{F}, \|\cdot\|_{L^2(P)}) \leq \log (N_0 N_1) \leq C_8\biggl(\frac{4}{\delta}\biggr)^{d/2} \!+ dp\log\biggl(1+\frac{2}{\delta}\biggr) \leq C_8\biggl(\frac{4}{\delta}\biggr)^{d/2} \!+ \frac{2dp}{\delta}. \label{ineq: MetricEntropyforF} 
    \end{align}
    Since all functions in $\mathcal{F}$ are uniformly bounded by 1, the $L_2(P)$-diameter of $\mathcal{F}$ is bounded by $2$. Thus, by Dudley's chaining \Citep[see e.g.][Theorem 5.22]{wainwright2019high}, for any $\epsilon \in [0, 1]$, we have 
    \begin{align*}
         \mathcal{R}_n(\mathcal{F}, P) &\leq 2 \epsilon + \frac{32}{\sqrt{n}}\E \int_{\epsilon/4}^2 \log^{1/2} N(\delta, \mathcal{F}, \|\cdot\|_{L^2(P)}) \, d \delta\\
        & \leq 2 \epsilon +  \frac{2^{5+3d/4} C_8^{1/2} }{n^{1/2}}\int^{2}_{\epsilon/4} \frac{1}{\delta^{d/4}}   \, d \delta  + \frac{64(dp)^{1/2}}{n^{1/2}} \int_{\epsilon / 4}^2 \frac{1}{\delta^{1/2}} \, d \delta.
        % &\leq 2\epsilon \sqrt{n}+ {(64 - 8 \epsilon)} \E\sqrt{C(4(1 + \|S\|_{L^2(P_n^2)}))^{d/2} \epsilon^{-d/2} + dp \log(1 + \frac{8(1 + \|S\|_{L^2(P_n^S)})}{\epsilon}) + 1} 
    \end{align*}
    By choosing $\epsilon\asymp n^{-2/d}$ if $d\geq 4$ and $\epsilon = 0$ otherwise, we deduce from the previous inequality that there exists $C>0$ depending only on $d$ such that
    \begin{align*}
        \mathcal{R}_n(\mathcal{F}, P) \leq C \begin{cases}
            n^{-2/d} + (p/n)^{1/2},  \quad &\text{if $d \geq 5$}  \\
            n^{-1/2}(\log n + p^{1/2}), \quad  &\text{if $d = 4$} \\
            (p/n)^{1/2}, \quad &\text{if $d \leq 3$,}
        \end{cases},
    \end{align*}
    completing the proof.
\end{proof}
\begin{lemma}\label{le: BoundforGbnm}
There exists $C>0$ depending only on $d,\alpha,\beta,\sigma_1,\sigma_2,\gamma_1,\gamma_2$, such that with probability at least $1-6/n$, both of the following inequalities hold:
\begin{gather}
    \sup_{b \in \mathcal{B}}\biggl| \int \psi^*_b(v) \, d(P^{A(b)} - P_n^{A(b)})(v) \biggr|\notag \leq C (\log m)^{\frac{2}{2 \wedge \alpha \wedge \beta}}n^{-1/2} \\
    \sup_{b \in \mathcal{B}}\biggl| \int \psi_b(u) \, d(P^U - P_m^{U})(u) \biggr|\notag \leq C (\log m)^{\frac{2}{2 \wedge \alpha \wedge \beta}}m^{-1/2}.
\end{gather}
\end{lemma}
\begin{proof}
Since adding a constant to $\psi_b^*$ or $\psi_b$ will not change the value of $\int \psi_b^*(v) \, d(P^{A(b)} - P_n^{A(b)})(v)$ or $\int \psi_b(u) \, d(P^{U} - P_m^{U})(u)$, we assume $\psi_b^*(0) = \psi_b(0) = 0$ with out loss of generality. We first note that $\E\|(b^* - b)\Sigma^{1/2}S\|^2 = \|b^* -b\|_{\Sigma}^2 \leq 1$, for any $b \in \mathcal{B}$. By Lemma \ref{le: AnticoncentrationforSumofTwoRVs} and the anti-concentration inequality of $\varepsilon$ given in \eqref{ineq: AntiConcentrationofvarepsilon}, there exists a constant $M_1>0$ depends on $\gamma_1$ and $\gamma_2$ such that the density function of $A(b)$, write as $f_{A(b)}$, have the anti-concentration inequality
    \begin{align}
        f_{A(b)}(v) \geq M_1 \exp(-2\gamma_2 \|v\|^2), \quad \text{for all $\|v\| \geq 2$}. \notag 
    \end{align}
    Then by recalling that $P^U \sim \sw{\sqrt{2d}}{2}$, we apply \citet[Theorem 11]{manole2021sharp}\footnote{In the original Theorem 11 of \cite{manole2021sharp}, a regular condition is required on the density function of the source probability measure. Nevertheless, it is indeed sufficient to reestablish the result by merely assuming an anti-concentration inequality on the density function of the source probability measure, as we have proven for $f_{A(b)}$ here.} to obtain that $\|\nabla\psi_b^*(v)\| \leq C (\|v\| + 1)$ for all $v \in \mathbb{R}^d$, where $C>0$ is a constant depending on $d, \gamma_1, \gamma_2$. Therefore, applying mean value theorem, we have $|\psi_b^*(v)| \leq C(\|v\| + 1)^2$ for all $v \in \mathbb{R}^d$.
    
    Define $ \Omega_1 := \{\max_{1 \leq i \leq n} \|(\varepsilon_i, S_i)\| \leq \kappa\}$ and $E_b := \{T_b(s, e) : (s, e) \in \mathcal{B}_{0,\kappa}^{d+p}\}$ for each fixed $b \in \mathcal{B}$. From the proof of Lemma~\ref{le: Controlvarphi*}, we have $\mathbb{P}(\Omega_1)\geq 1-2/n$. Then on $\Omega_1$ we can obtain that 
    \begin{align}
        \sup_{b \in \mathcal{B}}\Bigl|\int_{\mathbb{R}^d \setminus E_b} \psi_b^*(v) \, d(P^{A(b)} - P_n^{A(b)})(v) \Bigr| &=  \sup_{b \in \mathcal{B}}\Bigl|\int_{\mathbb{R}^{d+p} \setminus \mathcal{B}_{0, \kappa}^{d +p}} \psi_b^*\circ T_b(s, e) \, d(P^{S} \otimes P^{\varepsilon})(s, e) \Bigr| \notag \\
        &\leq C\sup_{b \in \mathcal{B}}\int_{\mathbb{R}^{d+p} \setminus \mathcal{B}_{0, \kappa}^{d+p}} (1 + \|T_b(s, e)\|)^2 \, d(P^{S} \otimes P^{\varepsilon})(s, e) \notag \\
        &\leq C\sup_{b \in \mathcal{B}}\int_{\mathbb{R}^{d+p} \setminus \mathcal{B}_{0, \kappa}^{d+p}} (1 + \|(s, e)\|)^2 \, d(P^{S} \otimes P^{\varepsilon})(s, e) \notag \\
        &\leq \frac{C^\prime}{n},  \notag 
    \end{align}
    % \begin{align}
    %     \sup_{b \in \mathcal{B}}\Bigl|\int_{\mathbb{R}^d \setminus \mathcal{B}_{0, \kappa}} \psi_b^*(v) d(P^{A(b)} - P_n^{A(b)})(v) \Bigr| &=  \sup_{b \in \mathcal{B}}\Bigl|\int_{\mathbb{R}^d \setminus \mathcal{B}_{0, \kappa}} \psi_b^*(v) dP^{A(b)}(v) \Bigr| \notag \\
    %     &\leq C\sup_{b \in \mathcal{B}}\int_{\mathbb{R}^d \setminus \mathcal{B}_{0, \kappa}} (1 + \|v\|)^2 dP^{A(b)}(v) \leq \frac{C^\prime}{n},  \notag 
    % \end{align}
    for some constant $C^\prime>0$ depending on $d, \alpha, \beta, \sigma_1, \sigma_2, \gamma_1, \gamma_2$, where we used the fact that $P^{(S, \varepsilon)} \sim \sw{\rho}{\alpha \wedge \beta}$ and Lemma \ref{Lem:TruncatedMoment} in the final inequality. It therefore remains to control 
     \[
        G:= \sup_{b \in \mathcal{B}} \Bigl|\int_{E_b} \psi_b^*(v) \, d(P^{A(b)} - P_n^{A(b)})(v) \Bigr| = \sup_{b \in \mathcal{B}}\Bigl| \int_{\mathcal{B}_{0, \kappa}^{d+p}} \psi_b^*\circ T_b(s, e) \, d(P^{S}\otimes P^{\varepsilon} - P^{S}_n\otimes P^{\varepsilon}_n)(s, e)\Bigr|.
    \]
    % \begin{align}
    %     \sup_{b \in \mathcal{B}}\Bigl|\int \psi_b^*(v) d(P^{A(b)} - P_n^{A(b)})(v) \Bigr| \leq \sup_{b \in \mathcal{B}} \Bigl|\int_{\mathcal{B}_{0, \kappa}} \psi_b^*(v) d(P^{A(b)} - P_n^{A(b)})(v) \Bigr| + \frac{C^\prime}{n}. \label{ineq: intcontrolpsi*}
    % \end{align}
    To simplify the notation, define the centered function \[\bar{\psi}_b^*(s, e) := {\psi}_b^*\circ T_b(s, e)\one\{\|(s, e)\|\leq \kappa\} - \E[{\psi}_b^*\circ T_b(S, \varepsilon)\one\{\|(S, \varepsilon)\| \leq \kappa\}],\] 
    then it follows that $\|\bar{\psi}_b^*\|_{\infty} \leq 2C(\kappa + 1)^2 \leq C (\log m)^{\frac{1}{2 \wedge \alpha \wedge \beta}}$. In this notation, we have $G = \sup_{b\in\mathcal{B}} |n^{-1}\sum_{i\in[n]} \bar\psi_b^*(S_i,\varepsilon_i)|$. By Markov's inequality, we then have 
    \begin{align}
        \E(G) &= \int_{0}^{+\infty} \Prob(G \geq t) \, dt \leq n^{-1/2} + C \int_{n^{-1/2}}^{+\infty} \frac{ (\log m)^{\frac{2}{2 \wedge \alpha \wedge \beta}}}{nt^2} dt \lesssim \frac{(\log m)^{\frac{2}{2 \wedge \alpha \wedge \beta}}}{\sqrt{n}}. \label{intcontrolexpectationofpsi*} 
    \end{align}
    We now claim that $G$, when viewed as a function of $(s_1,e_1),\ldots,(s_n,e_n)$, satisfies the bounded difference property \citep[see e.g.][(2.32)]{wainwright2019high}. By symmetry, it suffices to consider a perturbation on $(s_1, e_1)$. Define $v=(v_i)_{i = 1}^n, v'=(v_i')_{i = 1}^n$ where each $v_i = (s_i, e_i), v_i' = (s_i', e_i') \in \mathbb{R}^{d+p}$, such that $v_i = v_i'$ for any $i \not= 1$. We have 
    \begin{align} 
        \sup_{b\in\mathcal{B}} \biggl|\frac{1}{n}\sum_{i = 1}^n\bar{\psi}_b^*(v_i)\biggr| -  \sup_{b \in \mathcal{B}}\biggl|\frac{1}{n}\sum_{i = 1}^n\bar{\psi}_b^*(v_i')\biggr| &\leq \sup_{b\in\mathcal{B}} \biggl\{\biggl|\frac{1}{n}\sum_{i = 1}^n\bar{\psi}_b^*(v_i)\biggr| - \biggl|\frac{1}{n}\sum_{i = 1}^n\bar{\psi}_b^*(v_i')\biggr|\biggr\} \notag \\
        &\leq \frac{1}{n} \sup_{b\in\mathcal{B}}\Bigl|\bar{\psi}_b^*(v_1) - \bar{\psi}_b^*(v_1') \Bigr|\leq \frac{ 2C (\log m)^{\frac{1}{2 \wedge \alpha \wedge \beta}}}{n}, \notag 
    \end{align}
    estbalishing, the bounded difference property for $G$. Thus by McDiarmid's inequality \citep[see e.g.][Corollary 2.21]{wainwright2019high}, we obtain that the event 
    \[\Lambda_1:= \Bigl\{G \leq  \mathbb{E}G +  \frac{\sqrt{2}C(\log m)^{\frac{2}{2 \wedge \alpha \wedge \beta}}}{\sqrt{n}} \Bigr\},\]
    occurs with probability at least $1-1/m$.

    Thus, working on the event $ \Omega_1 \cap \Lambda_1$, we deduce from \eqref{intcontrolexpectationofpsi*} that 
    \begin{align}
        \sup_{b \in \mathcal{B}}\Bigl|\int \psi_b^*(v) d(P^{A(b)} - P_n^{A(b)})(v) \Bigr| \leq  \frac{C(\log m)^{\frac{2}{2 \wedge \alpha \wedge \beta}}}{\sqrt{n}}+ \frac{C^\prime}{n} \leq  \frac{C(\log m)^{\frac{2}{2 \wedge \alpha \wedge \beta}}}{\sqrt{n}}\notag,
    \end{align}
    for some constant $C>1$ depends on $d, \alpha, \beta, \gamma_1, \gamma_2, \sigma_1, \sigma_2$, which completes the first claim of the lemma.

    For the second claim, in order to bound $\int \psi_b(u) \, d(P^{U} - P_m^{U})(u)$, we notice that the anti-concentration property of $P^U$ holds due to the Gaussian assumption. Thus \citet[Theorem~11]{manole2021sharp} implies that $\|\nabla \psi_b(u)\| \leq \tilde C(1 + \|u\|)^{\frac{2}{\alpha \wedge \beta}}$ for some $\tilde C>0$ depending on $d, \sigma_1, \sigma_2, \alpha, \beta$, and it follows that $|\psi_b(u)| \leq  \tilde C (1 + \|u\|)^{\frac{2}{\alpha \wedge \beta} + 1}$.

    Define $\Omega_2 := \{\max_{1 \leq i \leq m}\|U_i\| \leq \gamma\}$. From the proof of Lemma~\ref{le: Controlvarphi*} again, we have $\mathbb{P}(\Omega_2)\geq 1-2/n$. Working on $\Omega_2$, we have 
    \begin{align}
         \sup_{b \in \mathcal{B}}\biggl|\int_{\mathbb{R}^d \setminus \mathcal{B}_{0, \gamma}} \psi_b(u) d(P^{U} - P_m^{U})(u) \biggr| &=  \sup_{b \in \mathcal{B}}\biggl|\int_{\mathbb{R}^d \setminus \mathcal{B}_{0, \gamma}} \psi_b(u) dP^{U}(u) \biggr| \notag \\
        &\leq \tilde C\int_{\mathbb{R}^d \setminus \mathcal{B}_{0, \gamma}} (1 + \|u\|)^{\frac{2}{\alpha\wedge\beta} + 1} dP^{U}(u) \leq \frac{\tilde C}{m} \notag,
    \end{align}
    for some constant $\tilde{C}>0$ depending on $d, \alpha, \beta, \sigma_1, \sigma_2$. Now, defining $\tilde G := \sup_{b\in\mathcal{B}}\bigl|\int_{\mathcal{B}_{0,\gamma}}\psi_b d(P^U-P_m^U)\bigr|$, by the same argument as in the proof of the first part of this lemma, there is an event $\Lambda_2$ with probability at least $1-m^{-1}$, such that on $\Omega_2\cap\Lambda_2$, we have 
      \begin{align}
        \sup_{b \in \mathcal{B}}\Bigl|\int \psi_b(u) d(P^{U} - P_m^{U})(u) \Bigr| \leq \tilde G + \frac{\tilde C^\prime}{m} \leq \mathbb{E}\tilde G + \frac{\bar C(\log m)^{\frac{2}{2\wedge \alpha\wedge \beta}}}{\sqrt{m}}+\frac{\tilde C}{m} \leq \frac{\bar C(\log m)^{\frac{2}{2\wedge \alpha\wedge \beta}}}{\sqrt{m}}, \label{ineq: intcontrolpsi}
    \end{align}
for $\bar{C}>0$ depending only on $d, \alpha, \beta, \sigma_1, \sigma_2, \gamma_2$.
\end{proof}

\begin{lemma}\label{le: ULLN}
There exists $C>0$ depending only on $d, \alpha, \beta, \sigma_1, \sigma_2,\gamma_1,\gamma_2$, such that  with probability at least $1 - 5/n$, we have 
\begin{gather}
    \sup_{b \in \mathcal{B}}\biggl|\int \|v\|^2 \, d \bigl(P_n^{A(b)} - P^{A(b)}\bigr)(v)\biggr|  \leq C (\log m)^{\frac{2}{2\wedge\alpha\wedge\beta}} n^{-1/2},\notag \\
    \biggl|\int \|u\|^2 \, d \bigl(P^{U} - P_m^U\bigr)(u) \biggr|\leq C\sqrt{\frac{\log m}{m}}.\notag 
\end{gather}
\end{lemma}
\begin{proof}
    Observe that the only property of $\psi_b^*$ that we used in the first part of the proof of Lemma~\ref{le: BoundforGbnm} is that $\|\nabla \psi_b^*(v)\|\leq C(\|v\|+1)$ for all $b\in\mathcal{B}$ and $v\in\mathbb{R}^d$. The same property is satisfied by the function $v\mapsto \|v\|^2$. Hence, a very similar proof to that of Lemma~\ref{le: BoundforGbnm} will establish the first claim here.

    As for the second inequality, since $U_i \stackrel{i.i.d.}{\sim} \mathcal{N}(0, I_d)$, we have $\sum_{i = 1}^m\|U_i\|^2 \sim \chi^2_{md}$. By \citet[Lemma~1]{laurent2000adaptive} we deduce that 
    \begin{align}
        \Prob\biggl(\Bigl|\frac{1}{m}\sum_{i = 1}^m \|U_i\|^2 - \E\|U\|^2\Bigr| \geq  \sqrt{\frac{2d \log m}{m}} + \frac{2\log m}{m}\biggr) \leq \frac{2}{m}, \notag 
    \end{align}
    which implies the second claim. 
\end{proof}

\begin{proof}{\textbf{of Theorem \ref{thm: FasterRate}}}
    Recalling that in the regime of (\ref{Condition: nmSufficientlyLarge}) event $\Theta$ holds with probability at least $1 - 4(\log n)^{-1}$, and working on $\Theta$ we have $\hat b \in \mathcal{B}$. Thus there exists $M>0$ depending only on $d$, $\alpha$, $\beta$, $\sigma_1$, $\sigma_2$, $\gamma_1$, $\gamma_2$ such that with probability at least $1- 33(\log n)^{-1}$, we have 
    \begin{align}
        \mathcal{L}(\hat b) - \mathcal{L}(b^*) &\leq 2 \sup_{b \in \mathcal{B}}|\mathcal{L}(b) - \mathcal{L}_{n,m}(b)| \notag \\ 
        &\leq \Bigl|\frac{1}{m}\sum_{i = 1}^m \|U_i\|^2 - \E\|U\|^2 \Bigr| + \sup_{b \in \mathcal{B}}\Bigl|\frac{1}{n}\sum_{i = 1}^n \|T_b(S_i, \varepsilon_i) \|^2 - \E \|T_b(S_i, \varepsilon_i)\|^2 \Bigr| \notag \\
        &\hspace{4.5cm}+ \sup_{b \in \mathcal{B}}\Bigl|\mathcal{W}_2^2(P^{A(b)}, P^U) - \mathcal{W}_2^2(P_n^{A(b)}, P_m^U) \Bigr| \notag \\
        & \leq M(\log m)^{\frac{8}{2\wedge \alpha\wedge\beta}}\biggl(\sqrt\frac{p}{n}+\frac{1}{n^{2/d}}\biggr), \label{ineq: SupUpperBound} 
    \end{align}
    where the second inequality uses the definition of $\wip{\cdot}{\cdot}$ and  in the final inequality, we used Lemma \ref{le: ULLN} to control the first two terms and Proposition~\ref{Prop:W2distdiff} for the last term. 

    On the other hand, by the lower bound developed in (\ref{ineq: lb}) and Lemma \ref{le: eleinequality} we have for $r:=\wip{P^\varepsilon}{P^U}$ that 
    \begin{align}
        \mathcal{L}(\hat b)-\mathcal{L}(b^*) \geq \sqrt{r^2 + \|b^* - \hat b\|_{\Sigma}^2 } - r  \geq 
        \frac{1}{2}(1 + r^2)^{-1/2}\|b^* - \hat b\|_{\Sigma}^2.\label{Eq:Thm5Bound2}
    \end{align}
    Combining~\eqref{ineq: SupUpperBound} with~\eqref{Eq:Thm5Bound2}, we obtain that 
    \begin{align}
        \|b^* - \hat b\|_{\Sigma} \leq M  (\log m)^{\frac{4}{2\wedge \alpha\wedge\beta}}  \biggl\{\biggl(\frac{p}{n}\biggr)^{1/4}+\frac{1}{n^{1/d}}\biggr\},
    \end{align}
    with probability at least $1- 33(\log n)^{-1}$. Here we close the proof. 
\end{proof}

\section{Auxillary results}\label{apx: Auxillary}
\begin{lemma}\label{le: eleinequality}
    For any $a \geq 0$, we have inequality
    \[
    \sqrt{a + x^2}\leq \begin{cases}
           \frac{x^2}{2\sqrt{a}} + \sqrt{a} &, \text{if $0\leq x \leq 1$}, \\
            (x-1) + \frac{1}{2\sqrt{a}} + \sqrt{a}&, \text{if $x > 1$}.
        \end{cases} \notag,
    \]
    and 
    \[
    \sqrt{a + x^2}\geq \begin{cases}
           \frac{x^2}{2\sqrt{a + 1}} +\sqrt{a} &, \text{if $0\leq x \leq 1$}, \\
            \frac{x-1}{\sqrt{a + 1}} + \frac{1}{2\sqrt{a + 1}} + \sqrt{a}&, \text{if $x > 1$}.
        \end{cases}\notag
    \]
\end{lemma}
\begin{proof}
    Write
     \begin{align}
         \sqrt{a + x^2} = \int_0^x \frac{t}{\sqrt{a +  t^2}} dt + \sqrt{a} \notag.
    \end{align}
    Thus the first inequality can be obtained by utilizing $t/\sqrt{a + t^2} \leq t/\sqrt{a}$ and $t/\sqrt{a + t^2} \leq 1$ in the case of $0 \leq t \leq 1$ and $t \geq 1$ respectively. The second inequality follows by noting that $t/\sqrt{a + t^2} \geq t/\sqrt{a +1}$ when $ 0\leq t \leq 1 $ and $t/\sqrt{a + t^2} \geq 1 / \sqrt{a + 1}$ when $t \geq 1$.
\end{proof}
\begin{lemma}
\label{le: wipLowerBoundSharp}
    There exist independent random vectors $Z$ and $\varepsilon$ such that $P^Z, P^\varepsilon\in\mathcal{P}_2(\mathbb{R}^d)\cap\mathcal{P}_{\mathrm{ac}}(\mathbb{R}^d)$ such that $\wip{Z+\varepsilon}{U}^2 = \wip{Z}{U}^2+\wip{\varepsilon}{U}^2$.
\end{lemma}
\begin{proof}
    Consider independent random vectors $Z\sim \mathcal{N}(0,\Sigma)$ and $\varepsilon\sim \mathcal{N}(0,\Gamma)$. By the same argument as in~\eqref{Eq:EllipticalWIP}, we have 
    \begin{align*}
        \wip{Z+\varepsilon}{U} &= \Tr((\Sigma+\Gamma)^{1/2})\\
        \wip{Z}{U} &= \Tr(\Sigma^{1/2})\\
        \wip{\varepsilon}{U} &= \Tr(\Gamma^{1/2}).
    \end{align*}
    Hence, the desired result hold if we take $\Sigma = \sigma^2 I_d$ and $\Gamma = \gamma^2 I_d$.
\end{proof}
\begin{proposition}\label{prop: swequivalent}
\label{Prop:SW}
    Let $X$ be a random vecotor. Then the following properties are equivalent:
    \begin{enumerate}[label=(\roman*)]
        \item There exists $\sigma > 0$ such that $\Prob(\|X\|\geq x) \leq 2 e^{-\frac{1}{2}(x/\sigma )^{\beta}}$ for all $x \geq 0$. \label{prop: sw1}
        \item There exists $K_{\sigma}>0$  such that $\{\mathbb{E}\|X\|^k\}^{1/k} \leq K_{\sigma} k^{1/\beta}$. \label{prop: sw2}
        \item There exists $K_\sigma^\prime>0$ such that $\E \exp{\bigl((\lambda \|X\|)^\beta\bigr)} \leq \exp{\bigl((\lambda K_\sigma^\prime)^{\beta}\bigr)}$ for all $|\lambda| \leq 1/K_\sigma^\prime. $  
        \item $X$ follows the $(\sigma, \beta)$-sub-weibull distribution. \label{prop: sw4} 
    \end{enumerate}
\end{proposition}
The proof follows by \citet[Theorem 2.1]{vladimirova2020sub}. 

\begin{proposition}\label{prop: swsum}
        For $p_1, p_2 \in \mathbb{N}$, let $X \in \mathbb{R}^{p_1}, Y \in \mathbb{R}^{p_2}$ be two independent sub-Weibull random vectors with parameter $(\sigma_1, \alpha)$ and $(\sigma_2, \beta)$ respectively. Then the following statements holds:
        \begin{enumerate}[label = (\roman*)]
            \item For matrices $A \in \mathbb{R}^{d \times {p_1}}$ and $B \in \mathbb{R}^{d \times {p_2}}$, there exists $\sigma>0$ depending only on $\sigma_1$, $\sigma_2$, $\|A\|_{\ope}$, $\|B\|_{\ope}$ such that $AX + BY \sim \sw{\sigma}{\alpha \wedge \beta}$. 
            \item There exists  $\sigma>0$ depending only on $\sigma_1, \sigma_2$ such that the concatenation of two random vectors $Z:= (X, Y) \in \mathbb{R}^{p_1 + p_2}$ is a sub-Weibull random vector with parameter $( \sigma, \alpha \wedge \beta)$.
        \end{enumerate}
\end{proposition}
\begin{proof}
    (i) Suppose $K_{\sigma_1}$ and $K_{\sigma_2}$ are the induced constants of $X$ and $Y$ by the part \ref{prop: sw2} of Proposition \ref{prop: swequivalent}. Then it follows that 
    \begin{align*}
        \bigl(\E\|AX + BY\|^k\bigr)^{1/k} &\leq (\E\|AX\|^k)^{1/k} + (\E\|BY\|^k)^{1/k} \\
        &\leq \|A\|_{\ope}(\E\|X\|^k)^{1/k} + \|B\|_{\ope}(\E\|Y\|^k)^{1/k} \notag \\
        & \leq \|A\|_{\ope}\vee \|B\|_{\ope} \cdot \bigl(K_{\sigma_1}k^{1/\alpha} + K_{\sigma_2}k^{1/\beta}\bigr) \notag \\
        & \leq 2( \|A\|_{\ope}\vee \|B\|_{\ope}) \cdot (K_{\sigma_1} \vee K_{\sigma_2}) k^{1/(\alpha \wedge \beta)}. \notag 
    \end{align*}
    This proves that $AX + BY$ satisfies part \ref{prop: sw2} in the Proposition \ref{prop: swequivalent} thus the conclusion follows by the equivalence of part $\ref{prop: sw2}$ and $\ref{prop: sw4}$. 

    (ii) For any integer $k \geq 1$, we have
    \begin{align*}
        \bigl(\E\|(X, Y)\|^k\bigr)^{1/k} &\leq \bigl(\E(\|X\| + \|Y\|)^k\bigr)^{1/k} \\
        &\leq \bigl(\E\|X\|^k \bigr)^{1/k}+ \bigl(\E\|Y\|^k\bigr)^{1/k}\leq (K_{\sigma_1} \vee K_{\sigma_2})k^{1/(\alpha \wedge \beta)}\notag,
    \end{align*}
    where the sub-Weibull assumption on $X$ and $Y$ have been exploited. The conclusion follows by employing Proposition \ref{prop: swequivalent}.
\end{proof}

\begin{lemma}
\label{Lem:TruncatedMoment}
    If $X$ is a $(\sigma,\beta)$-sub-Weibull random vector as defined in~\eqref{def: sub-weibull}, then for any $s>0$, there exists $C> 0$, depending on $s, \sigma, \beta$, such that $\mathbb{E}\bigl(\|X\|^s\one\{\|X\| \geq t\}\bigr) \leq Ce^{-\frac{1}{4}(t/\sigma)^\beta}$.
\end{lemma}
\begin{proof}
    We have 
    \begin{align*}
        \mathbb{E}\bigl(\|X\|^s\one\{\|X\| \geq t\}\bigr) & = \mathbb{E} \Bigl[\|X\|^s \one\Bigl\{e^{\frac{1}{4}(\|X\|/\sigma)^\beta} \geq e^{\frac{1}{4}(t/\sigma)^\beta}\Bigr\}\Bigr] \\
        &\leq \E\bigl\{\|X\|^s e^{\frac{1}{4}(\|X\|/\sigma)^\beta } e^{-\frac{1}{4} (t/\sigma)^{\beta}} \bigr\} \\
        &\leq e^{-\frac{1}{4} (t/\sigma)^{\beta}} \bigl\{\E\|X\|^{2s}\bigr\}^{1/2} \Bigl\{\E e^{\frac{1}{2}(\|X\|/\sigma)^\beta } \Bigr\}^{1/2} \\
        &\leq 2^{1/2}e^{-\frac{1}{4}(t/\sigma)^{\beta}} \bigl\{\E\|X\|^{2s}\bigr\}^{1/2},
    \end{align*}
    where we used the definition of $X$ being $(\sigma,\beta)$-sub-Weibull in final step. The desired bound follows since  by Proposition~\ref{Prop:SW}, we have $\mathbb{E}\|X\|^{2s} \leq C$ for some constant $C$ that depends on $s, \sigma, \beta$. 
\end{proof}

\begin{lemma}\label{le: AnticoncentrationforSumofTwoRVs}
    Suppose $X, Y$ are independent $d$-dimensional random vectors with finite second moment. If $X$ follows an absolutely continuous distribution with density function $f_X$ which admits the following anti-concentration inequality for some constant $\gamma_1, \gamma_2 >0$:
    \[
    f_X(x) \geq \gamma_1 \exp\bigl(-\gamma_2 \|x\|^2\bigr), \quad \forall \ \|x\| \geq \E\|Y\|^2.
    \]
    Then there exists a constant $K_1$ depends on $\gamma_1$ and $\gamma_2$ such that the density function of $V := X+ Y$, write as $f_V$, satisfying 
    \begin{align}
        f_V(v) \geq K_1 \exp{(-2\gamma_2\|v\|^2)}, \quad \forall \ \|v\| \geq 2 \E\|Y\|^2. \notag 
    \end{align}
\end{lemma}
\begin{proof}
    Write $M_2:= \E \|Y\|^2 < +\infty$. For all $\|v\| \geq 2 M_2$, we have
    \begin{align}
        f_V(v) &= \int f_X(v - y)f_Y(y) d y \geq \int_{\|y\|  \leq M_2} \gamma_1 \exp{\bigl(-\gamma_2\|v-y\|^2\bigr)} f_Y(y) d y \notag \\
        &\geq  \int_{\|y\| \leq M_2} \gamma_1^\prime \exp{\bigl(- 2\gamma_2\|v\|^2 \bigr)} f_Y(y) dy  \geq  \gamma_1^\prime\bigl(1- \frac{1}{M_2}\bigr) \exp{\bigl(- 2\gamma_2\|v\|^2 \bigr)} \notag, 
    \end{align}
    where $\gamma_1^\prime = \gamma_1\exp(-2\gamma_2M_2^2)$ and the last inequality is followed by the Markov inequality. Thus the result holds by letting $K_1 =  \gamma_1^\prime\bigl(1- \frac{1}{M_2}\bigr)$. 
\end{proof}

\begin{lemma}\label{le: ConjugateUpperBound}
    Let $\mathcal{X}, \mathcal{Y} \subseteq \mathbb{R}^d$ are Borel sets such that $L_2$ is bounded on $\mathcal{X} \times \mathcal{Y}$, i.e. $\|L_2\|_{\infty}:=\sup_{(x, y )\in \mathcal{X}\times \mathcal{Y}}L_2(x, y) < +\infty$. Then for any $\mu \in \mathcal{P}_2(\mathcal{X})$ and $\nu \in \mathcal{P}_2(\mathcal{Y})$ we have 
    \begin{align}
        \inf&\big\{J_{\mu, \nu}(\varphi, \psi): (\varphi, \psi) \in \tilde{\Phi}\}\notag \\
        &= \inf\big\{J_{\mu, \nu}(\varphi, \psi): (\varphi, \psi) \in \tilde{\Phi},  - \|L_2\|_{\infty} \leq\varphi - \|\cdot\|^2/2 \leq 0, \ 0 \leq \psi - \|\cdot\|^2/2 \leq  \|L_2\|_{\infty} \}, \notag 
    \end{align}
    where $\tilde{\Phi}:= \{(\varphi, \psi)\in L^1(\mathcal{X}) \times L^1(\mathcal{Y}): \varphi(x) + \psi(y) \geq x^T y, \ \forall (x, y)\in \mathcal{X} \times \mathcal{Y}\}$.
\end{lemma}
\begin{proof}
    Note that by the argument same as \eqref{opt: L2EquivalentForm} we have 
    \begin{align}
        \inf&\big\{J_{\mu, \nu}(\varphi, \psi): (\varphi, \psi) \in \tilde{\Phi}\} \notag \\
        &= \int_{\mathcal{X}} \frac{\|x\|^2}{2} d\mu(x) + \int_{\mathcal{Y}} \frac{\|y\|^2}{2} d\nu(y) - \sup\bigl\{J_{\mu, \nu}(\varphi, \psi): (\varphi, \psi) \in \Phi_2 \bigr\}, \label{eq: QuodraticForm} 
    \end{align}
    where $\Phi_2 := \{(\varphi, \psi)\in L^1(\mathcal{X}) \times L^1(\mathcal{Y}): \varphi(x) + \psi(y) \leq L_2(x, y), \ \forall (x, y)\in \mathbb{R}^d \times \mathbb{R}^d  \}$. Note by \citet[Remark 1.13]{villani2021topics}, we may restrict the supremum in the right-hand side of \eqref{eq: QuodraticForm} over some bounded functions:
    \begin{align}
         \sup&\bigl\{J_{\mu, \nu}(\varphi, \psi): (\varphi, \psi) \in \Phi_2 \bigr\}\notag \\
         &= \sup\bigl\{J_{\mu, \nu}(\varphi, \psi): (\varphi, \psi) \in \Phi_2, 0\leq \varphi \leq \|L_2\|_{\infty}, \, -\|L_2\|_{\infty} \leq \psi \leq 0 \bigr\} \label{opt: ReKantoDual}.
    \end{align}
    By \citet[Theorem 5.10]{villani2009optimal} we may further impose that $\varphi$ be c-concave and $\psi = \varphi^c$. Suppose $(\varphi_0, \varphi_0^c)$ be a solution to the right-hand side of \eqref{opt: ReKantoDual}. Define $\tilde{\varphi}:= \|\cdot\|^2/2 - \varphi_0$, $\tilde{\psi}:= \|\cdot\|^2/2 - \varphi_0^c$. Then by \eqref{eq: QuodraticForm} we have 
    \begin{align}
         \inf\big\{J_{\mu, \nu}(\varphi, \psi): (\varphi, \psi) \in \tilde{\Phi}\} = \int_{\mathcal{X}} \tilde{\varphi}(x) d \mu(x) + \int_{\mathcal{Y}} \tilde{\psi}(y) d\nu(y). \label{eq: Optimality} 
    \end{align}
    Moreover, note
    \begin{align}
        \tilde{\varphi}(x) = \|x\|^2/2 - \varphi_0(x) = \|x\|^2/2 - \inf_{y \in \mathcal{Y}}\{c(x, y ) - \varphi_0^c(y)\}= \sup_{y \in \mathcal{Y}}\{x^T y - (\|y\|^2/2 - \varphi_0^c(y))\} \notag, 
    \end{align}
    which implies that $(\tilde{\varphi}, \tilde{\psi}) \in \tilde{\Phi}$. Combine this with \eqref{eq: Optimality}, we proved that $(\tilde{\varphi}, \tilde{\psi})$ is an optimal solution to the left-hand side of \eqref{eq: QuodraticForm}. Finally, by the boundedness of $\varphi_0$ and $\varphi_0^c$, we have 
    \begin{align}
        0\leq\|x\|^2/2 - \tilde{\varphi}(x) \leq \|L_2\|_{\infty} \ \text{and} \ 
        -\|L_2\|_{\infty}\leq \|y\|^2/2 - \tilde{\psi}(y)\leq 0 \notag,
    \end{align}
     as desired. 
     \end{proof}

    \begin{theorem}\citep[Theorem 1]{fournier2015rate}\label{thm: WassRate}
    Let $X \sim P^X$ be a probability measure on $\mathbb{R}^d$ such that $M_\ell:= \E\|X\|^{\ell} < +\infty$ with $\ell \in (2, +\infty)$. If $P_n^X$ is the corresponding empirical distribution, then there exists a constant $C>0$ depending only on $d$ and $\ell$ such that for all $n \geq 1$,
    \begin{align}
        \E\left[\mathcal{W}_2^2(P^X, P_n^X)\right]\leq  C M_{\ell}^{2/\ell}\tau_n(d, \ell),
    \end{align}
    where 
    \begin{align}
    \tau_n(d, \ell):= \begin{cases}
    n^{-\frac{1}{2}} & \text { if $d < 4$} \\
    n^{-\frac{1}{2}} \log(1 + n) &\text{if $d = 4$} \\ 
    n^{-\frac{2}{d}} & \text{if $d>4$} 
    \end{cases} + \begin{cases}
    n^{-\frac{1}{d}} & \text{if $\ell >4$} \\ 
    n^{-\frac{1}{2}} \log (1 + n) & \text{if $\ell = 4$} \\ 
    n^{\frac{2 - \ell}{\ell}} &\text{if $2< \ell < 4$}.
    \end{cases} \notag 
    \end{align} 
%     \begin{align}
%         \tau_n(d, \ell):=  \begin{cases}n^{-1 / 2} + n^{\frac{2-\ell}{\ell}} & \text { if $d < 4$ and $\ell \not= 4$ } \\
% n^{-1 / 2} \log n + n^{\frac{2-\ell}{\ell}} & \text {if $d = 4$ and $\ell \not= 4$ } \\
% n^{-2 / d} + n^{\frac{2-\ell}{\ell}} & \text { if $d > 4$ and $\ell \not= \frac{d}{d-2}$}. \end{cases} \notag   
%     \end{align}
\end{theorem}

\section{Spatial reference distribution}\label{apx: ref_distribution}
% \TY{This is moved from the old Sec 2 to here temporarily}
% As noted by \cite{deb2021multivariate}, the conventional rank map can be regarded as a solution to the assignment problem of mapping a set of random samples $X_1, \dots, X_n$ that are drawn from a distribution $P^X \in \mathcal{P}_{\ell}(\mathbb{R})$ onto the discrete set $\{1/n, 2/n, \ldots, 1\}$ and the corresponding inverse map is defined as the quantile map. Thus, the traditional distribution function $F_X(x):= P^X(-\infty, x]$ can be seen as an optimal transport map from the interest distribution $P^X$ to the uniform distribution on $[0, 1]$, and its inverse map is the quantile function. However, we can choose other distributions such as the standard normal distribution and spherical uniform distribution as the reference distribution as well. In this case, the definition of the distribution function and the quantile function need to be modified accordingly. We refer to Appendix \ref{apx: ref_distribution} for more illustrations.
% \TY{[end of moved text]}

In this section, we derive the MCQR loss function under reference distribution $U[-1, 1]$, which may provide an intuitive example for the verification of Proposition \ref{prop: unique}. In one dimension, the traditional rank and quantile can be understood as a pair of optimal transport maps between the distribution of interest $X\sim P$ and the uniform distribution $U\sim U[0, 1]$. When $P$ does not assign mass to sets with Hausdorff dimension $0$, the corresponding distribution function $F$ and its inverse map $Q:=F^{-1}$ serve as the corresponding optimal transport map. This concept can be generalized to other reference distributions, for instance, $U[-1, 1]$. In this case, the spatial distribution function $F_{\mathrm{sp}}(\cdot):= 2F(\cdot) - 1$ takes on the role of $F$ in the previous case. Moreover, the corresponding check function needs to be modified as 
\begin{align}
    \rho_\tau^{\mathrm{sp}}(X - \theta):=   (1 + \tau)( X- \theta  ) - 2(X - \theta)\one\{X - \theta <0\} , \quad \forall \tau \in [-1, 1]  \notag. 
\end{align}
Suppose $V \sim U[-1, 1]$, then the composite quantile regression optimization becomes  
\begin{align}
    \E\int_{-1}^1\rho_\tau^{\mathrm{sp}}\bigl(Y - \beta^\top X &- q(\tau)\bigr) \cdot \frac{1}{2} \,d\tau  = \E\int_{-1}^1 \bigl(Y - \beta^\top X -q(\tau)\bigr)^{-}\, d\tau + \int_{-1}^1 \int_\tau^{-1} \frac{1}{2}q(\tau)\, d t d\tau  \notag \\
    &= \E \max_{t\in [-1, 1]} \int_t^1 -\bigl(Y-b X - q(\tau)\bigr)\,d\tau + \int_{-1}^1 \int_t^1 \frac{1}{2}q(\tau) \, d\tau dt  \notag \\
    &= \E \max_{t\in [-1, 1]}\bigl(-(1-t)(Y-b X) + \phi(t) \bigr) + \E\phi(V) \notag \\
    & = \E \max_{t\in [-1, 1]}\bigl(t(Y- b X) + \phi(t) \bigr) + \E\phi(V),\notag 
\end{align}
where $\phi(t) = \int_t^1 q(\tau) d\tau.$ Thus, applying the same argument as Lemma \ref{le: 1dmcqr} we can see that the composition quantile regression estimator of $b^*$ is once again
\begin{align}
    b^* = \argmin \wip{P^{Y-b X}}{P^V}. \notag 
\end{align}
This gives some intuition on Proposition \ref{prop: unique}. However, if choose the standard normal distribution as the reference distribution, we may not be able to find a straightforward optimal transport map as $F$ or $F_{\mathrm{sp}}$, but Proposition \ref{prop: unique} demonstrates the validity of this extension. 

\section{Spatial quantile}\label{Apx: SimulationDetails}
The concept of the spatial (or geometric) quantile was initially introduced by \cite{chaudhuri1996geometric}. Uniquely characterizing the underlying probability distribution as a special case of M-quantile, as demonstrated in \citep[Theorem 2.5]{koltchinskii1997m}, this quantile permits a seamless extension to the regression framework \citep{chakraborty2003multivariate} and functional quantile regression \citep{chakraborty2014spatial, chowdhury2019nonparametric}. A more recent development involves an extension to the hypersphere, as explored by \citet{konen2023spatial}. 

The definition of Spatial quantile starts from rewriting the check function $\rho_\tau(\cdot)$ as 
\begin{align}
    \rho_\tau(z) &= \frac{1}{2}\bigl(|z| + (2\tau - 1)z\bigr)  = \frac{1}{2}(|z| + vz) \notag, \ \text{for any $z \in \mathbb{R}$},
\end{align}
with $v = 2 \tau -1$. Thus a natural extension of the check function to the multi-dimensional case is by substituting the absolute value function by the $L_1$-loss function: 
\begin{align}
    \Phi_v(z) := \frac{1}{2}(\|z\| + v^\top z) \notag,
\end{align}
where $v = \tau u$, and $u \in \mathcal{S}^{d-1}$. This extension of the check function immediately leads to the following definition of spatial quantile:
\begin{definition}
    Suppose $Y\sim \Prob^Y$ is a random variable on $\mathbb{R}^d$ $(d\geq 1)$. Then for any $\tau \in [0, 1]$
    and $u \in \mathcal{S}^{d-1}$, the $\tau u$-spatial quantile of $P^Y$ is defined as
    \begin{align}
        \mathrm{Q}_{\tau u}= \argmin_{y \in \mathbb{R}^d}\E\Phi_{\tau u}(Y - y). \label{df: spatial} 
    \end{align}
\end{definition}
Note the solution of~\eqref{df: spatial} are such that 
\begin{align}
    \E\bigl(\frac{Y-Q_{\tau u}}{\|Y-Q_{\tau u}\|}\bigr) = - \tau u.  \notag 
\end{align}
Intuitively speaking, this indicates that $Q_{\tau u}$ defines a point in $\mathbb{R}^d$ such that the average unit vector from it to other random samples should be $\tau u$.

The generalization to quantile regression setting is simply by applying the spatial quantile definition to $Y-b^*X - a$, where $X \in \mathbb{R}^p$ is the covariate vector, $b^* \in \mathbb{R}^{d \times p}$ is the regression coefficient and $a$ is the intercept term. Specifically, for fixed $\tau \in [0, 1]$ and $u \in \mathcal{S}^{d-1}$, 
\begin{align}
    (a_{\tau u}, b_{\tau, u}) = \argmin_{b \in \mathbb{R}^{d \times p}, \ a \in \mathbb{R}^d}\E\Phi_{\tau u}(Y - bX -a). \notag 
\end{align}
Therefore, given observations $(Y_1, X_1), \ldots, (Y_n, X_n)$ satisfying equations
\[Y_i = b^*X_i + \varepsilon_i, \ i = 1, \ldots, n,\]
for some random residue terms $\varepsilon_i$'s that are independent with with $X_i$'s, the spatial quantile estimator of $b^*$ can be obtained by
\begin{align}
    (\hat b^{(\mathrm{sp})}, \hat a_{\tau u}^{(\mathrm{sp})})\argmin\limits_{b \in \mathbb{R}^{d \times p}, \ a \in \mathbb{R}^p}\frac{1}{n}\sum_{i = 1}^n \Phi_{\tau u}(Y_i -bX_i -a) \notag.
\end{align}
Therefore, the optimizer can be obtained by applying classical convex optimization algorithms. 

% \bibliographystyle{apalike}
% \bibliography{ref.bib}
\printbibliography
\end{refsection}
\end{document}